\documentclass{amsart}

\usepackage{amsmath}
\usepackage{amsthm,amssymb,color,comment}
\usepackage{bbm}
\usepackage{xcolor}
\usepackage{hyperref}
\usepackage[TS1,T1]{fontenc}
\usepackage{inputenc}
\usepackage{dsfont}
\usepackage{tikz}
\usepackage{enumitem}
\usepackage{setspace}
\usepackage{soul}
\usepackage[sort]{cite}% to arrange citations in increasing order

\numberwithin{equation}{section}

\newtheorem{thm}{Theorem}[section]
\newtheorem*{thm*}{Theorem}

\newtheorem{lem}[thm]{Lemma}

\theoremstyle{definition}

\DeclareMathOperator{\DIV}{div}

\DeclareMathOperator{\tr}{tr}
\DeclareMathOperator{\CURL}{curl}

\newcommand{\R}{\mathbb{R}}
\newcommand{\N}{\mathbb{N}}
\newcommand{\T}{\mathbb{T}}
\newcommand{\B}{\mathbb{B}}
\newcommand{\F}{\mathbb{F}}
\newcommand{\A}{\mathbb{A}}
\newcommand{\D}{\mathbb{D}}
\newcommand{\bG}{\mathbb{G}}
\newcommand{\OO}{\mathbb{O}}
\newcommand{\p}{\partial}

\newcommand{\diff}{\mathop{}\!\mathrm{d}}

\newcommand{\symmb}{\mathbb{R}^{2\times 2}_{\text{sym}}}

\newcommand{\vect}[1]{\boldsymbol{#1}}
\def\bq{\vect{q}}

\def\bv{\vect{v}}
\def\bu{\vect{u}}
\def\b0{\vect{0}}
\def\bphi{\vect{\varphi}}
\def\bomega{\vect{\omega}}

\makeatletter
\newcommand{\doublewidetilde}[1]{{%
  \mathpalette\double@widetilde{#1}%
}}
\newcommand{\double@widetilde}[2]{%
  \sbox\z@{$\m@th#1\widetilde{#2}$}%
  \ht\z@=.9\ht\z@
  \widetilde{\box\z@}%
}
\makeatother
\author{Miroslav Bul\'{\i}\v{c}ek}
\address{Mathematical Institute, Faculty of Mathematics and Physics, Charles University, Sokolovsk\'{a} 83, 186~75, Prague, Czech Republic}
\email{mbul8060@karlin.mff.cuni.cz}
\thanks{Miroslav Bul{\'\i}{\v{c}}ek was supported by the project No. 20-11027X financed by GA\v{C}R. He is a  member of the Ne\v{c}as Center for Mathematical Modeling. ORCID: 0000-0003-2380-3458; corresponding author}

\author{Jakub Woźnicki}
\address{Faculty of Mathematics, Informatics and Mechanics, University of Warsaw, Stefana Banacha 2, 02-097 Warsaw, Poland; Institute of Mathematics of Polish Academy of Sciences, Jana i J\k edrzeja \'Sniadeckich 8, 00-656 Warsaw, Poland}
\email{jw.woznicki@student.uw.edu.pl}
\thanks{Jakub Woźnicki was supported by National Science Center, Poland through the project no. 2023/49/N/ST1/02737. ORCID: 0000-0002-7720-5261.
}

\begin{document}

\title[Thermo-visco-elastic fluids]{On various models describing behaviour of thermoviscoelastic rate-type fluids}

\begin{abstract}
Viscoelastic rate-type fluid models are essential for describing the behavior of a wide range of complex materials, with applications in fields such as engineering, biomaterials, and medicine. These models are particularly useful for understanding the rheological properties of materials that exhibit both elastic and viscous behavior under deformation. However, many real-world applications involve significant thermal effects, where heat conduction and the temperature dependence of material properties must also be considered. In this paper, we introduce a thermodynamically consistent model for heat-conducting viscoelastic rate-type fluids and establish the existence of a global weak solution in a two-dimensional setting. The result holds under the condition that the initial energy and entropy are controlled in appropriate natural norms.
\end{abstract}

\keywords{non-Newtonian fluids, Giesekus model, thermo-visco-elasticity, global-in-time and large-data existence theory}
\subjclass{35A01, 35Q35, 76A10, 76D03}

\maketitle

\section{Introduction}

Our main goal in this paper is to study the equations describing the motion of incompressible viscoelastic and heat conducting fluids and in particular to establish the long-time and the large-data theory. We focus here only on the planar case, i.e., in what follows, the fluid occupies the Lipschitz domain $\Omega\subset \mathbb{R}^2$ and $T>0$ always denotes the length of time interval. The motion of incompressible fluid is in general described by the system of partial differential equations of the form
\begin{equation}\label{MB:1}
\begin{split}
    \partial_t \bv + \DIV_x (\bv\otimes \bv) -\DIV_x\T = \b0, \\
    \DIV_x u = 0,
\end{split}
\end{equation}
that is supposed to be satisfied in $\Omega_T:=(0,T)\times \Omega$. Here, $\bv:\Omega_T\to \mathbb{R}^2$ is the unknown velocity field and $\T:\Omega_T\to \symmb$ denotes the Cauchy stress tensor, and the density of the fluid is set to be equal to one for simplicity. The above system must be completed with the initial and boundary conditions, but most importantly, the constitutive relationship for $\T$ must be prescribed. We shall consider the classical form
\begin{equation}\label{Cauchy}
\T:=  -p\mathbb{I} + 2\nu \D \bv + 2g(\B - \mathbb{I}),
\end{equation}
where $p$ is the unknown pressure,  $\D\bv :=\frac12((\nabla \bv)+ (\nabla \bv)^T)$ is the symmetric velocity gradient and $\B:\Omega_T\to \symmb$ is the so-called extra stress. The material parameter $\nu$ - the viscosity, and $g$ - the shear modulus are in what follows the functions of the temperature. We will specify the particular assumptions below.

The first two terms on the right-hand side of~\eqref{Cauchy} describe the classical Newtonian fluid, but the last term, extra stress, represents the elastic effects. Hence, we need to add an equation for $\B$ and we shall assume the so-called rate-type models. This means that we complete the system~\eqref{MB:1}  by the following system of  equations
\begin{equation}\label{Gies}
\overset{\nabla}{\B} + \delta (\B^2 - \B)+ \gamma (\B - \mathbb{I}) = \OO,
\end{equation}
where $\delta$ and $\gamma$ are material parameters related to the relaxation times in the material and possibly depending on the temperature. The symbol $\overset{\nabla}{\B}$ denotes the Oldroyd upper convective derivative defined as
\begin{align}\label{Oldroyd}
    \overset{\nabla}{\B} := \partial_t \B + (\bv\cdot \nabla_x)\,\B - \nabla_x \bv\,\B - \B\,(\nabla_x \bv)^T.
\end{align}
Note that this is the often-used objective derivative applied to $\B$ in the context of rate-type models.

The system \eqref{MB:1}--\eqref{Gies} with $\delta:=0$ and $\gamma>0$ is generally referred to as the Oldroyd-B model, see \cite{oldroyd1950onthe}, and belongs to one of the classical models of viscoelasticity. However, any mathematical theory which would lead to the global well-posed theory is missing. The other model with $\gamma:=0$ and $\delta>0$ proposed by Giesekus, see \cite{giesekus1982asimple}, is also widely used and, contrary to the Oldroyd-B model, it seems to have much better mathematical properties. Note also that beside the upper convective derivative~\eqref{Oldroyd}, one can also consider a much more general class of objectives derivatives, see~\cite{johnson1977amodel}, but it usually does not affect the qualitative analysis. Similarly, one may also consider models of second order, see \cite{burgers1939mechanical}, but again, once we are able to deal with the classical Giesekus model, it seems that further generalisations are rather straightforward. Therefore, we do not consider them here and rather focus on canonical problems. Hence, the  system we are interested in reads 
\begin{equation}\label{eq:main_syst}
\left\{
\begin{aligned}
    &\partial_t \bv + \DIV_x (\bv\otimes \bv) -\DIV_x\T = 0, \\
    &\DIV_x \bv = 0,\\
    &\overset{\nabla}{\B} + \delta(\theta)(\B^2 - \B) = \OO,\\
    &-p\mathbb{I} + 2\nu(\theta)\D\bv + 2g(\theta)(\B - \mathbb{I}) = \T,
\end{aligned}
\right.
\end{equation}
and we assume that all material coefficients may depend on the temperature $\theta:\Omega_T \to \mathbb{R}$.  Once we have added thermal effects into the system (the dependence of coefficients on the temperature), we must also impose the validity of the first law of thermodynamics, i.e. the conservation of energy. One of the possible equivalent form (valid for sufficiently smooth solution) is the equation for the internal energy $e:\Omega_T \to \mathbb{R}$, where we assume that the heat flux is given by the Fourier law
\begin{align}\label{eq:main_syst_internal_energy}
    \p_t e + \DIV_x (\bv e) -\DIV_x(\kappa(\theta)\nabla_x\theta) = \T : \D\bv,
\end{align}
where $\kappa$ denotes the heat conductivity. Note here that the symbol $\mathbb{A} : \mathbb{B}$ denotes the matrix scalar product $\mathbb{A} : \mathbb{B} = \sum_{i,j=1}^2 \mathbb{A}_{ij} \mathbb{B}_{ij}$.

It remains to specify the storage and dissipative mechanisms of the fluid, or in the purely mathematical sense, it remains to specify the relation between the internal energy $e$, the temperature $\theta$ and the tensor $\mathbb{B}$. Further, this relation must be such that one can identify an entropy of the system, which fulfils the second law of thermodynamics. Here, we closely follow the approach developed in \cite{Raj00} and in \cite{Hron17}, where the Helmholtz free energy plays the crucial role. We may also refer to~\cite{DEO99} for the first systematic approach or to \cite{HuLe}, where the entropy estimates were derived for the first time.

Hence, we assume here that the Helmholtz free energy is given by
\begin{align}\label{helmholtz_free_energy}
\psi (\theta, \B) = -c_v\theta(\ln \theta - 1) + g(\theta)f(\B),
\end{align}
where the function $f:\symmb \to \mathbb{R}$ is given by the formula
\begin{align}\label{fdf}
    f(\B) = \tr \B - 2 - \ln\det\B.
\end{align}
The first term in~\eqref{helmholtz_free_energy} corresponds to the classical law when $e=c_v \theta$ with $c_v$ being the heat capacity, and the second term connects the elastic and the temperature effects. Note that if $g$ is a constant function, then this choice is very much classical in the theory of viscoelastic rate-type fluids. However, here, we deal with a more general setting similar to that of~\cite{bathory2021largedata}. Having introduced the free energy, we define the entropy $\eta$ as
$$
\eta=\eta (\theta, \B) := -\p_\theta \psi (\theta, \B) = c_v\ln\theta - g'(\theta)f(\B),
$$
and the internal energy is then related by the formula
\begin{align}\label{equation_for_internal_energy_for_g_theta}
e=e(\theta,\B):= \psi(\theta, \B) + \theta\eta(\theta,\B) = c_v\theta + (g(\theta) - \theta\,g'(\theta))f(\B).
\end{align}
Using these definitions, we see that
$$
\begin{aligned}
\partial_t e &= \partial_t \psi + \partial_t \theta \, \eta + \theta \, \partial_t \eta = \partial_t \theta(\partial_{\theta} \psi  +  \eta) + \partial_{\mathbb{B}} \psi \partial_t \mathbb{B} +\theta \partial_t \eta=\partial_{\mathbb{B}} \psi \partial_t \mathbb{B} +\theta \partial_t \eta\\
&\implies \partial_t \eta= \frac{\partial_t e-\partial_{\mathbb{B}} \psi \partial_t \mathbb{B}}{\theta} = \frac{\partial_t e - \partial_t \mathbb{B} : g(\theta)(\mathbb{I}-\mathbb{B}^{-1}) }{\theta},
\end{aligned}
$$
where we used the following identity 
$$
\partial_{\mathbb{B}} \psi (\theta,\B) = g(\theta) \partial_{\B} f(\B) = g(\theta)(\mathbb{I}-\mathbb{B}^{-1}).
$$
Thus, dividing \eqref{eq:main_syst_internal_energy} by $\theta$, taking the scalar product of the third equation in~\eqref{eq:main_syst} with $\frac{g(\theta)(\mathbb{I}-\mathbb{B}^{-1}) }{\theta}$ and subtracting the result, we obtain the following identity for the entropy
\begin{align}\label{eq:main_syst_entropy}
    \partial_t\eta + \DIV_x (\eta\, \bv) - \DIV_x(\kappa(\theta)\nabla_x(\ln\theta)) = \frac{\kappa(\theta)|\nabla_x\theta|^2}{\theta^2} + \frac{2\nu(\theta)|\D \bv|^2}{\theta} + \frac{\delta(\theta)|\B - \mathbb{I}|^2}{\theta}\ge 0.%,\quad \kappa(\theta) >0.
\end{align}
Here, it is essential that the right-hand side is nonnegative and the term $\kappa(\theta)\nabla_x(\ln\theta)$ represents the entropy flux. It also directly follows from the above procedure that the equation for the internal energy~\eqref{eq:main_syst_internal_energy} and the equation for the entropy~\eqref{eq:main_syst_entropy} are interchangeable in \eqref{eq:main_syst} for sufficiently regular solutions. This ``equivalence'' of~\eqref{eq:main_syst_internal_energy} and~\eqref{eq:main_syst_entropy} is also later used not only for the analysis of the problem but also for the definition of a notion of solution. 

Finally, we always assume that $\mathbb{B}$ is of the form
\begin{align*}
    \B = \F\,\F^T, \quad \F\in \R^{2\times 2},\quad \det \F > 0.
\end{align*}
The above identification can be guaranteed, by imposing the following equation for $\F$
\begin{equation}\label{eF}
\p_t\F + \DIV_x(\F\otimes \bv) - \nabla_x \bv\,\F + \frac{1}{2}\delta(\theta)(\F\,\F^T\,\F - \F) = 0.
\end{equation}
In fact, multiplying this identity by $\mathbb{F}^T$ on the right, then transposing the equation~\eqref{eF}  and multiplying the result by $\mathbb{F}$ from the left, adding the resulting identity together, we finally observe that $\B := \F \F^T$ then satisfies~\eqref{Gies} with $\gamma=0$.

We have already introduced the complete set of equations and we close the problem by imposing the following boundary and initial conditions
\begin{align}\label{eq:boundaryconditions}
\bv|_{\p\Omega} = 0, \quad \bv(0, x) = \bv_0(x), \quad \B(0, x) = \B_0 = \F_0\,\F_0^T, \quad \nabla_x\theta\cdot \mathbf{n} = 0, \quad \theta(0, x) = \theta_0.
\end{align}
Furthermore, the initial velocity is solenoidal and have bounded kinetic energy and the  internal initial energy and that the initial entropy are also under control, i.e.,
$$
\begin{aligned}
&\|\bv_0\|_2,\,\|\F_0\|_2,\,\|\tr \B_0 - 2 - \ln\det\B_0\|_1, \,\|\ln\det\B_0\|_\infty,\,\|\theta_0\|_1, \|\ln \theta_0\|_1 < +\infty,\\
&\det \F_0 >0, \qquad \DIV_x \bv_0 \equiv 0  \textrm{ a.e. in } \Omega.
\end{aligned}
$$

Finally, we introduce three specific choices of the form of the Helmholtz free energy and in particular the function $g$ appearing in~\eqref{helmholtz_free_energy}:
\begin{enumerate}[label = \bf(P\arabic*)]
    \item {$ g(\theta)\equiv const >0$:} In this case the entropy depends only on the temperature and has the form
    $$
    \eta = c_v\ln \theta.
    $$ \label{contant_g_case}
    Furthermore, one has a relatively good equation for the temperature $\theta$, namely multiplying \eqref{eq:main_syst_entropy} by $\theta$, we obtain
    \begin{align}\label{thetaP1}
    c_v\partial_t \theta + c_v\DIV_x (\theta \, \bv) - \DIV_x(\kappa(\theta)\nabla_x \theta) =  2\nu(\theta)|\D \bv|^2 +\delta(\theta)|\B - \mathbb{I}|^2.%,\quad \kappa(\theta) >0.
\end{align}
    
    \item $g(\theta) = \tilde{g}\theta$ and $\tilde{g}\in \mathbb{R}_+$: In this case,  the internal energy $e$ depends only on the temperature $\theta$ and is of the form
    $$
    e = c_v\theta.
    $$\label{linear_g_case}
    Therefore the equation for the internal energy \eqref{eq:main_syst_internal_energy} is identical to the evolutionary equation for the temperature.
    
    \item {\bf $g:\mathbb{R}_+ \to \mathbb{R}_+$ is continuous, nonnegative and concave:}  In this case, we do not have any simple form of the equation for the internal energy and the equations connecting temperature and the elastic tensor are coupled in a non-trivial way. The evolutionary equation for $\theta$ depends nonlinearly on $\B$ and also $e$.\label{general_g_case}
\end{enumerate}
Our aim is to develop the theory related to the three cases above, in particular, to the most difficult case~\ref{general_g_case}. To simplify the presentation, we prove the weak sequential stability of solutions to~\ref{general_g_case}, see Section~\ref{S:3}. Although it is not the detail proof of existence, it is clear that the complete proof can be established. This fact is very much supported by our second result of the paper, that is connected to the case~\ref{contant_g_case} and is presented in Section~\ref{S:4}, where we show the existence of weak solution. Section~\ref{S:2} is devoted to the basic notation, and in Appendix~\ref{AP} we state some classical auxiliary propositions and recall several useful facts.

To end this introductory part, we recall available results and emphasise the key novelty of the paper. We start by recalling the results without thermal effects. As we have already mentioned, the result for the full Oldroyd-B model is an open problem. The only available result (for global solutions) is due to Lions and Masmoudi~\cite{lions2000global}, where the existence is shown even in three dimensional setting but only for the so-called co-rotational case. On the other hand, for the Giesekus model (or for even more general class), the first serious attempt for global theory was done in~\cite{masmoudi2011global}. Although the key idea there is correct, there were still some gaps in the proof, that had to be corrected. The first rigorous result for such a class of fluids was established in \cite{bathory2021largedata} with one proviso, an additional stress diffusion term $\Delta \mathbb{B}$ was added to the problem, which simplified the analysis a little bit. The problem without stress diffusion was finally solved in~\cite{bulicek2022onplanar} for the two-dimensional setting and later in~\cite{bulicek2024threeD} for the three-dimensional setting and for a much more complex class of models. For compressible setting, we refer also to a particular result in ~\cite{bulicek2019onaclass}.  For problems with temperature, the situation is more delicate. Even for problems without viscoelastic effects, the theory is relatively new, and here we refer to \cite{Co00,BuFeMa09,BuMaRa09} for relevant results. The problems with elastic effects and temperature were first rigorously treated in~\cite{BuMaPrSu21}, where, however, only spherical stresses were assumed, i.e. $\mathbb{B}= b\mathbb{I}$. Much more general results were established in~\cite{BaBuMa24}, where the authors again assumed a modification of the Giesekus model by adding stress diffusion into the system and focussing only on the linear case~\ref{linear_g_case}. The key novelty of the paper is that we do not need any stress diffusion term in the equation or the simple form of the Helmholtz free energy. Furthermore, it is clear from the proof that we are also able to cover the three-dimensional case and also a more complex model - to such a setting will be devoted our forth-coming paper.

To end this part, it is essential to mention that the paper heavily relies not only on the results mentioned above, but also on the following two very classical techniques. First, we should mention the theory of renormalisation for the transport equation developed in~\cite{diperna1989ordinary}, which is used when dealing with the equation for $\mathbb{B}$. Second, the compensated compactness method applied to heat conductive fluids, which was introduced by Feireisl, see~\cite{Fe04} and further generalised and extended to other cases, see~\cite{FeNo17,FeNo22}  and~\cite{BuJuPoZa22}. These methods are used to deduce the compactness of temperature~$\theta$ and tensor~$\mathbb{B}$.

To finish this introductory part,  we recall the system in which we are interested. The above equations and constraints in  \eqref{eq:main_syst}--\eqref{eF} can be written in one form as the equations for the unknowns $(\bv, \theta, \F, p, \T, e)$ 
\begin{equation}\label{main_sys_for_g_theta}
\left\{
\begin{aligned}
    &\partial_t \bv + \DIV_x (\bv\otimes \bv) -\DIV_x\T = 0, \qquad \DIV_x \bv =0, \\
        &\p_t\F + \DIV_x(\F\otimes \bv) - \nabla_x \bv\,\F + \frac{1}{2}\delta(\theta)(\F\,\F^T\,\F - \F) = \mathbb{O},\\
        &\p_t e + \DIV_x(\bv  e) -\DIV_x(\kappa(\theta)\nabla_x\theta) = (2\nu(\theta)\D \bv + 2g(\theta)\F\,\F^T): \D\bv,\\
    &\T=-p\mathbb{I} + 2\nu(\theta)\mathbb{D}\bv + 2g(\theta)(\F\,\F^T),\qquad  e=\theta + (g(\theta) - \theta\,g'(\theta))f(\B).
\end{aligned}
\right.
\end{equation}
with the initial and boundary conditions
\begin{align}\label{eq:boundaryconditions_for_g_theta}
\bv|_{\p\Omega} = 0, \quad \nabla_x\theta\cdot\mathbf{n}|_{\p\Omega} = 0, \quad \bv(0, x) = \bv_0(x), \quad \F(0, x) = \F_0, \quad \theta(0, x) = \theta_0.
\end{align}
%One can observe that $\eqref{eq:main_syst}_3$ follows formally from $\eqref{main_sys_for_g_theta}_4$ by multiplying it by $\F^T$, its transposition by $\F$, and adding both equations together. Let us state the main theorem that we are going to prove.
The key result of the paper, without rigorous definitions, can then be formulated as follows. For rigorous statements, we refer the reader to Sections~\ref{S:3} and \ref{S:4}.
\begin{thm*}
Let $g$ in~\ref{general_g_case} be properly chosen. Then for the relevant initial data, there is a weak global-in-time solution to \eqref{main_sys_for_g_theta}--\eqref{eq:boundaryconditions_for_g_theta}.
\end{thm*}
The notion of a weak global-in-time solution here is also a key concept. While, we always deal with distributional solutions to \eqref{main_sys_for_g_theta}$_1$--\eqref{main_sys_for_g_theta}$_2$. Equation \eqref{main_sys_for_g_theta}$_3$ can hardly be obtained in a general setting. Therefore, we relax a notion of a solution to \eqref{main_sys_for_g_theta}$_3$ such that, we require only the entropy inequality, i.e., we use \eqref{eq:main_syst_entropy} with the inequality sign 
\begin{align}\label{eq:main_syst_entropy_II}
    \partial_t\eta + \DIV_x (\eta\, \bv) - \DIV_x(\kappa(\theta)\nabla_x(\ln\theta)) \ge \frac{\kappa(\theta)|\nabla_x\theta|^2}{\theta^2} + \frac{2\nu(\theta)|\D \bv|^2}{\theta} + \frac{\delta(\theta)|\B - \mathbb{I}|^2}{\theta}%,\quad \kappa(\theta) >0.
\end{align}
and we complete the system by requiring the conservation of the total energy, which means that we require 
\begin{equation}\label{concepts}
\frac{\diff}{\diff t} \int_{\Omega}\frac{|\bv|^2}{2} + e \diff x =0.
\end{equation}
It is worth noticing, that when $e,\eta,\bv$ are sufficiently regular and satisfy \eqref{eq:main_syst_entropy_II} and \eqref{concepts} then they satisfy \eqref{main_sys_for_g_theta}$_3$ as well. For such an approach, we refer the interested reader to \cite{Fe04}.

\section{Preliminaries} \label{S:2}
In this paper, $x\in\Omega \subset \mathbb{R}^d$ always denotes the spatial variable, while $t\in (0, T)$ is reserved for the time variable. Through the paper, we consider mostly the dimension $d=2$, but in several technical tools we also recall the general case $d\ge 2$. The scalar functions are written in italics, e.g. $a\in \mathbb{R}$, the vector-valued objects are written in bold face, e.g. $\vect{a}\in \mathbb{R}^d$, the matrices by capital bold face, e.g. $\mathbb{A}\in \mathbb{R}^{d\times d}$ and the third order tensor as $\mathfrak{A}\in \mathbb{R}^{d\times d \times d}$. In addition, to simplify the notation, we write $\vect{a}\cdot \vect{b}$ for a standard scalar product whenever $\vect{a},\vect{b}\in \mathbb{R}^d$. Similarly, by $\mathbb{A}: \mathbb{B}$, we denote the scalar product between two matrices $\mathbb{A}, \mathbb{B}\in \R^{d\times d}$, and their classical matrix product by $\mathbb{A}\,\mathbb{B}$. Finally, by the symbol $\mathfrak{A}\because \mathfrak{B}$, we will denote the scalar product between $\mathfrak{A}, \mathfrak{B}\in \R^{d\times d\times d}$. In addition, the symbol $\otimes$ is reserved for the tensorial product, i.e., for $\vect{a},\vect{b}\in \mathbb{R}^d$ we denote $\vect{a}\otimes \vect{b} \in \mathbb{R}^{d\times d}$ as $(\vect{a}\otimes \vect{b})_{ij}:=a_ib_j$ for $i,j=1,\ldots, d$. For a matrix $\mathbb{A} = (A_{ij})_{i,j=1}^d$ and a vector $\vect{b}=(b_1,\ldots, b_d)$ we define the third order tensorial product as
$$
(\mathbb{A}\otimes \vect{b})_{ijk} = A_{ij}\,b_k,
$$
and its divergence as
$$
\DIV_x(\mathbb{A}\otimes \vect{b}) = \sum_{j = 1}^d\p_{x_j}(b_j\,\mathbb{A}).
$$
We use the standard notation for the Sobolev and Lebesgue function space, and we frequently do not distinguish between their scalar-, vector-, or matrix-valued variants. In addition, to shorten the notation, we frequently use the following simplifications. When $f \in L^p(\Omega)$, we simplify it to $f \in L^p_x$. Similarly, if $f \in L^p(0,T; L^q(\Omega))$, $f \in L^p(0,T; W^{1,q}(\Omega))$ or $f \in L^p(0,T; W^{1,q}_0(\Omega))$, then we write $f \in L^p_t L^q_x$, $f \in L^p_t W^{1,q}_x$ or $f \in L^p_t W^{1,q}_{0,x}$ respectively (here, $W^{1,q}(\Omega)$ and $W^{1, q}_0(\Omega)$ are the usual Sobolev spaces). In addition, to emphasise spaces with zero divergence, we define $W^{1,p}_{0,\DIV}:=\overline{\{\bv \in \mathcal{C}^1(\Omega; \R^d), \quad \DIV_x \bv =0\}}^{\|\cdot \|_{1,p}}$ and also $L^{p}_{0,\DIV}:=\overline{\{\bv \in \mathcal{C}^1(\Omega; \R^d), \quad \DIV_x \bv =0\}}^{\|\cdot \|_{p}}$ for any $p\in [1,\infty)$. Further, the dual spaces to Sobolev spaces and their subset are defined as usual, e.g. $W^{-1,p'}(\Omega):= (W^{1,p}(\Omega))^*$, $W^{-1,p'}_0(\Omega):= (W^{1,p}_0(\Omega))^*$, $W^{-1,p'}_{0,\DIV}(\Omega):= (W^{1,p}_{0,\DIV}(\Omega))^*$, etc. Here, for $p\in [1, +\infty]$ we denote by $p'$ its H\"{o}lder conjugate.  Also in what follows, by the symbol $C>0$ we denote a generic constant, that can change line to line and will depend only on data. In case, there is any essential dependence on other quantities, it will be clearly described.

Now, we introduce the assumptions regarding the function $\nu, \kappa, \delta, g$ appearing in equations \eqref{eq:main_syst} - \eqref{eq:main_syst_entropy}. We assume that the viscosity $\nu:\mathbb{R}_+ \to \mathbb{R}$ and the heat conductivity $\kappa:\mathbb{R}\to \mathbb{R}$ are continuous functions fulfilling for some positive constants $C_1,C_2>0$ and for all $s\ge 0$
\begin{align}
    C_1 \leq &\kappa(s) \leq C_2,\label{bounds_kappa}\\
    C_1\leq &\nu(s)\leq C_2,\label{bouds_nu}.% \\
    %C^{-1}\leq &\delta(w) \leq C\label{bounds_delta}
\end{align}
Concerning the parameter $\delta:\mathbb{R}\to \mathbb{R}$, we also assume that it is a continuous function, but we need to distinguish between the cases~\ref{contant_g_case} and~\ref{general_g_case}. For~\ref{contant_g_case}, we assume that for all $s\ge 0$ there holds
\begin{align}
    %C^{-1}\leq &\kappa(s) \leq C,\label{bounds_kappa}\\
    %C^{-1}\leq &\nu(s)\leq C,\label{bouds_nu}.% \\
    C_1\leq \delta(s) \leq C_2,\label{bounds_delta}
\end{align}
while in case~\ref{general_g_case}, we require that for all $s\ge 0$ there holds
\begin{align}\label{bounds_delta2}
        C_1(1+s)\le \delta(s) \le C_2(1+s).
\end{align}
Finally, for the most general case~\ref{general_g_case}, we need to specify the assumptions on the function $g$, which read as follows: 
%
%For the case~\ref{contant_g_case} we assume, that there exists $C > 0$ such that
%for any $w \in \R$. And for the case \ref{general_g_case}, that there exist $c_\nu, C_\nu, c_\kappa, C_\kappa > 0$ such that
%\begin{align}
%    c_\kappa\leq &\kappa(w) \leq C_\kappa,\label{bounds_kappa2}\\
%    c_\nu\leq &\nu(w)\leq C_\nu.\label{bounds_nu2}
%\end{align}
%for any $w \in \R$,
%\begin{align}\label{bounds_delta2}
%        \delta(w) \sim w
%\end{align}
%for any $w \in \R$. 
%Moreover, for a function $g$ we assume that
We assume that $g:\mathbb{R}\to \mathbb{R}$ is a $\mathcal{C}^2$ function, which is concave and increasing. In addition, we assume that for all $s\in [0,\infty)$ and all $\lambda\in[0,1]$ there holds
\begin{align}
C_1&\le g(s)\le C_2,\label{function:g_bounds_linfty}\\
0&\le (1+s)g'(s) \le C_2. \label{function:g'_bounds_linfty}%\\
%h_\lambda(s) &:= \int_0^s z^\lambda g''(z)\diff z \ge -C_2.\label{function:g_technical_ass_1} 
\end{align}
We finish this introductory part by defining the auxiliary function $h_{\lambda}$, where the parameter $\lambda\in [0,1]$. We set
\begin{equation}
h_\lambda(s) := \int_0^s z^\lambda g''(z)\diff z.\label{function:g_technical_ass_1} 
\end{equation}
Note that it directly follows from the concavity of $g$ that $h_{\lambda}$ is nonpositive and fulfills 
\begin{equation}
\begin{aligned}
|h_\lambda(s)|&=-\int_0^s z^\lambda g''(z)\diff z \le -\int_0^s (1+z) g''(z)\diff z = -(1+s)g'(s) + g'(0) +\int_0^s g'(z)\diff z \\
&= -(1+s)g'(s) + g'(0) + g(s)-g(0)\le 2C_2, \label{function:g_technical_ass_2}
\end{aligned}
\end{equation}
where we used the assumptions \eqref{function:g_bounds_linfty}--\eqref{function:g'_bounds_linfty}.

%{\tt TODO: function space diver}

%\begin{enumerate}[label = (G\arabic*)]
%    \item $g\in \mathcal{C}^2(\R)$,
%    \item $g$ is concave,
%    \item $g$ is increasing,
%    \item $\max_{w\geq 0} g(w) < +\infty$, $\min_{w\geq 0}g(w) = g(0) > 0$,\label{function:g_bounds_linfty}
%    \item $g'(w)$ is a bounded function,\label{function:g'_bounds_linfty}
%    \item $w\,g'(w)$ and $h_\lambda(w)$ are bounded functions, where $h_\lambda(w) := \int_0^w z^\lambda g''(z)\diff z$.\label{function:g_technical_ass_1}
%    \iffalse
%    \item
%    $$
%    \sup_{w\geq 0}\left(\delta(w)(g'(w)w^{1-a} - h_{1-a}(w)) + \frac{w^a}{c_\nu}(g'(w)w^{1-a} - h_{1-a}(w))^2\right) = c_{\mathrm{tech}} < 1,
%    $$
%    where $h$ is defined as above and $a$ is given in the definition of $\delta$ \eqref{bounds_delta2}.\label{function:g_technical_ass_2}
%    \fi
%\end{enumerate}

\section{Weak sequential stability of strong solutions for the case \ref{general_g_case}}\label{S:3}

We first formulate a rigorous version of the main theorem of the paper. It deals with the most general case~\ref{general_g_case} and the sequential stability of the solutions. 
\begin{thm}\label{main_theorem_case_general_g}
Let  $\{\bv_n, \F_n, \theta_n\}_{n=1}^{\infty}\subset  \mathcal{C}^1((0,T)\times\Omega;\R^2)\times \mathcal{C}^1((0,T)\times\Omega;\R^{2\times 2})\times \mathcal{C}^{1,2}((0,T)\times\Omega;\R_+)$ with $\det \F_n, \theta_n >0$ be a sequence of solutions to~\eqref{main_sys_for_g_theta}--\eqref{eq:boundaryconditions_for_g_theta} with initial conditions $\{\bv_0^n, \theta_0^n, \F_0^n\}_{n=1}^{\infty}$ fulfilling $\det\F_0^n, \theta_0^n > 0$ and 
\begin{equation}\label{conv:init}
\begin{aligned}
\bv_0^n &\to \bv_0 &&\textrm{strongly in } L^2_{0,\DIV},\\
\theta_0^n &\to \theta_0 &&\textrm{strongly in } L^1_x,\\
\ln \theta_0^n &\to \ln \theta_0 &&\textrm{strongly in } L^1_x,\\
\F_0^n &\to \F_0 &&\textrm{strongly in } L^2_x,\\
\ln \det F_0^n &\to \ln \det \F_0 && \textrm{strongly in } L^2_x. 
\end{aligned}
\end{equation}
Then, there exists a subsequence (which we do not relabel) and a triple  $(\bv,\F,\theta)$ such that: we have the following convergence results for $\bv_n$:
\begin{equation}\label{cnv_v}
\begin{aligned}
    \bv_n &\overset{*}{\rightharpoonup} \bv &&\text{ weakly* in }L^\infty_t L^2_x,\\
    \bv_n &\rightharpoonup \bv &&\text{ weakly in }L^2_t W^{1,2}_{0,x},\\
    \bv_n &\rightharpoonup \bv &&\text{ weakly in } L^4_{t,x},\\
    \bv_n &\rightarrow \bv &&\text{ strongly in }\mathcal{C}([0, T]; L^2_{0,\DIV}),
\end{aligned}
\end{equation}
the following convergence results for $\F_n$:
\begin{equation}\label{cnv_F}
\begin{aligned}
    \F_n &\overset{*}{\rightharpoonup} \F &&\text{ weakly* in }L^\infty_t L^2_x,\\
    \F_n &\rightharpoonup \F &&\text{ weakly in } L^4_{t,x},\\
    \F_n &\rightarrow \F &&\text{ strongly in } L^p_{t,x}, 1\leq p < 4,
\end{aligned}
\end{equation}
the following convergence results for $\theta_n$: 
\begin{equation}\label{cnv_theta}
\begin{aligned}
    \theta_n &\rightarrow \theta &&\text{ strongly in }L^{2-\varepsilon}_{t,x},\, \varepsilon\in(0,1),\\
    \nabla_x\theta_n &\rightharpoonup \nabla_x\theta &&\text{ weakly in }L^{\frac{4}{3} - \varepsilon}_{t,x}, \,\varepsilon\in \left(0, \frac{1}{3}\right),
\end{aligned}
\end{equation}
and the following convergence results for ``entropy'' quantities 
\begin{align}
    \ln(\theta_n) &\rightharpoonup \ln(\theta) &&\text{ weakly in }L^2_t W^{1,2}_x,\\
    \ln(\theta_n) &\rightarrow \ln(\theta) &&\text{ strongly in }L^{3 - \varepsilon}_{t,x}, \,\varepsilon\in (0, 1),\\
    \tr (\F_n\,\F_n^T) &\rightharpoonup \mathrm{tr}(\F\,\F^T) &&\text{ weakly in }L^2_{t,x},\\
    \ln\det(\F_n\,\F_n^T) &\overset{*}{\rightharpoonup} \ln\det(\F\,\F^T) &&\text{ weakly* in }L^\infty_t L^{\frac{4}{3} - \varepsilon}_x, \,\varepsilon\in \left(0, \frac{1}{3}\right).
\end{align}
The limiting functions $(\bv, \F, \theta)$ solves  \eqref{main_sys_for_g_theta}--\eqref{eq:boundaryconditions_for_g_theta} in the following sense:
\begin{equation}
\begin{aligned}\label{weak_formulation_u_g_theta}
        &\int_0^T\int_\Omega -\bv \cdot\p_t\bphi - \bv\otimes \bv : \nabla_x\bphi + \nu(\theta)\D\bv :\nabla_x\bphi + g(\theta)\F\,\F^T:\nabla_x\bphi\diff x\diff t \\
        &\qquad = \int_\Omega \bv_0(x)\cdot\bphi(0, x)\diff x
\end{aligned}
\end{equation}
for any $\vect{\varphi}\in \mathcal{C}^1_c([0, T)\times\Omega; \R^2)$ with $\DIV_x\bphi = 0$,
\begin{equation}
\begin{aligned}\label{weak_formulation_F_g_theta}
        &\int_0^T\int_\Omega -\F:\p_t\bG - \F\otimes \bv \because \nabla_x \bG - \nabla_x \bv\,\F : \bG + \frac{1}{2}\delta(\theta)(\F\,\F^T\,\F - \F): \bG\diff x\diff t \\
        &\qquad = \int_\Omega \F_0(x) : \bG(0, x)\diff x
    \end{aligned}
\end{equation}
    for any $\bG\in \mathcal{C}^1_c([0, T)\times\Omega; \R^{2\times 2})$,
\begin{equation}
\begin{aligned}\label{weak_formulation_theta_g_theta}
        &\int_0^T\int_\Omega -(\ln(\theta) - g'(\theta)f(\F\,\F^T))\p_t\phi -(\ln(\theta) - g'(\theta)f(\F\,\F^T))\bv \cdot \nabla_x\phi + \kappa(\theta)\nabla_x\ln\theta\,\nabla_x\phi \diff x\diff t\\
        &\qquad \geq \int_0^T\int_\Omega \left(\frac{\kappa(\theta)|\nabla_x\theta|^2}{\theta^2} + \frac{2\nu(\theta)|\D \bv|^2}{\theta} + \frac{\delta(\theta)|\B - \mathbb{I}|^2}{\theta} \right)\diff x\diff t .%,\quad \kappa(\theta)
    \end{aligned}
\end{equation}
    for any $\phi\in \mathcal{C}^1_c([0, T)\times\Omega)$, $\phi\geq 0$,
\begin{equation}\label{Energy_def}
\begin{aligned}
&-\int_0^T \int_{\Omega} \left(\frac{|\bv|^2}{2} + \theta + (g(\theta) - \theta\,g'(\theta))f(\F\, \F^T)\right) \partial_t \varphi \diff x \diff t\\ 
&\qquad =\varphi(0)\int_{\Omega} \left(\frac{|\bv_0|^2}{2} + \theta_0 + (g(\theta_0) - \theta_0\,g'(\theta_0))f(\F_0\, \F_0^T)\right)\diff x
\end{aligned}
\end{equation} 
for any $\varphi\in \mathcal{C}^1_c([0, T))$. Moreover $\theta > 0$ and $\det\F > 0$.
\end{thm}

%\begin{proof}

The rest of this section is devoted to the proof of Theorem~\ref{main_theorem_case_general_g}, which is divided into several subparts.

\subsection{Auxiliary identities for \texorpdfstring{$\B_n$}{B} and \texorpdfstring{$\theta_n$}{}} 

We begin by recovering the equation for $\B_n := \F_n\,\F_n^T$. Namely, we multiply the equation $\eqref{main_sys_for_g_theta}_2$ by $\F_n^T$ from the right and the transpose of the mentioned equation by $\F_n$ from the left. Adding the results together we obtain
\begin{align}\label{equation_for_B_n_for_g_theta}
        \p_t\B_n + \DIV_x \left(\B_n \bv_n\right) + \delta(\theta_n)(\B_n^2 - \B_n) = \nabla_x \bv_n\,\B_n + \B_n\,(\nabla_x \bv_n)^T.
\end{align}
Next, similarly as in~\eqref{eq:main_syst_entropy}, we can derive the following identity
\begin{equation}
\begin{aligned}\label{eq:entropy_equality_for_g_theta}
    &\partial_t\eta_n + \DIV_x (\eta_n\, \bv_n) - \DIV_x(\kappa(\theta_n)\nabla_x(\ln\theta_n)) \\
    &\qquad = \frac{\kappa(\theta_n)|\nabla_x\theta_n|^2}{\theta_n^2} + \frac{2\nu(\theta_n)|\D \bv_n|^2}{\theta_n} + \frac{\delta(\theta_n)|\B_n - \mathbb{I}|^2}{\theta_n},%,\quad \kappa(\theta) >0.
\end{aligned}
\end{equation}
where the entropy is given as 
\begin{equation}\label{eq:entropyn}
\eta_n:= \ln \theta_n -g'(\theta_n) f(\B_n).
\end{equation}
We also at this point recall the definition of the internal energy
\begin{equation}\label{inten}
e_n:=\theta_n + (g(\theta_n) - \theta_n\,g'(\theta_n))f(\B_n).
\end{equation}
%
%
%With this, we may apply the material derivative to the equation \eqref{equation_for_internal_energy_for_g_theta} to get
%$$
%    \p_t e_n + u_n\,\nabla_x e_n = \p_{\B}\psi(\theta_n, \B_n)(\p_t\B_n + u_n\,\nabla_x\B_n) + \theta_n(\p_t\eta_n + u_n\nabla_x\eta_n),
%$$
%which together with the equations $\eqref{main_sys_for_g_theta}_5$, \eqref{equation_for_B_n_for_g_theta} and the fact that $f'(\B_n) = \mathbb{I} - \B_n^{-1}$ implies
%\begin{align}\label{eq:entropy_equality_for_g_theta}
%       \p_t\eta_n + u_n\cdot\nabla_x\eta_n - \DIV_x(\kappa(\theta_n)\nabla_x(\ln\theta_n)) = \kappa(\theta_n)\frac{|\nabla_x\theta_n|^2}{\theta^2_n} + %\nu(\theta_n)\frac{|Du_n|^2}{\theta_n} + \delta(\theta_n)\frac{|\B_n - \mathbb{I}|^2}{\theta_n}.
%\end{align}
Before moving forward, let us prove a structural lemma connected to equation \eqref{equation_for_B_n_for_g_theta} and \eqref{eq:entropy_equality_for_g_theta}.
\begin{lem}\label{lem:renormalized_equation_for_theta}
Let a triple $(\bv_n,\theta_n, \B_n)$ solve \eqref{eq:entropy_equality_for_g_theta}--\eqref{equation_for_B_n_for_g_theta} with $\eta_n := \ln(\theta_n) - g'(\theta_n)f(\B_n)$.  Then, for any $\lambda > 0$ there holds
\begin{equation}\label{lem:renorm}
            \begin{aligned}
                &\p_t \left( \frac{\theta_n^{\lambda}}{\lambda} - h_\lambda(\theta_n)\, f(\B_n)\right) + \DIV_x\left( \left(\frac{\theta_n^{\lambda}}{\lambda} - h_\lambda(\theta_n)\,f(\B_n)\right)\bv_n\right)\\
                &\quad- \DIV_x\left(\kappa(\theta_n)\nabla_x \frac{\theta_n^{\lambda}}{\lambda}\right)- 2(\B_n - \mathbb{I}) : \D\bv_n(g'(\theta_n)\theta_n^\lambda - h_\lambda(\theta_n))\\
                &= \left(\frac{(1-\lambda)\kappa(\theta_n)|\nabla_x\theta_n|^2}{\theta_n^2} + \frac{2\nu(\theta_n)|\D\bv_n|^2}{\theta_n} + \frac{\delta(\theta_n)|\B_n - \mathbb{I}|^2}{\theta_n}(h_{\lambda}(\theta_n)\theta^{1-\lambda}+1-g'(\theta_n)\theta_n)\right)\theta_n^\lambda %-  \delta(\theta_n)\frac{|\B -\mathbb{I}|^2}{\theta_n} (g'(\theta_n)\theta_n^\lambda - h_\lambda(\theta_n)) \theta_n^{\lambda},
            \end{aligned}
\end{equation}
where $h_{\lambda}(\theta)$ is defined in~\eqref{function:g_technical_ass_1} and $f(\B)$ is given by \eqref{fdf}.
\end{lem}
\begin{proof}
Taking the scalar product of~\eqref{equation_for_B_n_for_g_theta} with $\mathbb{I} - \B_n^{-1}$ and using the fact that $\partial_{\B_n} f(\B_n) = \mathbb{I} - \B_n^{-1}$, we deduce that
\begin{align}\label{eq:equation_for_f_of_B_for_g_theta}
            \p_t f(\B_n) + \DIV_x (f(\B_n) \bv_n) + \delta(\theta)|\B_n - \mathbb{I}|^2 = 2(\B_n - \mathbb{I}) : \D \bv_n.
\end{align}
Using the definition of the entropy in~\eqref{eq:entropy_equality_for_g_theta}, we have
\begin{equation}
\begin{aligned}\label{eq:entropy_equality_for_g_theta2}
    &\partial_{t}\theta_n \left(\frac{1}{\theta_n}-g''(\theta_n) f(\B_n)\right) -g'(\theta_n)\partial_t f(\B_n) \\
    &\quad + \DIV_x (\theta_n\, \bv_n) \left(\frac{1}{\theta_n}-g''(\theta_n) f(\B_n)\right) - g'(\theta_n)\DIV_x (f(\B_n)\, \bv_n)  \\
    &\quad - \DIV_x(\kappa(\theta_n)\nabla_x(\ln\theta_n)) \\
    & = \frac{\kappa(\theta_n)|\nabla_x\theta_n|^2}{\theta_n^2} + \frac{2\nu(\theta_n)|\D \bv_n|^2}{\theta_n} + \frac{\delta(\theta_n)|\B_n - \mathbb{I}|^2}{\theta_n},%,\quad \kappa(\theta) >0.
\end{aligned}
\end{equation}
Then we multiply~\eqref{eq:equation_for_f_of_B_for_g_theta}  by $g'(\theta_n)$ and add the result to  into \eqref{eq:entropy_equality_for_g_theta2}, we obtain
\begin{equation*}
\begin{aligned}%\label{eq:entropy_equality_for_g_theta2}
    &\partial_{t}\theta_n \left(\frac{1}{\theta_n}-g''(\theta_n) f(\B_n)\right)  + \DIV_x (\theta_n\, \bv_n) \left(\frac{1}{\theta_n}-g''(\theta_n) f(\B_n)\right)  \\
    &\qquad +g'(\theta_n) \delta(\theta)|\B_n - \mathbb{I}|^2 - \DIV_x(\kappa(\theta_n)\nabla_x(\ln\theta_n)) \\
    & = \frac{\kappa(\theta_n)|\nabla_x\theta_n|^2}{\theta_n^2} + \frac{2\nu(\theta_n)|\D \bv_n|^2}{\theta_n} + \frac{\delta(\theta_n)|\B_n - \mathbb{I}|^2}{\theta_n} + 2g'(\theta_n)(\B_n - \mathbb{I}) : \D \bv_n.%,\quad \kappa(\theta) >0.
\end{aligned}
\end{equation*}
Next, we multiply the result by $\theta_n^{\lambda}$ and use the definition of $h_{\lambda}$ in~\eqref{function:g_technical_ass_1} to observe that
\begin{equation}
\begin{aligned}\label{some_equation_for_lemma}
    &\partial_{t}\left(\frac{\theta_n^{\lambda}}{\lambda}\right)-f(\B_n) \partial_t h_{\lambda}(\theta_n)  + \DIV_x \left(\frac{\theta_n^{\lambda}\, \bv_n}{\lambda}\right) -f(\B_n) \DIV_x (h_{\lambda}(\theta_n)\, \bv_n) \\
    &\qquad - \DIV_x\left(\frac{\kappa(\theta_n)\nabla_x \theta_n^{\lambda}}{\lambda}\right)-2g'(\theta_n)\theta_n^{\lambda}(\B_n - \mathbb{I}) : \D \bv_n.\\
    & = \left(\frac{\kappa(\theta_n)(1-\lambda)|\nabla_x\theta_n|^2}{\theta_n^2} + \frac{2\nu(\theta_n)|\D \bv_n|^2}{\theta_n} + \frac{\delta(\theta_n)(1-g'(\theta_n)\theta_n)|\B_n - \mathbb{I}|^2}{\theta_n}\right)\theta_n^{\lambda} %,\quad \kappa(\theta) >0.
\end{aligned}
\end{equation}
%Hence, after we multiply the equation above by $\theta^\lambda$ we get
%        \begin{equation}\label{some_equation_for_lemma}
%            \begin{split}
%                &\frac{1}{\lambda}\p_t\theta^\lambda - \p_t[h_\lambda(\theta)]\, f(\B) + \frac{1}{\lambda}u\cdot\nabla_x\theta^\lambda - u\cdot\nabla_x[h_\lambda(\theta)]\,f(\B)\\
%                &\phantom{=}- \DIV_x(\kappa(\theta)\nabla_x(\ln\theta))\theta^\lambda  + g'(\theta)\theta^\lambda\delta(\theta)|\B -\mathbb{I}|^2 - 2g'(\theta)\theta^\lambda(\B - \mathbb{I}) : Du\\
%                &= \left(\kappa(\theta)\frac{|\nabla_x\theta|^2}{\theta^2} + \nu(\theta)\frac{|Du|^2}{\theta} + \delta(\theta)\frac{|\B - \mathbb{I}|^2}{\theta}\right)\theta^\lambda.
%            \end{split}
%        \end{equation}
Finally, we multiply \eqref{eq:equation_for_f_of_B_for_g_theta} by $h_\lambda(\theta_n)$ and subtract the result from \eqref{some_equation_for_lemma} to get 
\begin{equation*}
\begin{aligned}%\label{some_equation_for_lemma}
    &\partial_{t}\left(\frac{\theta_n^{\lambda}}{\lambda}- h_{\lambda}(\theta_n)f(\B_n) \right)  + \DIV_x \left(\left(\frac{\theta_n^{\lambda}}{\lambda} -h_{\lambda}(\theta_n)f(\B_n)\right)\bv_n\right) \\
    &\qquad - \DIV_x\left(\frac{\kappa(\theta_n)\nabla_x \theta_n^{\lambda}}{\lambda}\right)+2\left(h_{\lambda}(\theta_n)-g'(\theta_n)\theta_n^{\lambda}\right)(\B_n - \mathbb{I}) : \D \bv_n\\
    & = \left(\frac{\kappa(\theta_n)(1-\lambda)|\nabla_x\theta_n|^2}{\theta_n^2} + \frac{2\nu(\theta_n)|\D \bv_n|^2}{\theta_n}\right)\theta_n^{\lambda} + \delta(\theta_n)(h_{\lambda}(\theta_n)+\theta_n^{\lambda-1}-g'(\theta_n)\theta_n^{\lambda})|\B_n - \mathbb{I}|^2%,\quad \kappa(\theta) >0.
\end{aligned}
\end{equation*}
and we see that \eqref{lem:renorm} directly follows.
\end{proof}

\subsection{Uniform estimates}
We proceed and show some basic bounds for our sequence. Multiplying $\eqref{main_sys_for_g_theta}_1$ by $\bv_n$ and adding the result  to $\eqref{main_sys_for_g_theta}_3$, integrating the result over $\Omega$, using the boundary conditions for $\bv_n$ and $\theta_n$ in \eqref{eq:boundaryconditions_for_g_theta} and the fact that $\DIV_x \bv_n=0$, we deduce
\begin{equation}\label{Energy_n}
\partial_t \int_{\Omega} \frac{|\bv_n|^2}{2} + e_n \diff x =0.
\end{equation} 
First, the internal energy can be estimated as
$$
e_n=\theta_n + (g(\theta_n)-\theta_n g'(\theta_n))f(\B_n)\ge \theta_n +g(0)f(\B_n)\ge \theta_n + C_1f(\B_n),
$$
where we used the concavity of $g$ and the assumption~\eqref{function:g_bounds_linfty}. Since $f$ is nonnegative function and $\theta_n>0$, we see that also $e_n>0$ and using the convergence properties of the initial conditions~\eqref{conv:init}, we have
\begin{equation}
\begin{aligned}\label{u_n_and_theta_n_Linfty_bounds}
      \|\bv_n\|^2_{L^\infty_t L^2_x} + \|\theta_n\|_{L^\infty_t L^1_x} + C_1\|f(\B_n)\|_{L^\infty_t L^1_x}  &\leq \|\bv_n\|^2_{L^\infty_t L^2_x} + \|e_n\|_{L^\infty_t L^1_x} \\
      &= \|\bv^n_0\|^2_{ L^2_x} + \|e_0^n\|_{ L^1_x}\le C.
\end{aligned}
\end{equation}
We continue, by deducing the estimates coming from the entropy. We integrate \eqref{eq:entropy_equality_for_g_theta} over $\Omega$, use the boundary conditions and the fact that  $\DIV_x \bv_n=0$, and also using the assumption on the function $\delta$ stated in~\eqref{bounds_delta2}, we get that for any $\tau\in (0,T)$
\begin{align*}
        C_1&\int_0^{\tau}\int_\Omega \frac{|\nabla_x \theta_n|^2}{(\theta_n)^2}+\frac{|\D \bv_n|^2}{\theta_n}+\frac{(1+\theta_n)|\B_n - \mathbb{I}|^2}{\theta_n}\diff x\diff\tau \\
        &\leq \int_0^{\tau}\int_{\Omega}\p_t\eta_n + \DIV_x(\eta_n \bv_n) - \DIV_x(\kappa(\theta_n)\nabla_x(\ln\theta_n))\diff x\diff\tau\\
        &= \left(\int_{\Omega} \ln(\theta_n) - g'(\theta_n)\,f(\B_n)\diff x\right)\Bigg|_0^{\tau}.
\end{align*}
Since $\tau\in (0,T)$ is arbitrary, $g$ is nondecreasing,  $f$ is nonnegative and $\ln x \leq x - 1$, we can use the assumptions~\eqref{function:g_technical_ass_1} and \eqref{bounds_kappa} to deduce %apply the bound~\eqref{u_n_and_theta_n_Linfty_bounds}, and the Assumptions \ref{function:g_bounds_linfty}, \ref{function:g_technical_ass_1} implies
\begin{equation}
\begin{aligned}\label{bound_on_B_n_L2}
       \left\| \frac{\D\bv_n}{\sqrt{\theta_n}}\right\|_{L^2_{t,x}}+ \|\ln \theta_n\|_{L^{\infty}_t L^1_x} + \|\nabla_x \ln \theta_n\|_{L^{2}_{t,x}}+ \left\|\frac{(1+\sqrt{\theta_n})(\B_n - \mathbb{I})}{\sqrt{\theta_n}}\right\|_{L^2_{t,x}} &\le C,%(\|\bv_0^n\|_{ L^2_x}, \|\theta^n_0\|_{L^1_x}, \|\ln \theta_0^n\|_{ L^1_x}, \|f(\B_0^n)\|_{ L^1_x})\\
        %&\le C.
\end{aligned}
\end{equation}
where we employed also the uniform estimate \eqref{u_n_and_theta_n_Linfty_bounds}.
In particular, we deduce from the above estimates and the fact that $\B_n = \F_n\, \F_n^T$ that 
\begin{align}
         \|\B_n\|_{L^2_{t,x}}&=\|\F_n\,\F_n^T\|_{L^2_{t,x}} \leq C,\label{bounds_on_F_n_F_n_transposed_L2}\\
         \|\F_n\|_{L^4_{t,x}} &\leq C.\label{bounds_on_F_n_L4}
\end{align}
We continue with uniform estimates for $\bv_n$. We take the scalar product of $\eqref{main_sys_for_g_theta}_1$ with $\bv_n$ and integrate the result over $\Omega$ to get for arbitrary $t\in (0,T)$
\begin{equation}
\begin{aligned}\label{kinetic:energy}
        &\frac{1}{2}\int_\Omega |\bv_n(t)|^2\diff x + \int_0^t\int_{\Omega}2\nu(\theta_n)|\D \bv_n|^2\diff x\diff \tau \\
        &\qquad = -\int_0^t\int_\Omega g(\theta_n)\B_n : \D \bv_n\diff x\diff \tau + \frac{1}{2}\int_{\Omega}|\bv_0^n|^2\diff x.
\end{aligned}
\end{equation}
Employing the Young inequality, we get
\begin{align*}
        \int_0^t\int_{\Omega}\nu(\theta_n)|\D \bv_n|^2\diff x\diff \tau \le \int_0^t\int_\Omega \frac{(g(\theta_n))^2}{\nu(\theta_n)}|\B_n|^2 \diff x\diff \tau + \int_{\Omega}|\bv_0^n|^2\diff x.
\end{align*} 
Hence, we may use the uniform estimate~\eqref{bounds_on_F_n_F_n_transposed_L2}, the assumption on the parameters \eqref{bouds_nu}, \eqref{function:g_bounds_linfty} and the assumptions on the initial data~\eqref{conv:init} and deduce that 
%\begin{align*}
%        &\int_0^t\int_{\Omega}2\nu(\theta_n)|D u_n|^2\diff x\diff t \leq \frac{\|g\|_\infty}{8c_\nu}\int_0^t\int_\Omega|\F_n\,\F_n^T|^2\diff x\diff \tau + 2c_\nu\int_0^t\int_\Omega |D u_n|^2\diff x\diff\tau + \frac{1}{2}\|u_0\|^2_2\\
%        &\phantom{=}\leq 2c_\nu\int_0^t\int_\Omega |D u_n|^2\diff x\diff\tau + C(\|u_0\|_{ L^2_x}, \|\theta_0\|_{L^1_x}, \|\ln(\theta_0)\|_{ L^1_x}, \|f(\B_0)\|_{ L^1_x}).
%\end{align*}
%Thus
\begin{align}\label{bounds_on_symetric_gradient_u_n_L2}
        \|\D \bv_n\|_{L^2_{t,x}} \leq C(\|\bv_0^n\|_{ L^2_x}, \|\theta_0^n\|_{L^1_x}, \|\ln \theta_0^n \|_{ L^1_x}, \|f(\B_0^n)\|_{ L^1_x}) \le C.,
\end{align}
The above estimate, the assumption \eqref{bounds_kappa} and the Korn inequality imply that 
\begin{align}
        \|\sqrt{\nu(\theta_n)}\D \bv_n\|_{L^2_{t,x}} &\leq C,\label{bounds_on_symmetric_gradient_u_n_nu_L2}\\
        \|\nabla_x \bv_n\|_{L^2_{t,x}} &\leq C. \label{bounds_on_gradient_u_n_nu_L2}
\end{align}
The Sobolev embedding, the interpolation theorem and the uniform bound~\eqref{u_n_and_theta_n_Linfty_bounds} further lead to    
\begin{align}
        \|\bv_n\|_{L^4_{t,x}} &\leq C,\label{bounds_on_u_n_nu_L4}\\
        \|\bv_n\|_{L^2_t L^b_x} &\leq C(b) \quad \textrm{ for all } b < +\infty.\label{bounds_on_u_n_nu_L2_Lalmostinfty}
\end{align}
To finish the uniform bounds for the velocity, we consider the estimates on $\partial_t \bv_n$. It follows from  \eqref{main_sys_for_g_theta} that  
\begin{equation}
\begin{aligned}\label{bound_on_partial_t_u_n}
        \|\p_t \bv_n\|_{L^2_t W^{-1,2}_{0, \DIV}} \le \left\| \bv_n \otimes \bv_n - 2\nu(\theta_n)\D \bv_n -2g(\theta_n)\B_n\right\|_{L^2_{t,x}} \le C,
\end{aligned}
\end{equation}    
where for the second inequality, we used the assumptions \eqref{bouds_nu} and \eqref{function:g_bounds_linfty}, and the uniform estimates \eqref{bounds_on_F_n_F_n_transposed_L2}, \eqref{bounds_on_gradient_u_n_nu_L2} and  \eqref{bounds_on_u_n_nu_L4}.   
%    
%let us also note that the bounds \eqref{bounds_on_symmetric_gradient_u_n_nu_L2}, and \eqref{bounds_on_F_n_L4}, as well as the equation %$\eqref{main_sys_for_g_theta}_1$ we may deduce that    
    
We continue with improving the estimates. We start with $\F_n$. Taking the scalar product of~\eqref{main_sys_for_g_theta}$_2$ with~$\F_n$ and integrating the result over $(0,t)\times \Omega$, we deduce the identity 
\begin{equation}
\begin{aligned}\label{some_random_inequality2}
        &\|\F_n(t)\|_2^2 + \int_0^t\int_\Omega\delta(\theta_n)|\F_n\,\F_n^T|^2\diff x\diff \tau \\
        &\qquad = \int_0^t\int_\Omega 2\B_n : \D \bv_n\diff x\diff \tau + \int_0^t\int_\Omega\delta(\theta_n)|\F_n|^2\diff x\diff\tau + \|\F_0^n\|_2^2,
\end{aligned}
\end{equation}
which by using the H\"{o}lder inequality, Young inequality and a matrix inequality $|\F|^4 \leq 2|\F\,\F^T|^2$ implies
    \begin{multline*}
        \|\F_n(t)\|_2^2 + \frac{1}{2}\int_0^t\int_\Omega\delta(\theta_n)|\F_n|^4\diff x\diff\tau \\
        \leq 2\|\F_n\,\F_n^T\|_2\|\D \bv_n\|_2 + \int_0^t\int_\Omega \delta(\theta_n)\diff x\diff\tau + \frac{1}{4}\int_0^t\int_\Omega\delta(\theta_n)|\F_n|^4\diff x\diff\tau + \|\F_0\|_2^2.
    \end{multline*}
Hence, by \eqref{bounds_on_F_n_F_n_transposed_L2}, \eqref{bounds_on_symetric_gradient_u_n_L2}, the assumptions postulated for  $\delta$ in ~\eqref{bounds_delta2}, the bounds \eqref{u_n_and_theta_n_Linfty_bounds} and the fact that $t\in (0,T)$ is arbitrary, we have
\begin{align}\label{bounds_on_F_n_Linfty}
        \|\F_n\|_{L^\infty_t L^2_x}^2 + \int_0^T\int_\Omega\delta(\theta_n)|\F_n|^4\diff x\diff t \leq C.
\end{align}
In particular, coming back with the obtained bound to the inequality \eqref{some_random_inequality2} gives us
\begin{align}\label{bounds_on_F_n_with_delta}
        \int_0^T \int_\Omega\delta(\theta_n)|\F_n\,\F_n^T|^2\diff x\diff t \leq C.
\end{align}

To derive some compactness for $\theta_n$ we use Lemma~\ref{lem:renormalized_equation_for_theta}. We integrate \eqref{lem:renorm} over $\Omega$ and $(0,t)$, neglect terms having the proper sign to deduce for any $\lambda \in (0,1)$ that
\begin{equation}\label{prpr}
    \begin{split}
        &(1-\lambda)\int_0^t\int_\Omega \kappa(\theta_n)\frac{|\nabla_x\theta_n|^2}{\theta_n^{2-\lambda}}\diff x\diff \tau \\
        &\quad \leq \frac{1}{\lambda}\int_{\Omega}\theta_n^\lambda(t)\diff x\diff\tau - \int_\Omega h_\lambda(\theta_n(t))\,f(\B_n(t))\diff x\\
        &\qquad + \int_0^t\int_\Omega \left(\delta(\theta_n)|\B_n - \mathbb{I}|^2-2(\B_n - \mathbb{I}):\D \bv_n\right)\left(g'(\theta_n)\theta_n^\lambda - h_\lambda(\theta_n)\right)\diff x\diff \tau.
    \end{split}
\end{equation}
Using the assumption the assumption \eqref{function:g'_bounds_linfty} and the bound on $h_{\lambda}$ in \eqref{function:g_technical_ass_2}, we have
$$
|g'(\theta_n)\theta_n^\lambda| + |h_{\lambda}(\theta_n)| \le C.
$$
Thus, using this bound in \eqref{prpr}, combining it with the H\"{o}lder inequality and the assumption on $\kappa$ in \eqref{bounds_kappa} and the already derived uniform estimates \eqref{u_n_and_theta_n_Linfty_bounds}, \eqref{bounds_on_F_n_with_delta}, \eqref{bounds_on_symetric_gradient_u_n_L2}, \eqref{bound_on_B_n_L2}, we get
\begin{align}\label{bound_grad_theta_n_with_lambda}
        \int_0^T\int_{\Omega}\frac{|\nabla_x\theta_{n}|^2}{\theta_n^{\lambda}}\diff x\diff t \leq C(\lambda),\quad\text{ for all }\lambda\in (1, 2).
\end{align}
With help of this estimate, we continue with further estimates for $\theta_n$. We recall the interpolation inequality
\begin{equation}\label{interpolation_inequality_g}
\|v\|^a_{L^a(\Omega)}\leq C\left(\|v\|^a_{L^{\frac{2}{2-\lambda}}(\Omega)}+\|v\|^{a-2}_{L^{\frac{2}{2-\lambda}}
(\Omega)}\|\nabla_x v\|^{2}_{L^{2}(\Omega)} \right),
\end{equation}
which is valid for all $\lambda \in [1,2)$  and an $a$ defined as
$$
a:=2 + \frac{2}{2-\lambda}.
$$
Hence, with using the interpolation~\eqref{interpolation_inequality_g} and the uniform bounds \eqref{u_n_and_theta_n_Linfty_bounds} and~\eqref{bound_grad_theta_n_with_lambda}, we see that for any $\lambda\in (1,2)$
\begin{equation}\label{almost_uniform_bound_temp_g}
\begin{split}
    \int_0^T\int_\Omega \theta_n^{3 - \lambda}\diff x\diff t &= \int_0^T\int_\Omega \left(\theta_n^{\frac{2 - \lambda}{2}}\right)^a\diff x\diff t \le C \int_0^T\|\theta_n^{\frac{2 - \lambda}{2}}\|_{\frac{2}{2 - \lambda}}^a + \|\theta^{\frac{2 - \lambda}{2}}_n\|_{\frac{2}{2 - \lambda}}^{a - 2}\|\nabla_x(\theta_n^{\frac{2-\lambda}{2}})\|_2^2\diff t\\
    &= C\int_0^T\left(\int_{\Omega}\theta_n\diff x\right)^{3 - \lambda}\diff t + C(\lambda)\int_0^T\left(\int_\Omega\theta_n\diff x\right)\,\left(\int_{\Omega}\frac{|\nabla_x\theta_n|^2}{\theta_n^\lambda}\diff x\right)\diff t\\
    &\leq  C(\lambda).
\end{split}
\end{equation}
We also show the inhomogeneous estimate for $\theta_n$. Recall another interpolation inequality
\begin{equation}\label{inter23}
\|v\|_{L^p(\Omega)} \le C(p,\Omega)\left(\|v\|_{L^2(\Omega)} + \|v\|^{\frac{2}{p}}_{L^2(\Omega)}\|\nabla_x v\|_{L^2(\Omega)}^{\frac{p-2}{p}}\right),
\end{equation}
which is valid for all $p\in (2,\infty)$. Then,  for arbitrary~$\lambda\in (1,2)$, we set~$v:=(\theta_n)^{\frac{2-\lambda}{2}}$ in~\eqref{inter23} and by using the H\"{o}lder inequality and the uniform bounds~\eqref{u_n_and_theta_n_Linfty_bounds} and~\eqref{bound_grad_theta_n_with_lambda}, we deduce 
\begin{equation}\label{inhomest}
\begin{aligned}
\|(\theta_n)^{\frac{2-\lambda}{2}}\|^{\frac{2p}{p-2}}_{L_t^{\frac{2p}{p-2}} L_x^p}&=\int_0^T \|(\theta_n)^{\frac{2-\lambda}{2}}\|^{\frac{2p}{p-2}}_{L_x^p}\diff t \\
&\le C(p)\int_0^T \|(\theta_n)^{\frac{2-\lambda}{2}}\|^{\frac{2p}{p-2}}_{L^2_x} + \|(\theta_n)^{\frac{2-\lambda}{2}}\|^{\frac{4}{p-2}}_{L^2_x}\|\nabla_x (\theta_n)^{\frac{2-\lambda}{2}}\|_{L^2_x}^{2} \diff t\\
&\le C(p,\lambda) \left(\|\theta_n\|^{\frac{p(2-\lambda)}{p-2}}_{L^{\infty}_t L^1_x} + \|\theta_n\|^{\frac{2(2-\lambda)}{p-2}}_{L^{\infty}_t L^2_x} \int_0^T \int_{\Omega} \frac{|\nabla_x \theta_n|^2}{(\theta_n)^{\lambda}}\diff x \diff t \right)\le C(p,\lambda).
\end{aligned}
\end{equation}
Next, with the use of Young's inequality, \eqref{almost_uniform_bound_temp_g} and \eqref{bound_grad_theta_n_with_lambda} we may infer
\begin{equation}\label{almost_uniform_bound_grad_temp_g}
\begin{split}
\int_0^T\int_\Omega|\nabla_x \theta_n|^{2 - \frac{2}{3}\lambda}\diff x\diff t &= \int_0^T\int_\Omega\left(\frac{|\nabla_x \theta_n|^{2}}{\theta_n^\lambda}\right)^{\frac{2 - \frac{2}{3}\lambda}{2}}\theta_n^{\frac{(2 - \frac{2}{3}\lambda)\lambda}{2}}\diff x\diff t\\
&\leq \int_0^T\int_\Omega\frac{|\nabla_x \theta_n|^{2}}{\theta_n^\lambda}+\theta_n^{\frac{(2 - \frac{2}{3}\lambda)\lambda}{2 - (2 - \frac{2}{3}\lambda)}}\diff x\diff t\\
&= \int_0^T\int_\Omega\frac{|\nabla_x \theta_n|^{2}}{\theta_n^\lambda}+\theta_n^{3 - \lambda}\diff x\diff t\le C(\lambda),
\end{split}
\end{equation}
for any $\lambda\in(1,2)$. To summarize, it follows from \eqref{inhomest} (where $p>2$ and $\lambda\in (1,2)$ are arbitrary), from
\eqref{almost_uniform_bound_temp_g} and \eqref{almost_uniform_bound_grad_temp_g} give the bounds
    \begin{align}
    \|\theta_n\|_{L_t^{q} L_x^{p}}&\le C(p,q)  &&\textrm{for all } p\in [1,\infty) \textrm{ and } q\in \left[1,\frac{p}{p-1}\right),\label{final_bounds_on_theta_n2}\\
         \|\theta_n\|_{L^{q}_{t,x}} &\leq C(q) &&\textrm{for all } q\in [1,2), \label{final_bounds_on_theta_n}\\
        \|\nabla_x\theta_n\|_{L^{q}_{t,x}} &\leq C(q) &&\textrm{for all } q\in \left[1,\frac43\right).\label{final_bounds_on_grad_theta_n}
    \end{align}
Next goal is to establish the integrability of the sequence $f(\B_n)$, which appears in the internal energy and entropy. First, due to the matrix inequality $|\tr \F | \leq \sqrt{2}|\F|$ we have 
\begin{equation}
    \begin{aligned}\label{bounds_on_trace_B_n}
        &\|\tr \B_n\|_{L^2_{t,x}} \leq \sqrt{2}\|\B_n\|_{L^2_{t,x}} \leq C,\\
        &\|\tr \B_n \|_{L^\infty_t L^1_x} \leq \sqrt{2}\|\B_n\|_{L^\infty_t L^1_x} \leq C.
    \end{aligned}
\end{equation}
The above estimate gives us the information on the one part of $f(\B_n)$. To get also the estimate on~$\ln \det \B_n$, we derive the identity for this quantity. To do so,  we take the scalar product of~\eqref{equation_for_B_n_for_g_theta} with $\B_n^{-1}$ (note that $\B_n$ is assumed to be positively definite) and get (using the fact that $\DIV_x \bv_n =0$)
\begin{align}\label{equation_for_ln_det_B_n}
        \p_t\ln\det\B_n + \DIV_x (\bv_n \ln\det\B_n) + \delta(\theta_n) \tr (\B_n - \mathbb{I}) = 0.
\end{align}
To deduce the information for the last term on the right hand side, we first recall the assumption on $\delta$ in~\eqref{bounds_delta2} and then it directly follows from \eqref{u_n_and_theta_n_Linfty_bounds} and  \eqref{final_bounds_on_theta_n} that 
%    \iffalse
%    \begin{align*}
%        \|\delta(\theta_n)\|_{L^{\frac{1}{a}(2 - \lambda)}} \leq C(T, \lambda, \|u_0\|_2, \|\F_0\|_2, \|\theta_0\|_1, \|f(\B_0)\|_1),
%    \end{align*}
%    for $\lambda \in (0, 1)$, which makes us arrive at
%    \fi
\begin{align}\label{bounds_for_delta_theta_n}
        \|\delta(\theta_n)\|_{L_t^{r} L_x^{p}}  + \|\delta(\theta_n)\|_{L^{\infty}_{t} L^1_{x}}\leq C(q,p,r)
    \end{align}
for all $q\in [1,2)$, all $p\in [1,\infty)$ and all $r\in[1,\frac{p}{p-1})$. With the help of this identity, we derive two estimates, the first one homogeneous with respect to space and time variables and the second one to get the optimal control in~\eqref{equation_for_ln_det_B_n}. 
Thus, by the H\"{o}lder inequality, the bounds \eqref{bounds_for_delta_theta_n}, \eqref{bounds_on_F_n_with_delta}, and the matrix inequality $|\tr \F| \leq \sqrt{2}|\F|$  we obtain for all $\varepsilon\in (0,1)$
\begin{equation}\label{some_random_inequality3}
    \begin{split}
        &\int_0^T\int_\Omega |\delta(\theta_n)\tr (\B_n - \mathbb{I})|^{\frac{4- 2\varepsilon}{3 - \varepsilon}}\diff x\diff t =\int_0^T\int_\Omega \left(\delta(\theta_n)\right)^{\frac{2- \varepsilon}{3 - \varepsilon}} \left(\delta(\theta_n)|\tr(\B_n - \mathbb{I})|^2\right)^{\frac{2- \varepsilon}{3 - \varepsilon}}\diff x\diff t\\
        &\quad \leq \left(\int_0^T \int_{\Omega} (\delta(\theta_n))^{2-\varepsilon} \diff x \diff t\right)^{\frac{1}{3-\varepsilon}}\left(\int_0^T\int_\Omega \delta(\theta_n)|\mathrm{tr}(\B_n - \mathbb{I})|^2\diff x\diff t\right)^{\frac{2-\varepsilon}{3 - \varepsilon}}\leq C(\varepsilon),
    \end{split}
\end{equation}
and similarly,  for any $q\in [1,2)$ we can use the H\"{o}lder inequality and the uniform bound~\eqref{bounds_for_delta_theta_n} to deduce
\begin{equation}\label{some_random_inequality32}
    \begin{split}
        &\int_0^T\left(\int_\Omega |\delta(\theta_n)\tr (\B_n - \mathbb{I})|^{q}\diff x\right)^{\frac{1}{q}}\diff t =\int_0^T\left(\int_\Omega \left(\delta(\theta_n)\right)^{\frac{q}{2}}  \left(\delta(\theta_n)|\tr (\B_n - \mathbb{I})|^2\right)^{\frac{q}{2}}\diff x\right)^{\frac{1}{q}}\diff t\\
        &\quad \leq \int_0^T \left(\int_\Omega \left(\delta(\theta_n)\right)^{\frac{q}{2-q}} \diff x \right)^{\frac{2-q}{2q}} \left( \int_{\Omega}\delta(\theta_n)|\tr (\B_n - \mathbb{I})|^2\diff x \right)^{\frac{1}{2}}\diff t\\
        &\quad\le  \|\delta(\theta_n)\|^{\frac12}_{L_t^{1} L_x^{\frac{q}{2-q}}} \left( \int_0^T\int_{\Omega}\delta(\theta_n)|\tr (\B_n - \mathbb{I})|^2\diff x \diff t\right)^{\frac{1}{2}} \le C(q).
    \end{split}
\end{equation}
In addition, by the very similar manipulation, we deduce that for any $q\in [1,\frac87)$
\begin{equation}\label{some_random_inequality33}
    \begin{split}
        &\int_0^T\int_\Omega |\delta(\theta_n) \F_n \F_n^T \F_n|^q +|\delta(\theta_n)\F_n |^q \diff x \diff t \leq C\int_0^T \int_\Omega 1+  (\delta(\theta_n))^{\frac{q}{4}}\left(\delta(\theta_n) |\F_n|^4\right)^{\frac{3q}{4}}\diff x \diff t\\
        &\quad \leq C + C\left(\int_0^T \int_\Omega (\delta(\theta_n))^{\frac{q}{4-3q}}\right)^{\frac{4}{4-3q}}  \left(\int_0^T \int_\Omega \delta(\theta_n) |\F_n|^4\diff x \diff t\right)^{\frac{3q}{4}}\le C(q),
    \end{split}
\end{equation}
where we used the estimates \eqref{bounds_on_F_n_Linfty} and \eqref{bounds_for_delta_theta_n}. Note that condition $q<\frac87$ is equivalent to condition $\frac{q}{4-3q}<2$, which is required in \eqref{bounds_for_delta_theta_n}.
For the entropy estimates,  we  multiply \eqref{equation_for_ln_det_B_n} by $q|\ln\det\B_n|^{q-2} \ln\det\B_n$,  integrate over $\Omega$ and use the H\"{o}lder inequality to get 
$$
\begin{aligned}
\frac{\diff}{\diff t} \int_{\Omega} |\ln\det\B_n|^{q}\diff x &= -q\int_{\Omega}\delta(\theta_n) \tr (\B_n - \mathbb{I}) |\ln\det\B_n|^{q-2} \ln\det\B_n\diff x\\
&\le \left(\int_{\Omega}\left|\delta(\theta_n) \tr (\B_n - \mathbb{I})\right|^q \diff x\right)^{\frac{1}{q}} \|\ln\det\B_n\|^{q-1}_{L^q_x}.
\end{aligned}
$$
The Gr\"{o}nwall inequality and the uniform estimate~\eqref{some_random_inequality32} then implies 
% Due to the H\"{o}lder inequality, \eqref{some_random_inequality3} and the bounds \eqref{u_n_and_theta_n_Linfty_bounds}, \eqref{bounds_on_trace_B_n} the above implies
%    \begin{equation}
%        \begin{split}
%            \|\ln\det\B_n&\|_{L^\infty_t L^{\frac{5-3\varepsilon}{4-2\varepsilon}}_x} \lesssim \|\ln\det\B_0\|_{\infty} + %\|\ln\det\B_n\|_{1}\left(\int_0^T\int_\Omega |\delta(\theta_n)\mathrm{tr}(\B_n - \mathbb{I})|^{\frac{4-2\varepsilon}{3-\varepsilon}}\diff x\diff %\tau\right)^{\frac{3-\varepsilon}{4-2\varepsilon}}\\
%            &\leq   \|\ln\det\B_0\|_{\infty} + (\|f(\B_n)\|_{1} + \|\mathrm{tr}(\B_n) - 2\|_{1})\left(\int_0^T\int_\Omega |\delta(\theta_n)\mathrm{tr}(\B_n - %\mathbb{I})|^{\frac{4-2\varepsilon}{3-\varepsilon}}\diff x\diff \tau\right)^{\frac{3-\varepsilon}{4-2\varepsilon}}\\
%            &\leq C(\varepsilon, \|u_0\|_{ L^2_x}, \|\theta_0\|_{ L^1_x}, \|\ln(\theta_0)\|_{L^1_x}, \|f(\B_0)\|_{ L^1_x}, \|\F_0\|_{L^2_x}, %\|\ln\det\B_0\|_{L^\infty_x}).
%        \end{split}
%    \end{equation}
%    In fact, with the obtained bound we may apply the same reasoning and multiply the equation \eqref{equation_for_ln_det_B_n} by $|\ln\det\B_n|^{\frac{(5-3\varepsilon)(1-\varepsilon)}{(4 - 2\varepsilon)^2} - 1}\ln\det\B_n$. In the end, by a recursive argument
%    $$
%    \|\ln\det\B_n\|_{L^\infty_t L^\lambda_x}\leq C(\lambda, \|u_0\|_{ L^2_x}, \|\theta_0\|_{ L^1_x}, \|\ln(\theta_0)\|_{L^1_x}, \|f(\B_0)\|_{ L^1_x}, %\|\F_0\|_{L^2_x}, \|\ln\det\B_0\|_{L^\infty_x}),
%    $$
%    for any $\lambda < \frac{4-2\varepsilon}{3-\varepsilon}$, or easier
\begin{align}\label{bound_on_ln_det_B_n}
        \|\ln\det\B_n\|_{L^\infty_t L^{q}_x}\leq C(q, \|\ln\det\B_0^n\|_{L^q_x}),
\end{align}
for any $q\in [1,2)$. Thus, with the use of \eqref{bounds_on_trace_B_n}, \eqref{bound_on_ln_det_B_n} and due to bounds on initial conditions $\ln \det \F_n$ in~\eqref{conv:init}, we obtain
\begin{align}\label{final_bounds_f_of_B_n}
        \|f(\B_n)\|_{L^{2-\varepsilon}_{t,x}} \leq C(\varepsilon) \textrm{ for any } \varepsilon\in (0,1).
\end{align}

We finish this part by deriving the uniform estimates for all terms that appear in the entropy identity~\eqref{eq:entropy_equality_for_g_theta}. Recall the definition of the entropy $\eta_n$ in~\eqref{eq:entropyn} 
\begin{equation*}%\label{eq:entropyn}
\eta_n= \ln \theta_n -g'(\theta_n) f(\B_n).
\end{equation*}
It follows from the uniform bound~\eqref{bound_on_B_n_L2} and Poincar\'{e}--Wirtinger inequality that
%
%It follows from~\eqref{eq:entropy_equality_for_g_theta} after integration over $(0,T)\times \Omega$ that
%    \begin{align}\label{bounds_on_grad_log_theta_n}
%    \int_0^T\int_\Omega |\nabla_x\ln(\theta_n)|^2\diff x\diff t = \int_0^T\int_\Omega \frac{|\nabla_x\theta_n|^2}{\theta_n^2}\diff x\diff t \leq  C(\|u_0\|_{ L^2_x}, \|\theta_0\|_{L^1_x}, \|\ln(\theta_0)\|_{ L^1_x}, \|f(\B_0)\|_{ L^1_x}).
%    \end{align}
%    and
%    \begin{align*}
%        \int_\Omega \ln(\theta_n)\diff x \geq \int_\Omega \ln(\theta_0)\diff x - C(\|g'(\theta_0)f(\B_0)\|_1) \geq -C(\|\ln(\theta_0)\|_{L^1_x}, \|f(\B_0)\|_{L^1_x}).
%    \end{align*}
%    By \eqref{u_n_and_theta_n_Linfty_bounds} we may estimate
%    \begin{align*}
%        \int_{\{x: \theta_n(t,x) > 1\}}\ln(\theta_n(t))\diff x\leq \int_{\{x:\theta_n(t) > 1\}}\theta_n(t) - 1\diff x \leq C(\|u_0\|_{L^2_x}, \|\theta_0\|_{L^1_x}, \|f(\B_0)\|_{L^1_x}).
%    \end{align*}
%    Hence,
%    \begin{align}\label{bounds_ln_theta_Linfty_L1}
%        \|\ln(\theta_n)\|_{L^\infty_t L^1_x} \leq C(\|u_0\|_{L^2_x}, \|\theta_0\|_{L^1_x}, \|f(\B_0)\|_{L^1_x}, \|\ln(\theta_0)\|_{L^1_x}),
%    \end{align}
%    
%and we may utilize Poincar\'{e}--Wirtinger inequality as well as \eqref{bounds_on_grad_log_theta_n} and \eqref{bounds_ln_theta_Linfty_L1} to deduce
\begin{align*}%\label{bounds_log_theta_n}
    \|\ln \theta_n\|_{L^{2}_t W^{1,2}_x} \leq  C
\end{align*}
and it follows from the uniform estimate \eqref{final_bounds_f_of_B_n}, the above inequality and the assumption~\eqref{function:g'_bounds_linfty} that
\begin{align}\label{bounds_entropy_n}
        \|\eta_n\|_{L^{2-\varepsilon}_{t,x}}\leq  C \qquad \textrm{for all }\varepsilon \in (0,1).
\end{align}
Using this inequality, the uniform estimate \eqref{bounds_on_u_n_nu_L4} for $\bv_n$ and the classical H\"{o}lder inequality, we see
\begin{align}\label{bounds_entropy_n2}
        \|\eta_n \, \bv_n\|_{L^{\frac{4}{3}-\varepsilon}_{t,x}} \leq  C \quad \textrm{ for all } \varepsilon \in \left(0,\frac13\right].
\end{align}
Next, defining the flux 
\begin{equation}\label{flux_n}
    \bq_n = \eta_n\,\bv_n - \kappa(\theta_n)\nabla_x(\ln \theta_n),
\end{equation}
it follow from~\eqref{bounds_entropy_n2} and \eqref{bound_on_B_n_L2} that 
%Then, by \eqref{bounds_log_theta_n}, \eqref{bounds_on_u_n_nu_L4}, and H\"{o}lder's inequality
%\begin{align*}
%        \|\ln(\theta_n)\,u_n\|_{L^{\frac{4}{3}}_{t,x}} \leq  C(\|u_0\|_{L^2_x}, \|\theta_0\|_{L^1_x}, \|f(\B_0)\|_{L^1_x}, \|\ln(\theta_0)\|_{L^1_x}).
%\end{align*}
%Similarly, by \eqref{bounds_on_trace_B_n}, \eqref{bounds_on_u_n_nu_L4}, and H\"{o}lder's inequality
%\begin{align*}
%        \|\mathrm{tr}(\B_n)\,u_n\|_{L^{\frac{4}{3}}_{t,x}}\leq C(\|u_0\|_{L^2_x}, \|\theta_0\|_{L^1_x}, \|f(\B_0)\|_{L^1_x}, \|\ln(\theta_0)\|_{L^1_x}),
%\end{align*}
%and by \eqref{bound_on_ln_det_B_n}, \eqref{bounds_on_u_n_nu_L2_Lalmostinfty}, H\"{o}lder's inequality
%\begin{align*}
%        \|\ln\det\B_n\,u_n\|_{L^2_t L^{\frac{5}{4}}_x} \leq C(\|u_0\|_{ L^2_x}, \|\theta_0\|_{ L^1_x}, \|\ln(\theta_0)\|_{L^1_x}, \|f(\B_0)\|_{ L^1_x}, \|\F_0\|_{L^2_x}, \|\ln\det\B_0\|_{L^\infty_x}).
%\end{align*}
%Thus
%    \begin{align}\label{bound_on_u_n_mulitplied_entropy}
%        \|\eta_n\,u_n\|_{L^{\frac{5}{4}}_{t,x}} \leq C(\|u_0\|_{ L^2_x}, \|\theta_0\|_{ L^1_x}, \|\ln(\theta_0)\|_{L^1_x}, \|f(\B_0)\|_{ L^1_x}, \|\F_0\|_{L^2_x}, \|\ln\det\B_0\|_{L^\infty_x}).
%    \end{align}
%    In particular, the bound above, together with \eqref{bounds_on_grad_log_theta_n} implies
\begin{align}\label{flux_n_estimate}
        \|\bq_n\|_{L^{\frac{4}{3}-\varepsilon}_{t,x}} \leq C \qquad \textrm{ for all }\varepsilon\in \left(0, \frac13 \right].
\end{align}

\subsection{Convergence results based on the uniform estimates}
In this section we use the reflexivity of underlying spaces and the uniform bounds to get the weak convergence results for the sequence of solutions. 
First, we focus on the convergence results for $\bv_n$. Using  \eqref{u_n_and_theta_n_Linfty_bounds}, \eqref{bounds_on_gradient_u_n_nu_L2}, \eqref{bound_on_partial_t_u_n} and the Aubin--Lions lemma (see Lemma~\ref{aubin-lions}) we can find a subsequence that we do not relabel and $\bv$ such that
\begin{equation}
\begin{aligned}\label{strong_convergence_of_u_n_C_Lp}
            \bv_n &\overset{*}{\rightharpoonup} \bv &&\text{ weakly* in }L^\infty_t L^2_x,\\
            \bv_n &\rightharpoonup \bv &&\text{ weakly in }L^2_t W^{1,2}_{0,\DIV},\\
            \bv_n &\rightharpoonup \bv &&\text{ weakly in }L^4_{t,x},\\
            \bv_n &\rightarrow \bv &&\text{ strongly in }L^p_{t,x} \textrm{ for all } p\in [1,4),\\
            \partial_t \bv_n  &\rightharpoonup \partial_t \bv &&\text{ weakly in } L^2_t W^{-1,2}_{0,\DIV},
\end{aligned}
\end{equation}
which directly implies~\eqref{cnv_v}$_1$--\eqref{cnv_v}$_3$. Similarly, due to the uniform bounds~\eqref{bounds_on_F_n_Linfty}, we can find a subsequence that we do not relabel and~$\F$ such that
\begin{equation}\label{conv_Fn}
\begin{aligned}
            \F_n &\overset{*}{\rightharpoonup} \F &&\text{ weakly* in }L^\infty_t L^2_x,\\
            \F_n &\rightharpoonup \F &&\text{ weakly in }L^4_{t,x},
\end{aligned}
\end{equation}  
that gives ~\eqref{cnv_F}$_1$--\eqref{cnv_F}$_2$. Next, for the temperature, we use  \eqref{final_bounds_on_theta_n2}--\eqref{final_bounds_on_grad_theta_n} and wee that for a subsequence 
\begin{equation}\label{conv_thetan}
\begin{aligned}
\theta_n &\rightarrow \theta &&\text{ weakly in }L_t^{q} L_x^{p} \textrm{ for all } p\in [1,\infty) \textrm{ and } q\in \left[1,\frac{p}{p-1}\right),\\
\theta_n &\rightarrow \theta &&\text{ weakly in }L_t^{q}W^{1,q}_x  \textrm{ for all } q\in \left[1,\frac{4}{3}\right)
\end{aligned}
\end{equation}
that is \eqref{cnv_theta}. For nonlinear term, where we cannot a priori identify the limit due to the missing compactness of the temperature $\theta_n$ and the extra stress tensor $\F_n$, we simply use the symbol $\overline{a}$ for a weak limit of the sequence $a_n$. Therefore, using \eqref{strong_convergence_of_u_n_C_Lp}--\eqref{conv_thetan}, the assumptions \eqref{bouds_nu}, \eqref{bounds_delta2}, \eqref{function:g_bounds_linfty} and the uniform estimates \eqref{some_random_inequality3}--\eqref{some_random_inequality33} and the fact that $\B_n = \F_n \, \F_n^T$, we deduce  
\begin{equation}\label{nonlinear_n}
\begin{aligned}
            %\F_n &\overset{*}{\rightharpoonup} \F &&\text{ weakly* in }L^\infty_t L^2_x,\\
            %\F_n &\rightharpoonup \F &&\text{ weakly in }L^4_{t,x},\\
            \nabla_x \bv_n\,\F_n &\rightharpoonup \overline{\nabla_x \bv \,\F} &&\text{ weakly in }L^{\frac{4}{3}}_{t,x},\\
            \F_n\,\F_n^T &\rightharpoonup \overline{\F\,\F^T} &&\text{ weakly in }L^{2}_{t,x},\\
            |\F_n|^2 &\rightharpoonup \overline{|\F|^2} &&\text{ weakly in }L^{2}_{t,x},\\
            \F_n\,\F_n^T\,\F_n &\rightharpoonup \overline{\F\,\F^T\,\F} &&\text{ weakly in }L^{\frac{4}{3}}_{t,x},\\
            %|\F_n\,\F_n^T|^2 &\overset{*}{\rightharpoonup} \overline{|\F\,\F^T|^2} &&\text{ weakly* in }\mathcal{M}([0, T]\times\overline{\Omega}),\\
            %\nabla_x u_n\,\F_n\,\F_n^T &\overset{*}{\rightharpoonup} \overline{\nabla_x u\,\F\,\F^T} &&\text{ weakly* in }\mathcal{M}([0, T]\times\overline{\Omega}),\\
            %|D u_n|^2 &\overset{*}{\rightharpoonup} \overline{|D u|^2} &&\text{ weakly* in }\mathcal{M}([0, T]\times\overline{\Omega}),\\
            g(\theta_n)\F_n\,\F_n^T &\rightharpoonup \overline{g(\theta)\F\,\F^T} &&\text{ weakly in }L^{2}_{t,x},\\
            %\delta(\theta_n)|\F_n|^2 &\rightharpoonup \overline{\delta(\theta)|\F|^2} &&\text{ weakly in }L^{\frac{4}{3}-\varepsilon}_{t,x},\\
            \delta(\theta_n)\F_n\,\F_n^T\,\F_n &\rightharpoonup \overline{\delta(\theta)\F\,\F^T\,\F} &&\text{ weakly in }L^{q}_{t,x} \textrm{ for all }q\in \left[1, \frac87 \right),\\
            \delta(\theta_n)\F_n &\rightharpoonup \overline{\delta(\theta)\F} &&\text{ weakly in }L^{q}_{t,x} \textrm{ for all }q\in \left[1, \frac43 \right),\\
            %\delta(\theta_n)|\F_n\,\F_n^T|^2 &\overset{*}{\rightharpoonup} \overline{\delta(\theta)|\F\,\F^T|^2} &&\text{ weakly* in }\mathcal{M}([0, T]\times\overline{\Omega}),\\
            \nu(\theta_n)\D \bv_n &\rightharpoonup \overline{\nu(\theta)\D \bv} &&\text{ weakly in }L^{2}_{t,x}.
            %\theta_n &\rightarrow \theta &&\text{ strongly in }L^{2-\varepsilon}_{t,x},\\
            %\delta(\theta_n) &\rightarrow \delta(\theta) &&\text{ strongly in }L^{2-\varepsilon}_{t,x},\\
            %g(\theta_n) &\overset{*}{\rightharpoonup} g(\theta) &&\text{ weakly* in }L^{\infty}_{t,x},\\
            %\nu(\theta_n) &\overset{*}{\rightharpoonup} \nu(\theta) &&\text{ weakly* in }L^{\infty}_{t,x},
\end{aligned}
\end{equation}
Hence, we may let $n\to \infty$ in \eqref{main_sys_for_g_theta}$_1$--\eqref{main_sys_for_g_theta}$_2$ and using the convergence results for initial data \eqref{conv:init} we can obtain by very classical procedure that
\begin{equation}
\begin{aligned}\label{almost_weak_formulation_u}
        &\int_0^T\int_\Omega -\bv \cdot\p_t\bphi - \bv\otimes \bv : \nabla_x\bphi + \overline{\nu(\theta)\D\bv} :\nabla_x\bphi + \overline{g(\theta)\F\,\F^T}:\nabla_x\bphi\diff x\diff t \\
        &\qquad = \int_\Omega \bv_0(x)\cdot\bphi(0, x)\diff x
\end{aligned}
\end{equation}
for any $\vect{\varphi}\in \mathcal{C}^1_c([0, T)\times\Omega; \R^2)$ with $\DIV_x\bphi = 0$,
\begin{equation}
\begin{aligned}\label{almost_weak_formulation_F}
        &\int_0^T\int_\Omega -\F:\p_t\bG - \F\otimes \bv \because \nabla_x \bG - \overline{\nabla_x \bv\,\F} : \bG + \frac{1}{2}\left(\overline{\delta(\theta) \F\,\F^T\,\F} - \overline{\delta(\theta)\F}\right): \bG\diff x\diff t \\
        &\qquad = \int_\Omega \F_0(x) : \bG(0, x)\diff x
    \end{aligned}
\end{equation}
for any $\bG\in \mathcal{C}^1_c([0, T)\times\Omega; \R^{2\times 2})$. Note that \eqref{almost_weak_formulation_u}--\eqref{almost_weak_formulation_F} implies \eqref{weak_formulation_u_g_theta}--
\eqref{weak_formulation_F_g_theta} provided that we show the point-wise convergence results
\begin{equation}\label{pointn}
\begin{aligned}
\theta_n &\to \theta &&\textrm{almost everywhere in } (0,T)\times \Omega,\\
\F_n &\to \F &&\textrm{almost everywhere in } (0,T)\times \Omega.
\end{aligned}
\end{equation}
In the remaining parts of this section we focus mainly on the proof of \eqref{pointn}.
%and to show it we need to have proper description of the weak limits of nonlinear terms. 
%
%
%    \begin{align}\label{almost_weak_formulation_u}
%        \int_0^T\int_\Omega -u\cdot\p_t\phi - u\otimes u : \nabla_x\phi + \nu(\theta)Du:\nabla_x\phi + g(\theta)\overline{\F\,\F^T}:\nabla_x\phi\diff x\diff t = \int_\Omega u_0(x)\cdot\phi(0, x)\diff x,
%    \end{align}
%    for any $\phi\in C^1_c([0, T)\times\Omega;\R^2)$ with $\DIV_x\phi = 0$,
%    \begin{align}\label{almost_weak_formulation_F}
%        \int_0^T\int_\Omega -\F : \p_t\phi - \F\otimes u \because \nabla_x\phi - \overline{\nabla_x u\,\F} : u + \frac{1}{2}\delta(\theta)(\overline{\F\,\F^T\,\F} - \F):\phi\diff x\diff t = \int_\Omega \F_0 : \phi(0, x)\diff x,
%    \end{align}
%    for any $\phi\in C^1_c([0, T)\times\Omega;\R^{2\times 2})$.
%
%
%Due to Banach--Alaoglu's theorem, Lebegue's dominated convergence theorem, Aubin-Lions' lemma \ref{aubin-lions} and \eqref{u_n_and_theta_n_Linfty_bounds}, \eqref{bound_on_partial_t_u_n}, \eqref{bounds_on_symmetric_gradient_u_n_nu_L2} - \eqref{bounds_on_u_n_nu_L2_Lalmostinfty}, \eqref{bounds_on_F_n_Linfty}, \eqref{bounds_on_F_n_L4}, \eqref{bounds_on_F_n_with_delta}, \eqref{final_bounds_on_theta_n}, \eqref{final_bounds_on_grad_theta_n}, \eqref{almost_everywhere_convergence_theta_n}

\subsection{Compactness of the temperature \texorpdfstring{$\theta_n$}{t}} For the proof of the compactness of the temperature, we use the new variant of the entropy method adapted to the setting of the paper. To do this, we apply the div--curl lemma on the properly chosen quantities that lead to the desired compactness. 
We recall the definition of entropy $\eta_n$ in \eqref{eq:entropyn} and the corresponding estimate \eqref{bounds_entropy_n}--\eqref{bounds_entropy_n2} and also the definition of the flux $\bq_n$ in~\eqref{flux_n} and the related estimate~\eqref{flux_n_estimate}. Therefore, we can extract a subsequence and $\overline{\eta}$ and $\overline{\bq}$ such that 
\begin{equation}
\begin{aligned}
        \eta_n &\rightharpoonup \overline{\eta} &&\text{ weakly in }L^{q}_{t,x} \textrm{ for all } q\in [1,2),\\
        \bq_n &\rightharpoonup \overline{\bq} &&\text{ weakly in }L^{q}_{t,x} \textrm{ for all } q\in \left[0, \frac43 \right).
        %\{\DIV_{t,x}(\eta_n, q_n)\}_{n\in\N}&\text{ compact in } (W^{1,3}_0((0,T)\times\Omega))^*.
\end{aligned}\label{divcurl1}
\end{equation}
Furthermore, it follows from~\eqref{eq:entropy_equality_for_g_theta} and that the time-space vector $(\eta_n,\bq_n)$ and its time-space divergence fulfill  
\begin{equation}\label{divcurl2}
\begin{aligned}
        \|\DIV_{t,x}(\eta_n, \bq_n)\|_{L^1_{t,x}} &= \|\p_t\eta_n + \DIV_x \bq_n\|_{L^1_{t,x}} \\
        &= \left\| \frac{\kappa(\theta_n)|\nabla_x\theta_n|^2}{\theta^2_n} + \frac{\nu(\theta_n)|\D \bv_n|^2}{\theta_n} + \frac{\delta(\theta_n)|\B_n - \mathbb{I}|^2}{\theta_n}\right\|_{L^1_{t,x}}\leq C,
    \end{aligned}
\end{equation}
where for the last inequality we used~\eqref{bound_on_B_n_L2}. Consequently, using the Sobolev embedding we see that 
\begin{align}\label{precom1}
        \{\DIV_{t,x}(\eta_n, \bq_n)\}_{n=1}^{\infty} \qquad \text{ is pre-compact in } (W^{1,4}_0((0,T)\times\Omega))^*.
\end{align}

In a very similar way, we use the estimates \eqref{final_bounds_f_of_B_n} and~\eqref{strong_convergence_of_u_n_C_Lp} to get 
\begin{equation}\label{divcurl3}
\begin{aligned}
        f(\B_n) &\rightharpoonup \overline{f(\B)} &&\text{ weakly in }L^{q}_{t,x} \textrm{ for all } q\in [1,2),\\
        f(\B_n)\, \bv_n &\rightharpoonup \overline{f(\B)} \, \bv &&\text{ weakly in }L^{q}_{t,x} \textrm{for all }q\in \left[1,\frac43 \right).\
        %\{\DIV_{t,x}(f(\B_n), f(\B_n)\,u_n)\}_{n\in\N}&\text{ compact in }(W^{1,3}_0((0,T)\times\Omega))^*.
\end{aligned}
\end{equation}
Next, we deduce from~\eqref{eq:equation_for_f_of_B_for_g_theta} that 
\begin{equation}\label{divcurl4}
\begin{aligned}
        \|\DIV_{t,x}(f(\B_n), f(\B_n)\, \bv_n)\|_{L^1_{t,x}} &= \|\p_tf(\B_n) + \DIV_x \left(f(\B_n)\, \bv_n\right)\|_{L^1_{t,x}} \\
        &= \left\|2(\B_n - \mathbb{I}) : \D \bv_n-\delta(\theta)|\B_n - \mathbb{I}|^2\right\|_{L^1_{t,x}}\leq C,
    \end{aligned}
\end{equation}
where the last inequality follows from the uniform estimates~\eqref{bounds_on_F_n_with_delta} and~\eqref{bounds_on_gradient_u_n_nu_L2}. Therefore, using again the Sobolev embedding we get that 
\begin{align}\label{precom2}
        \{\DIV_{t,x}(f(\B_n), f(\B_n)\, \bv_n)\}_{n=1}^{\infty} \qquad \text{ is pre-compact in } (W^{1,4}_0((0,T)\times\Omega))^*.
\end{align}
%
%    
%    \begin{align*}
%        f(\B_n) &\rightharpoonup \overline{f(\B)}\,\,\quad\text{ weakly in }L^{\frac{5}{4}}_{t,x},\\
%        f(\B_n)\, u_n &\rightharpoonup \overline{f(\B)\,u}\quad\text{ weakly in }L^{\frac{5}{4}}_{t,x},\\
%        \{\DIV_{t,x}(f(\B_n), f(\B_n)\,u_n)\}_{n\in\N}&\text{ compact in }(W^{1,3}_0((0,T)\times\Omega))^*.
%    \end{align*}
%    
%    
We have constructed two weakly convergent time-space vector fields, namely $(f(\B_n), f(\B_n)\bv_n)_{n=1}^{\infty}$ and $(\eta_n, \bq_n)_{n=1}^{\infty}$, whose space-time divergence $\DIV_{t,x}$ is precompact in $W^{-1,\frac43}_{0,t,x}$. Next, we focus on finding a weakly convergent vector field whose time-space $\CURL_{t,x}$ is precompact in $W^{-1,\frac43}_{0,t,x}$. 
%
%    And due to the boundedness of $g'(w)$, there exists $\overline{g'(\theta)f(\B)}\in L^{\frac{5}{4}}_{t,x}$, such that
%    \begin{align*}
%        g'(\theta_n)f(\B_n) \rightharpoonup \overline{g'(\theta)f(\B)} \text{ weakly in }L^{\frac{5}{4}}_{t,x}.
%    \end{align*}
To do so, we consider the sequence $\{\theta_n^{\frac{1}{3}}\}_{n=1}^{\infty}$ and by \eqref{final_bounds_on_theta_n} we know that there exists $\overline{\theta^{\frac{1}{3}}}$ such that 
\begin{align}\label{theta_n_c}
        \theta_n^{\frac{1}{3}}&\rightharpoonup \overline{\theta^{\frac{1}{3}}} &&\text{ weakly in }L^{q}_{t,x} \textrm{ for all }q\in[1,6).
\end{align}
Next, it follows from~\eqref{bound_grad_theta_n_with_lambda} that  
%\begin{align*}
%        \|\theta_n^{\frac{1}{3}}\|_{L^{5+\varepsilon}_{t,x}}\leq C(\varepsilon, \|u_0\|_{ L^2_x}, \|\theta_0\|_{ L^1_x}, \|\ln(\theta_0)\|_{L^1_x}, \|f(\B_0)\|_{ L^1_x}, \|\F_0\|_{L^2_x}),\quad\varepsilon\in(0,1),
%    \end{align*}
%and by 
\begin{align*}
        \|\nabla_x \theta_n^{\frac{1}{3}}\|_{L^2_{t,x}} = \frac{1}{3}\left\|\frac{\nabla_x\theta_n}{\theta_n^{\frac23}}\right\|_{L^2_{t,x}} \leq C.
\end{align*}
Therefore, for the vector field given as $(\theta_n^{\frac{1}{3}},0,0)$, we get
\begin{align*}
        \|\CURL_{t,x}\,(\theta_n^{\frac{1}{3}}, 0, 0)\|_{L^2_{t,x}} \le C \|\nabla_x \theta_n^{\frac13}\|_{L^2_{t,x}} \leq C.
\end{align*}
Consequently, by the Sobolev embedding, we have
\begin{align}\label{precom3}
        \{\CURL_{t,x}\,(\theta_n^{\frac{1}{3}}, 0, 0)\}_{n=1}^{\infty} \qquad \text{ is pre-compact in } (W^{1,2}_0((0,T)\times\Omega))^*.
\end{align} 

Hence, we are able to apply the div--curl lemma (see Lemma~\ref{div-curl}) on the sequences 
$$
\{(\theta_n^{\frac{1}{3}}, 0, 0)\}_{n=1}^{\infty}, \quad \{(f(\B_n), f(\B_n)\, \bv_n)\}_{n=1}^{\infty} \quad \textrm{ and } \quad \{(\eta_n, \bq_n)\}_{n=1}^{\infty}
$$ 
and thanks to \eqref{divcurl1}, \eqref{precom1}, \eqref{divcurl3}, \eqref{precom2}, \eqref{theta_n_c} and \eqref{precom3} we observe that 
\begin{align*}
(\theta_n^{\frac{1}{3}}, 0, 0) \cdot (f(\B_n), f(\B_n)\, \bv_n) &\rightharpoonup (\overline{\theta^{\frac{1}{3}}}, 0, 0) \cdot (\overline{f(\B)}, \overline{f(\B)}\, \bv) &&\textrm{weakly in } L^{1}_{t,x},\\
(\theta_n^{\frac{1}{3}}, 0, 0) \cdot (\eta_n, \bq_n) &\rightharpoonup (\overline{\theta^{\frac{1}{3}}}, 0, 0) \cdot (\overline{\eta}, \overline{\bq}) &&\text{ weakly in }L^{1}_{t,x}.
\end{align*}
This in particular gives that 
%\begin{align*}
%        \theta_n^{\frac{1}{3}}&\rightharpoonup \overline{\theta^{\frac{1}{3}}} \quad\text{ weakly in }L^{5+\varepsilon}_{t,x},\\
%        \{\mathrm{curl}_{t,x}\,(\theta_n^{\frac{1}{3}},0,0)\}_{n\in\N}&\text{ compact in }(W^{1,2}_0((0, T)\times\Omega))^*.
%\end{align*}
%    Hence, by the div-curl lemma \ref{div-curl}
\begin{align}
        \eta_n\,\theta_n^{\frac{1}{3}} &\rightharpoonup \overline{\eta}\,\overline{\theta^{\frac{1}{3}}} &&\text{ weakly in }L^{1}_{t,x},\label{after_div_curl_conv_eta_multiplied_theta}\\
        f(\B_n)\,\theta_n^{\frac{1}{3}}&\rightharpoonup\overline{f(\B)}\,\overline{\theta^{\frac{1}{3}}}&&\text{ weakly in }L^1_{t,x}\label{after_div_curl_conv_f_multipled_theta}.
\end{align}
Since the function $f$ is nonnegative and $g'$ is nonincreasing (due to the fact that $g$ is concave), we have that for any $w\in L^q_{t,x}$ with $q>5$
$$
    0\leq f(\B_n)(-g'(\theta_n) + g'(w))(\theta_n^{\frac{1}{3}} - w^\frac{1}{3}) \qquad \textrm{ almost everywhere in } (0,T)\times \Omega.
$$
In particular, for $w := \left(\overline{\theta^{\frac{1}{3}}}\right)^3$ we obtain
\begin{align*}
      0\leq f(\B_n)\left(-g'(\theta_n) + g'\left(\left(\overline{\theta^{\frac{1}{3}}}\right)^3\right)\right)\left(\theta_n^{\frac{1}{3}} - \overline{\theta^{\frac{1}{3}}}\right),
\end{align*}
which by simple algebraic manipulation leads to 
\begin{align}\label{stt1}
        -f(\B_n)g'(\theta_n)\overline{\theta^{\frac{1}{3}}} \leq -f(\B_n)g'(\theta_n)\theta_n^{\frac{1}{3}} + f(\B_n) g'\left(\left(\overline{\theta^{\frac{1}{3}}}\right)^3\right)\left(\theta_n^{\frac{1}{3}} - \overline{\theta^{\frac{1}{3}}}\right)
\end{align}
almost everywhere in $(0,T)\times \Omega$. Next, since $g'$ is bounded, see~\eqref{function:g'_bounds_linfty}, we can also deduce  
\begin{align}
        g'(\theta_n)f(\B_n) &\rightharpoonup \overline{g'(\theta)f(\B)} &&\text{ weakly in }L^{q}_{t,x} \textrm{ for all } q\in [1,2),\label{AAA1}\\
        f(\B_n)g'(\theta_n)\theta_n^{\frac{1}{3}} &\rightharpoonup \overline{f(\B)g'(\theta)\theta^{\frac{1}{3}}} &&\text{ weakly in }L^1_{t,x}.\label{AAA2}
\end{align}
Thanks to \eqref{after_div_curl_conv_f_multipled_theta}, and due to the fact that $g'\left(\overline{\theta^{\frac{1}{3}}}\right)\in L^{\infty}_{t,x}$, we have
$$
\lim_{n\to \infty}\int_{0}^{T} \int_{\Omega} f(\B_n) g'\left(\left(\overline{\theta^{\frac{1}{3}}}\right)^3\right)\left(\theta_n^{\frac{1}{3}} - \overline{\theta^{\frac{1}{3}}}\right)\diff x \diff t =0.
$$
Letting $n\to \infty$, using the above identity and also \eqref{AAA1} we get
\begin{align}\label{ineq:for_monotonicity_trick_f_of_B}
        \int_0^T\int_\Omega -\overline{g'(\theta)f(\B)}\,\overline{\theta^{\frac{1}{3}}}\diff x\diff t\leq \int_0^T\int_\Omega -\overline{g'(\theta)f(\B)\theta^{\frac{1}{3}}}\diff x\diff t.
\end{align}
Denoting also $\overline{\ln \theta}$ as the weak limit
$$
\ln \theta_n \rightharpoonup \overline{\ln \theta}\qquad \textrm{ weakly in } L^2_{t,x},
$$
using \eqref{ineq:for_monotonicity_trick_f_of_B}, the definition of $\eta_n$, see~\eqref{eq:entropyn}, and  applying~\eqref{after_div_curl_conv_eta_multiplied_theta}, and \eqref{ineq:for_monotonicity_trick_f_of_B} we obtain
\begin{equation}\label{ineq:for_monotonicity_trick_ln_theta}
        \begin{aligned}
            \int_0^T\int_\Omega \overline{\ln \theta }\,\overline{\theta^{\frac{1}{3}}}\diff x\diff t &\overset{\eqref{eq:entropyn}}{=} \int_0^T\int_\Omega \overline{\eta}\,\overline{\theta^{\frac{1}{3}}}\diff x\diff t + \int_0^T\int_\Omega \overline{g'(\theta)f(\B)}\,\overline{\theta^{\frac{1}{3}}}\diff x\diff t\\
            &\overset{\eqref{ineq:for_monotonicity_trick_f_of_B}}\ge  \lim_{n\to+\infty}\int_0^T\int_\Omega\eta_n\,\theta_n^{\frac{1}{3}}\diff x\diff t + \int_0^T\int_\Omega \overline{g'(\theta)f(\B)\theta^{\frac{1}{3}}}\diff x\diff t\\
            &\overset{\eqref{eq:entropyn}}{=} \lim_{n\to \infty} \int_0^T\int_\Omega \ln \theta_n\theta_n^{\frac{1}{3}}\diff x\diff t - \lim_{n\to \infty}\int_0^T\int_\Omega g'(\theta_n)f(\B_n)\theta_n^{\frac{1}{3}}\diff x\diff t\\
             &\qquad {}\qquad + \int_0^T\int_\Omega \overline{g'(\theta)f(\B)\theta^{\frac{1}{3}}}\diff x\diff t\\
            &\; \geq \int_0^T\int_\Omega \overline{(\ln \theta )\theta^{\frac{1}{3}}}\diff x\diff t.
        \end{aligned}
\end{equation}
Now, notice that $w\mapsto \ln(w)$ and $w\mapsto w^{\frac{1}{3}}$ are increasing functions. Thus, for any nonnegative function  $w\in L^1((0, T)\times\Omega)$, fulfilling  $\ln w \in L^2_{t,x}$, there holds
\begin{align}\label{Minty3}
        \int_0^T\int_\Omega (\ln \theta_n - \ln w )(\theta_n^{\frac{1}{3}} - w^{\frac{1}{3}})\diff x\diff t \geq 0.
\end{align}
Due to \eqref{ineq:for_monotonicity_trick_ln_theta} we may let $n\to \infty$  and deduce
\begin{align}\label{ineq:monotonicity_trick_for_arbitrary_w}
        \int_0^T\int_\Omega (\overline{\ln \theta } - \ln w)(\overline{\theta^{\frac{1}{3}}} - w^{\frac{1}{3}})\diff x\diff t \geq 0.
\end{align}
Moving forward, we repeat the Minty method. Assume that $h\in L^{\infty}_{t,x}$ is arbitrary and $\lambda >0$ and set
$$
        w := e^{\overline{\ln\theta} - \lambda h}
$$
in~\eqref{ineq:monotonicity_trick_for_arbitrary_w}. The fact that such $w$ is admissible follows from the following. 
Since $h\in L^\infty_{t,x}$, then $e^{-\lambda h}\in L^\infty_{t,x}$. Moreover, since the exponential is a convex function, we can use the weak lower semicontinuity to deduce     
%$$
%        \phi \mapsto \int_0^T\int_\Omega e^\phi\diff x\diff t
%    $$
%    we can deduce
$$
        \|\exp(\overline{\ln\theta})\|_{L^1_{t,x}} \leq \liminf_{n\to\infty}\|\exp(\ln\theta_n)\|_{L^1_{t,x}} =\liminf_{n\to\infty}\|\theta_n\|_{L^1_{t,x}} \leq C.
$$
Thus, $w$ is an admissible function in \eqref{ineq:monotonicity_trick_for_arbitrary_w} and with this choice after division by $\lambda$ it follows that 
$$
        \int_0^T\int_\Omega \left(\overline{\theta^{\frac{1}{3}}} - e^{\frac{1}{3}\left(\overline{\ln\theta} - \lambda h\right)}\right) \, h\diff x\diff t\geq 0.
$$
Finally, letting  $\lambda\to 0_+$ and using the fact that $h$ is arbitrary we get
\begin{equation}\label{Minty7}
    \overline{e^{\frac{1}{3}\ln\theta}} = \overline{\theta^{\frac{1}{3}}} = e^{\frac{1}{3}\overline{\ln\theta}} \text{ a.e. in }(0, T)\times\Omega.
\end{equation}
We show that the above identity implies the strong convergence of the temperature claimed in~\eqref{pointn}$_1$.
To do so, we set $w := \theta_n$ in~\eqref{ineq:monotonicity_trick_for_arbitrary_w}. Then, using~\eqref{ineq:for_monotonicity_trick_ln_theta}, \eqref{Minty7} and the fact that the exponential is the increasing function, we deduce
$$
\begin{aligned}
&\lim_{n\to+\infty}\int_0^T\int_\Omega \left|(\overline{\ln \theta } - \ln \theta_n)\left(\overline{\theta^{\frac{1}{3}}} - \theta_n^{\frac{1}{3}}\right)\right|\diff x\diff t =\lim_{n\to+\infty}\int_0^T\int_\Omega \left|(\overline{\ln \theta } - \ln \theta_n)\left(e^{\frac13 \overline{\ln \theta}} - e^{\frac13 \ln \theta_n}\right)\right|\diff x\diff t \\ 
&\quad =\lim_{n\to+\infty}\int_0^T\int_\Omega (\overline{\ln \theta} - \ln \theta_n)(\overline{\theta^{\frac{1}{3}}} - \theta_n^{\frac{1}{3}})\diff x\diff t  = 0.
\end{aligned}
$$
Hence, up to the subsequence that we do not relabel,
\begin{align*}
       (\overline{\ln \theta} - \ln \theta_n )(e^{\frac{1}{3}\overline{\ln\theta}} - e^{\frac{1}{3}\ln\theta_n}) = (\overline{\ln \theta} - \ln \theta_n)(\overline{\theta^{\frac{1}{3}}} - \theta_n^{\frac{1}{3}}) \rightarrow 0 \text{ almost everywhere  in }(0, T)\times\Omega.
\end{align*}
Since $x \mapsto e^{\frac{1}{3}x}$ is a strictly increasing function, the above convergence result is possible only if 
\begin{align*}
        \ln \theta_n \rightarrow \overline{\ln\theta} \quad \text{ almost everywhere in }(0, T)\times\Omega,
\end{align*}
which in turn implies that
    \begin{align}\label{almost_everywhere_convergence_theta_n}
        \theta_n \rightarrow \theta := e^{\overline{\ln\theta}} \text{ almost everywhere in }(0, T)\times\Omega,
    \end{align}
that is \eqref{pointn}$_1$.

\subsection{Compactness of \texorpdfstring{$\F_n$}{F}}
In this part we show~\eqref{pointn}$_2$ and even more we prove the following convergence result
\begin{align}\label{strong_convergence_F_n}
    \F_n \rightarrow \F \text{ strongly in }L^2((0, T)\times\Omega).
\end{align}
We closely follow~\cite[Subsection 6.4]{bulicek2022onplanar} and~\cite[Subsection 3.4]{bulicek2024threeD}. Although the two-dimensional setting could indicate that the primary choice would be the use of~\cite{bulicek2022onplanar} the opposite is true. In fact, if the viscosity $\nu$ and the shear modules $g$ were independent of the temperature, the proof in~\cite{bulicek2022onplanar} could easily be adapted to our setting. However, the temperature dependence brings additional difficulties\footnote{The key problem is that even if we have the almost everywhere convergence of the temperature $\theta_n\to \theta$, we cannot prove the following identification
\begin{align*}
         \overline{g(\theta)(\F\,\F^T):\D \bv} &=  g(\theta)\overline{(\F\,\F^T):\D\bv},\\
         \overline{\nu(\theta)|\D \bv|^2} &= \nu(\theta)\overline{|\D \bv|^2}
\end{align*}
in sense of measures, which is true in case of constant $\nu$ and $g$.}
and therefore we must proceed differently, while clearly clearly indicating the corresponding differences.

In \cite{bulicek2022onplanar,bulicek2024threeD}, it is shown, that \eqref{strong_convergence_F_n} holds true provided that there exists $L\in L^2((0, T)\times\Omega)$ such that 
\begin{align}\label{inequality_to_prove_strong_L2}
        \int_0^T\int_\Omega -(\overline{|\F|^2} - |\F|^2)\p_t\phi - \bv(\overline{|\F|^2} - |\F|^2)\,\nabla_x\phi \diff x\diff t\leq \int_0^T\int_\Omega L(\overline{|\F|^2} - |\F|^2)\phi\diff x\diff t
\end{align}
holds for any nonnegative $\phi\in \mathcal{C}^1_0((-\infty, T)\times\Omega)$. Such a setting would be insufficient for the paper, but we can rather straightforwardly generalise the above result as follows. The convergence properties~\eqref{strong_convergence_F_n} follows from \eqref{inequality_to_prove_strong_L2} by using the renormalization procedure (see (cf. \cite[(6.79)]{bulicek2022onplanar} for details) and for that it is enough to require only that 
\begin{equation}
\begin{split}
&L\in L^1((0, T)\times \Omega),\\
&L(\overline{|\F|^2} - |\F|^2)\in L^1((0, T)\times\Omega),
\end{split}\label{need2}
\end{equation}
which is the setting we can obtain. 

We proceed here only formally, and for rigorous proof we refer to~\cite{bulicek2024threeD}. We also need to use here the concept of the biting limit and the Chacon biting lemma, see~\cite{BaMu89}. According to that, we know that for any sequence $a_n$ fulfilling
$$
\|a_n\|_{L^1_{t,x}}\le C,
$$
there exists $a\in L^1((0,T)\times \Omega)$ and there exists nondecreasing sequence of measurable sets $E_j\subset (0,T)\times \Omega$, such that
\begin{equation}\label{biting}
\begin{aligned}
a_n &\rightharpoonup a &&\textrm{weakly in } L^1 (E_j) \quad \textrm{ for all }j\in \mathbb{N},\\
|((0,T)\times \Omega) \setminus E_j| & \to 0 &&\textrm{as $j\to \infty$.}  
\end{aligned}
\end{equation}
We call $a$ the biting limit. It is clear that in case the classical weak limit exits, the biting and the weak limit must coincide in the space $L^1$ and therefore in what follows we primarily work with the biting limits.  

Thanks to  Egorov's theorem and \eqref{almost_everywhere_convergence_theta_n} we may identify the weak limits in \eqref{nonlinear_n} and to observe 
\begin{equation}\label{nonlinear_f}
\begin{aligned}
        \overline{g(\theta)\F\,\F^T} &= g(\theta)\overline{\F\,\F^T}= g(\theta) \B,\\
        \overline{\delta(\theta)|\F|^2} &= \delta(\theta)\overline{|\F|^2},\\
        \overline{\delta(\theta)\F\,\F^T\,\F} &= \delta(\theta)\overline{\F\,\F^T\,\F},\\
        \overline{\nu(\theta)\D \bv_n} &= \nu(\theta) \D \bv
\end{aligned}
\end{equation}
almost everywhere in $(0, T)\times\Omega$. In addition, due to the uniform estimates \eqref{bounds_on_F_n_Linfty}--\eqref{bounds_on_F_n_with_delta} and also \eqref{bounds_on_symetric_gradient_u_n_L2}, by using the Egorov theorem and the point-wise convergence of the temperature $\theta_n$ in~\eqref{almost_everywhere_convergence_theta_n}, we may conclude that the biting limits fulfills the following
\begin{equation}\label{biting2}
\begin{aligned}
         \overline{\delta(\theta)|\F\,\F^T|^2} &=  \delta(\theta)\overline{|\F\,\F^T|^2},\\
         \overline{\nu(\theta)|\D \bv|^2} &= \nu(\theta)\overline{|\D \bv|^2},\\
         \overline{g(\theta) \F\, \F^T :\D \bv} &= g(\theta)\overline{ \F\, \F^T :\D \bv}
\end{aligned}
\end{equation}
almost everywhere in $(0, T)\times\Omega$.  

Next, take the scalar product of \eqref{main_sys_for_g_theta}$_2$ with $2\F$ and obtain
\begin{equation*}
\begin{split}
        \p_t|\F_n|^2 + \DIV_x(|\F_n|^2 \bv_n) - 2\nabla_x \bv_n : (\F_n \, \F_n^T)  + \delta(\theta_n)(|\F_n\,\F^T_n|^2- |\F_n|^2) = 0.
\end{split}
\end{equation*}
Letting $n\to \infty$ we gain with the help of the above convergence results (see \cite{bulicek2024threeD} for rigorous justification) the following
\begin{equation}\label{lim11}
\begin{split}
        \p_t\overline{|\F|^2} + \DIV_x(\overline{|\F|^2}\bv) - 2\overline{\nabla_x\bv : (\F \, \F^T)}  + \delta(\theta)(\overline{|\F\,\F^T|^2}- \overline{|\F|^2}) = 0.
\end{split}
\end{equation}
Here, we want to set $\mathbb{G}:=\F \phi$ in  \eqref{almost_weak_formulation_F}, however such a setting is not possible due to the low integrability of the term  $\delta(\theta)\overline{\F\,\F^T\,\F}$, which does not belong to $L^{\frac{4}{3}}$, and therefore we cannot test by functions which are only in $L^4$. To solve this issue, we use as a test function 
$$
    \frac{\F}{1 + \varepsilon|\F|},
$$
which is bounded. Thanks to this choice, and after the classical renormalisation procedure (see \cite{diperna1989ordinary}), and with the help of already obtain convergence results, we deduce
\begin{equation*}
\begin{split}
        &2\p_t\frac{\varepsilon |\F|^2 -\ln (1+\varepsilon|\F|^2)}{\varepsilon^2} + 2\DIV_x\left(\frac{\varepsilon |\F|^2 -\ln (1+\varepsilon|\F|^2)}{\varepsilon^2} \bv\right)\\
         &\quad - 2\overline{\nabla_x \bv \, \F} : \frac{\F}{1+\varepsilon|\F|}  + \delta(\theta)\left(\overline{\F\, \F^T\, \F} - \F\right): \frac{\F}{1+\varepsilon|\F|} = 0.
\end{split}
\end{equation*}
Thanks to the fact we work with the biting limits, we can now easily let $\varepsilon \to 0_+$ and conclude
\begin{equation}
\label{lim12}
\begin{split}
        &\p_t |\F|^2 + \DIV_x\left(|\F|^2 \bv\right) - 2\overline{\nabla_x \bv \, \F} : \F  + \delta(\theta)\left(\overline{\F\, \F^T\, \F} - \F\right):\F = 0.
\end{split}
\end{equation}
Subtracting \eqref{lim12} from \eqref{lim11}, we see that
\begin{equation}\label{lim13}
\begin{split}
        &\p_t(\overline{|\F|^2}-|\F|^2) + \DIV_x\left((\overline{|\F|^2}-|\F|^2)\bv\right)  + \delta(\theta)\left(\overline{|\F\,\F^T|^2}-\overline{\F\, \F^T\, \F}:\F\right)\\ 
        &\quad = 2\left(\overline{\nabla_x\bv : (\F \, \F^T)}-\overline{\nabla_x \bv \, \F} : \F\right) +\delta(\theta)(\overline{|\F|^2}-|\F|^2)
\end{split}
\end{equation}
Here, the last term on the left hand side is nonnegative since the mapping $\F \mapsto \F\, \F^T \, \F$ is monotone, see \cite[Lemma~4.2]{bulicek2022onplanar}. The last term on the right hand side is in the required form and we need to focus on the first term on the right hand side. To do so, we use \cite[Theorem 1.7]{bulicek2024threeD} and it follows from the equation \eqref{eq:main_syst}$_1$, the convergence results \eqref{strong_convergence_of_u_n_C_Lp} and \eqref{conv_Fn}, and from the assumptions \eqref{bouds_nu} and \eqref{function:g_bounds_linfty} that
\begin{equation*}
\overline{\nu(\theta)|\D \bv|^2} + \overline{g(\theta) \F \,\F^T : \nabla_x \bv} =\overline{\nu(\theta) \D \bv} : \D\bv + \overline{g(\theta) (\F\,\F^T)} : \nabla_x \bv,
\end{equation*}
almost everywhere in $(0,T)\times \Omega$. Recall, that we consider her the biting limits. Using the strong convergence of the temperature \eqref{almost_everywhere_convergence_theta_n}, we deduce from the above identity that 
\begin{equation}\label{key}
\frac{\nu(\theta)}{g(\theta)}\overline{|\D \bv|^2} + \overline{ \F \,\F^T : \nabla_x \bv} =\frac{\nu(\theta)}{g(\theta)}|\D \bv|^2 + \overline{ (\F\,\F^T)} : \nabla_x \bv.
\end{equation}
Inserting \eqref{key} into \eqref{lim13}, where we also neglect the last term on the left hand side, leads to the following inequality 
\begin{equation}\label{lim14}
\begin{aligned}
        &\p_t(\overline{|\F|^2}-|\F|^2) + \DIV_x\left((\overline{|\F|^2}-|\F|^2)\bv\right)  +\frac{2\nu(\theta)}{g(\theta)} \left(\overline{|\D \bv|^2} -|\D \bv|^2\right)\\ 
        &\quad \le   2\left(\overline{ (\F\,\F^T)} : \nabla_x \bv-\overline{\nabla_x \bv \, \F} : \F\right) +\delta(\theta)(\overline{|\F|^2}-|\F|^2)\\
        &\quad =   2\left(\overline{ (\F\,\F^T)}-\F\, \F^T\right) : \D \bv +2\left(  \nabla_x \bv \, \F-\overline{\nabla_x \bv \, \F}\right):\F +\delta(\theta)(\overline{|\F|^2}-|\F|^2)\\
        &\quad \le 2|\D\bv|(\overline{|\F|^2}-|\F|^2)+2|\F|(\overline{|\nabla_x\bv|^2}-|\nabla_x\bv|^2)^{\frac12}(\overline{|\F|^2}-|\F|^2)^{\frac12} +\delta(\theta)(\overline{|\F|^2}-|\F|^2)
\end{aligned}
\end{equation}
Using  the assumptions \eqref{bouds_nu} and \eqref{function:g_bounds_linfty} and also the localised version of the Korn inequality, see Appendix in \cite{bulicek2024threeD}, we deduce from \eqref{lim14} 
\begin{equation}\label{lim15}
\begin{aligned}
        &\p_t(\overline{|\F|^2}-|\F|^2) + \DIV_x\left((\overline{|\F|^2}-|\F|^2)\bv\right) +\frac{2\nu(\theta)}{g(\theta)} \left(\overline{|\D \bv|^2} -|\D \bv|^2\right)\\ 
        &\quad \le 2|\D\bv|(\overline{|\F|^2}-|\F|^2)+C|\F|^2(\overline{|\F|^2}-|\F|^2)+\frac{2\nu(\theta)}{g(\theta)} \left(\overline{|\D \bv|^2} -|\D \bv|^2\right)+\delta(\theta)(\overline{|\F|^2}-|\F|^2).
\end{aligned}
\end{equation}
Finally, from the above inequality, we can deduce \eqref{inequality_to_prove_strong_L2}, where 
$$
    L := C(1 + |\D \bv|+|\F|^2 + \delta(\theta)).
$$
Consequently, the strong convergence~\eqref{strong_convergence_F_n} follows.
We want to emphasize that the above computations was rather formal, and we encourage the interested reader to \cite{bulicek2022onplanar,bulicek2024threeD} for rigorous justifications of several steps. 
%
%\begin{equation}
%\begin{aligned}\label{almost_weak_formulation_F}
%        &\int_0^T\int_\Omega -\F:\p_t\bG - \F\otimes \bv \because \nabla_x \bG - \overline{\nabla_x \bv\,\F} : \bG + \frac{1}{2}\left(\overline{\delta(\theta) \F\,\F^T\,\F} - \overline{\delta(\theta)\F}\right): \bG\diff x\diff t \\
%        &\qquad = \int_\Omega \F_0(x) : \bG(0, x)\diff x
%    \end{aligned}
%\end{equation}
%
%\cite[Theorem~1.7]{bulicek2024threeD}
    
The convergence \eqref{strong_convergence_F_n} combined with the convergence results obtained previously is enough to conclude \eqref{weak_formulation_u_g_theta}, \eqref{weak_formulation_F_g_theta} from \eqref{almost_weak_formulation_u} and \eqref{almost_weak_formulation_F}. To show \eqref{weak_formulation_theta_g_theta}, we can use the a~priori bounds \eqref{bounds_entropy_n}--\eqref{flux_n_estimate} and let $n\to \infty$ in the weak formulation of~\eqref{eq:entropy_equality_for_g_theta}. Using the point-wise convergence of $\theta_n$ and $\F_n$, and the Fatou lemma for the nonnegative terms on the right hand side, we can easily conclude \eqref{weak_formulation_theta_g_theta}.

Furthermore, due to  \eqref{almost_everywhere_convergence_theta_n} and \eqref{strong_convergence_F_n}, and since $\theta_n>0$ and $\det \F_n >0$, we have 
\begin{align}\label{theta_det_F_geq_0}
    \theta \geq 0,\quad \det\F \geq 0,\quad\text{ a.e. in }(0, T)\times\Omega.
\end{align}
Even more, the  Fatou lemma together with the uniform estimates \eqref{bound_on_B_n_L2}, \eqref{bound_on_ln_det_B_n} imply that
$$
    \|\ln\theta\|_{L^\infty_t L^1_x} + \|\ln\det(\F\,\F^T)\|_{L^\infty_t L^1_x} \leq C,
$$
which, together with \eqref{theta_det_F_geq_0}, is enough to deduce
$$
    \theta > 0, \quad \det\F > 0.
$$
%\end{proof}
almost everywhere in $(0,T)\times \Omega$.

To finish the proof, we multiply \eqref{Energy_n} by $\varphi \in \mathcal{C}^1_c(-\infty, T)$ and integrate over $(0,T)$ to get
\begin{equation}\label{Energy_n_in}
-\int_0^T \int_{\Omega} \left(\frac{|\bv_n|^2}{2} + e_n\right) \partial_t \varphi \diff x \diff t =\varphi(0)\int_{\Omega} \frac{|\bv_0^n|^2}{2} + e_0^n \diff x.
\end{equation} 
Using the uniform bounds \eqref{final_bounds_f_of_B_n} and \eqref{final_bounds_on_theta_n}, the definition of $e_n$ in \eqref{inten} and the strong convergence of $\theta_n$, $\F_n$ and $\bv_n$, we can let $n \to \infty$ in \eqref{Energy_n_in} and deduce \eqref{Energy_def}, where we also use the assumptions on $g$, see \eqref{function:g_bounds_linfty}--\eqref{function:g'_bounds_linfty}, and the strong convergence properties of the initial conditions in \eqref{conv:init}. The proof is complete.

\section{Existence of the weak solutions for the case \ref{contant_g_case}}\label{S:4}

We now focus on the problem~\ref{contant_g_case} and set $c_v=g(\theta)\equiv 1$ for simplicity. For this setting, we consider the system \eqref{main_sys_for_g_theta}--\eqref{eq:boundaryconditions_for_g_theta}, where we replace~\eqref{main_sys_for_g_theta}$_3$ by~\eqref{thetaP1}. Note that \eqref{main_sys_for_g_theta}$_3$ and~\eqref{thetaP1} are equivalent on the level of classical solutions.  Our main result reads as follows.
\begin{thm}\label{main_theorem_case_constant_g}
Assume that $\nu$,  $\kappa$ and $\delta$ are continuous functions satisfying \eqref{bounds_kappa}--\eqref{bounds_delta}. Let initial conditions  $\{\bv_0, \theta_0, \F_0\}$ fulfill
\begin{equation}\label{conv:initII}
\begin{aligned}
\bv_0 \in L^2_{0,\DIV},\quad  \theta_0 \in L^1_x, \quad \ln \theta_0 \in L^1_x, \quad \F_0 \in L^2_x \quad \textrm{ and }\quad  \ln \det F_0 \in L^2_x 
\end{aligned}
\end{equation}
and  $\det\F_0>0$ and $\theta_0 > 0$ almost everywhere in $\Omega$. Then there exists a triple $\{\bv, \F, \theta \}$ such that  
\begin{equation}\label{spaces}
\begin{aligned}
\bv &\in \mathcal{C}([0,T]; L^2_{0,\DIV}) \cap L^2_t W^{1,2}_{0,x},\\
\F&\in \mathcal{C}([0,T]; L^2_x) \cap L^4_{t,x},\\
\theta &\in L^{\infty}_t L^1_x \cap L^{p}_{t,x}\cap L^{q}_tW^{1, q}_{x} &&\textrm{  for any } p\in [1,2) \textrm{ and }q\in \left[1,\frac43\right),\\
\ln \theta &\in L^{\infty}_t L^1_x, \qquad \ln \det \F \in  L^{\infty}_t L^1_x,
\end{aligned}
\end{equation}
where $\det \F>0$ and $\theta >0$ almost everywhere in $(0,T)\times \Omega$. The functions $(\bv, \F, \theta)$ solves  \eqref{main_sys_for_g_theta}--\eqref{eq:boundaryconditions_for_g_theta} in the following sense:
\begin{equation}
\begin{aligned}\label{weak_formulation_u_constant_g}
        &\int_0^T\int_\Omega -\bv \cdot\p_t\bphi - \bv\otimes \bv : \nabla_x\bphi + 2\nu(\theta)\D\bv :\nabla_x\bphi + 2\F\,\F^T:\nabla_x\bphi\diff x\diff t \\
        &\qquad = \int_\Omega \bv_0(x)\cdot\bphi(0, x)\diff x
\end{aligned}
\end{equation}
for any $\vect{\varphi}\in \mathcal{C}^1_c([0, T)\times\Omega; \R^2)$ with $\DIV_x\bphi = 0$,
\begin{equation}
\begin{aligned}\label{weak_formulation_F_constant_g}
        &\int_0^T\int_\Omega -\F:\p_t\bG - \F\otimes \bv \because \nabla_x \bG - \nabla_x \bv\,\F : \bG + \frac{1}{2}\delta(\theta)(\F\,\F^T\,\F - \F): \bG\diff x\diff t \\
        &\qquad = \int_\Omega \F_0(x) : \bG(0, x)\diff x
    \end{aligned}
\end{equation}
    for any $\bG\in \mathcal{C}^1_c([0, T)\times\Omega; \R^{2\times 2})$,
\begin{equation}
\begin{aligned}\label{weak_formulation_theta_constant_g}
       & -\int_0^T\int_\Omega\theta\,\p_t\phi\diff x\diff t - \int_0^T\int_{\Omega}\theta\,\bv\cdot \nabla_x\phi\diff x\diff t + \int_0^T\int_{\Omega}\kappa(\theta)\nabla_x\theta\cdot\nabla_x\phi\diff x\diff t\\
   &\quad  -\int_0^T\int_{\Omega}2\nu(\theta)|\D\bv|^2\phi\diff x\diff t - \int_0^T\int_{\Omega}\delta(\theta)|\F\,\F^T - \mathbb{I}|^2\phi\diff x\diff t = \int_\Omega \theta_0(x)\phi(0,x)\diff x
    \end{aligned}
\end{equation}
for any $\phi\in \mathcal{C}^1_c([0, T)\times\Omega)$.
\end{thm}
Since the proof of Theorem \ref{main_theorem_case_constant_g} is very similar to the one conducted in \cite{bulicek2022onplanar,bathory2021largedata}, and does not include many new techniques, compared to the proof of the Theorem \ref{main_theorem_case_general_g} we will skip most of the technical detail and simply focus on the differences in the approach.

\subsection{Approximating scheme}
To prove our main theorem, we use a four-step approximation scheme. Our first goal will be to prove the existence of a solution to the system
\begin{equation}\label{epsilon_app}
\left\{
\begin{aligned}
&\partial_t \bv + \DIV_x (\bv\otimes \bv) -\DIV_x\T = 0, \qquad \DIV_x \bv = 0,\\
&\p_t\F + \DIV_x(\F\otimes \bv) - \nabla_x \bv\,\F + \frac{1}{2}\delta(\theta)(\F\,\F^T\, \F - \F) = \varepsilon\Delta\F,\\
&-p\mathbb{I} + 2\nu(\theta)\D \bv + 2(\F\,\F^T - \mathbb{I}) = \T,\\
&\p_t\theta + \DIV_x (\theta \bv) -\DIV_x(\kappa(\theta)\nabla_x\theta) = 2\nu(\theta)|\D \bv |^2 + \delta(\theta)|\F\,\F^T - \mathbb{I}|^2.
\end{aligned}
\right.
\end{equation}
for a fixed $\varepsilon\in(0, 1)$. The initial and boundary conditions are set  as
\begin{align}\label{eq:boundaryconditions_approx_syst}
\bv|_{\p\Omega} = 0, \quad \bv(0, x) = \bv_0(x), \quad \F(0, x) = \F_0, \quad \nabla_x\F\cdot\mathbf{n}=0, \quad \nabla_x\theta\cdot \mathbf{n} = 0, \quad \theta(0, x) = \theta^r_0,
\end{align}
for $r \in (0,1)$ and
\begin{align}\label{thetar}
    \theta^r_0(x) = \left\{\begin{aligned}
         &\theta_0(x) \quad &&\text{ whenever }r\leq \theta_0(x)\leq r^{-1},  \\
         &1 \quad &&\text{ otherwise.}
    \end{aligned}
    \right.
\end{align}
To prove the existence of the $\varepsilon$-approximation, we employ the Galerkin approximation. Most importantly, we split the convergence results. The Galerkin approximative schema for $\theta$ and $\bv$ and $\F$ allows testing only by linear functions of the solutions. However, and as can be seen by the result of Lemma~\ref{lem:renormalized_equation_for_theta}, the expected bounds on the temperature are obtained via testing by a nonlinear function. Therefore, the existence scheme is not completely trivial and here we follow the methods developed in \cite{bathory2021largedata,BuMaRa09,BuFeMa09}. This means that we first converge in the equation for the temperature but keep the equation for $\bv$ and $\F$ in the Galerkin form. Then, we can deduce the optimal estimates for the temperature and consequently also for $\bv$ and $\F$. Next, we remove the finite-dimensional approximation for all quantities and finally, we let $\varepsilon \to 0_+$ and $r \to 0_+$ in \eqref{epsilon_app}. 

\subsection{Galerkin approximation} We fix $s>3$. Since $\Omega$ is two dimensional domain, we have
\begin{align}\label{morrey_embedding_galerkin}
W^{s-1, 2}(\Omega) \hookrightarrow L^\infty(\Omega).
\end{align}
Next, we consider $\{\bomega_n\}_{n\in\N}$, $\{\A_n\}_{n\in\N}$, $\{\phi_m\}_{m\in\N}$ to be  the orthogonal bases of $W^{s, 2}_{0,\DIV}(\Omega; \R^2)$, $W^{s,2}(\Omega;\R^{2\times 2})$ and $W^{s, 2}(\Omega; \R)$ respectively, that are also orthonormal with repsect to  $L^2$, and whose projections (in $L^2$)  are continuous. We also introduce the projection of the initial data as 
\begin{equation}\label{init_G}
\begin{aligned}
\bv_{0n}&:=\sum_{i=1}^n\bomega_i \left(\int_{\Omega}\bv_{0}\bomega_i \diff x\right)\qquad &&\textrm{ and we have }\|\bv_{0n}\|_{L^2_x}\le \|\bv_{0}\|_{L^2_x},\\
\F_{0n}&:=\sum_{i=1}^n\A_i \left(\int_{\Omega}\F_{0}\A_i \diff x\right)\qquad &&\textrm{ and we have }\|\F_{0n}\|_{L^2_x}\le \|\F_{0}\|_{L^2_x},\\
\theta_{0m}^r&:=\sum_{i=1}^m\phi_i \left(\int_{\Omega}\theta_{0}^r\phi_i \diff x\right)\qquad &&\textrm{ and we have }\|\theta_{0n}^r\|_{L^2_x}\le \|\theta_{0}^r\|_{L^2_x}
\end{aligned}
\end{equation}
and note that we have the following convergence results
\begin{equation}\label{conv_init}
\begin{aligned}
\bv_{0n}&\to \bv_{0} &&\textrm{strongly in } L^2_x,\\ 
\F_{0n}&\to \F_{0} &&\textrm{strongly in } L^2_x,\\ 
\theta_{0n}^r &\to \theta_{0}^r &&\textrm{strongly in } L^2_x,
\end{aligned}
\end{equation}
where $\theta_0^r$ is defined in \eqref{thetar}. We look for the Galerkin approximation of \eqref{epsilon_app}, which has the following form
\begin{align}\label{def.galerkin}
    \bv_{nm}(t, x) = \sum_{i = 1}^n \alpha^{nm}_i(t)\bomega_i(x),\quad \F_{nm}(t, x) = \sum_{i = 1}^n \beta^{nm}_i(t)\A_i(x),\quad \theta_{nm}(t, x) = \sum_{i = 1}^m\gamma^{nm}_i(t)\phi_i(x),
\end{align}
and we require the approximations to satisfy the following initial conditions
\begin{equation}\label{init_G2}
\bv_{nm}(0, x)=\bv_{0n}(x), \qquad \F_{nm}(0, x)=\F_{0n}(x), \qquad \theta_{nm}(0, x)=\theta_{0m}^r( x)
\end{equation}
for almost all $x\in \Omega$, and we also require them to satisfy the following system of ordinary differential equations:
\begin{equation}
\begin{aligned}\label{eq:galerkin_velocity}
    &\frac{\diff }{\diff t}\int_{\Omega}\bv_{nm}\cdot\bomega_j\diff x - \int_{\Omega}\bv_{nm}\otimes \bv_{nm}:\nabla_x\bomega_j\diff x\\
    &\qquad + \int_{\Omega}2\nu(\theta_{nm})\D \bv_{nm} :\nabla_x\bomega_j\diff x + \int_{\Omega}2\F_{nm}\,\F_{nm}^T : \nabla_x\bomega_j\diff x = 0,
\end{aligned}
\end{equation}
for any $j = 1,\ldots, n$;
\begin{equation}\label{eq:galerkin_elastic}
\begin{aligned}
    &\frac{\diff}{\diff t}\int_{\Omega}\F_{nm}: \A_j\diff x - \int_{\Omega}\F_{nm}\otimes \bv_{nm} \because \nabla_x \A_j\diff x - \int_{\Omega}(\nabla_x \bv_{nm}\, \F_{nm}):\A_j\diff x  \\
    &\quad +\frac{1}{2}\int_{\Omega}\delta(\theta_{nm})\left(\F_{nm}\,\F_{nm}^T\,\F_{nm}-\F_{nm}\right) : \A_j\diff x  + \varepsilon\int_{\Omega}\nabla_x \F_{nm} \because \nabla_x \A_j\diff x = 0,
\end{aligned}
\end{equation}
for any $j = 1, \ldots, n$; 
\begin{equation}\label{eq:galerkin_temperature}
\begin{aligned}
    &\frac{\diff}{\diff t}\int_{\Omega}\theta_{nm}\phi_i\diff x - \int_{\Omega}\bv_{nm} \theta_{nm}\cdot \nabla_x \phi_i\diff x + \int_{\Omega}\kappa(\theta_{nm})\nabla_x\theta_{nm}\cdot\nabla_x\phi_i\diff x\\
    &\qquad -\int_{\Omega}2\nu(\theta_{nm})|\D\bv_{nm}|^2\phi_i\diff x - \int_{\Omega}\delta(\theta_{nm})|\F_{nm}\,\F_{nm}^T - \mathbb{I}|^2\phi_i\diff x = 0
\end{aligned}
\end{equation}
for any  $i = 1,\ldots, m$. 

The local-in-time existence of a solution to the above system follows from the Carath\'{e}odory theory and the global-in-time existence then follows from the estimates deduced in the next section.

\subsection{Convergence with \texorpdfstring{$m\to+\infty$}{m}}

We follow very closely the argumentation in \cite[Appendix B - estimate (B.15)]{bulicek2022onplanar} and therefore the proof here will be sketchy and we focus mainly onl the parts that are different. Hence, we  multiply \eqref{eq:galerkin_velocity} and \eqref{eq:galerkin_elastic} by $\alpha_j$ and $2\beta_j$ respectively, after summing over $j=1,\ldots, n$ and adding both equations together, we deduce also with the help of the Young inequality and the assumptions on material parameters \eqref{bounds_kappa}--\eqref{bounds_delta}
\begin{equation}
\begin{aligned}\label{galerkin_bounds_on_velocity_and_elastic_symm gradient}
    &\|\bv_{nm}\|_{L^\infty_t L^2_x} + \|\F_{nm}\|_{L^\infty_t L^2_x} + \|\D \bv_{nm}\|_{L^2_{t, x}} + \|\F_{nm}\|_{L^4_{t, x}} + \sqrt{\varepsilon}\|\nabla_x \F_{nm}\|_{L^2_{t,x}}\\
    &\qquad  \leq C(\|\bv_{0n}\|_{L^2_x}, \|\F_{0n}\|_{L^2_x})\le C,
\end{aligned}
\end{equation}
where for the second inequality we used the properties of initial conditions in~\eqref{init_G}.  Using also the  Korn inequality, we have
\begin{align}\label{galerkin_bounds_on_velocity_and_elastic}
    \|\bv_{nm}\|_{L^\infty_t L^2_x} + \|\F_{nm}\|_{L^\infty_t L^2_x} + \|\nabla_x \bv_{nm}\|_{L^2_{t, x}} + \|\F_{nm}\|_{L^4_{t, x}} + \sqrt{\varepsilon}\|\nabla_x \F_{nm}\|_{L^2_{t,x}} \leq C.
\end{align}
It is worth noting here, that the above estimates are independent of $n$, $m$ and $r$ and due to weak-lower semicontinuity are kept till the end of the proof. We follow by estimates that are $n$- or $r$-dependent. Due to the orthogonality of the bases $\{\bomega_n\}_{n\in\N}$ and $\{\A_n\}_{n\in\N}$, the embedding \eqref{morrey_embedding_galerkin}, and \eqref{galerkin_bounds_on_velocity_and_elastic} we get
\begin{equation}
\begin{aligned}\label{infty_bounds_for_velocit_and_elastic}
    &\|\bv_{nm}\|_{L^\infty_t W^{1, \infty}_x} \leq \|\bv_{nm}\|_{L^\infty_t W^{s, 2}_x} \leq C(n)\|\bv_{nm}\|_{L^\infty_t L^2_x} \leq C(n),\\
    &\|\F_{nm}\|_{L^\infty_t W^{1, \infty}_x} \leq \|\F_{nm}\|_{L^\infty_t W^{s, 2}_x} \leq C(n)\|\F_{nm}\|_{L^\infty_t L^2_x} \leq C(n).
\end{aligned}
\end{equation}
Next, we deduce the estimates for the temperature. We multiply \eqref{eq:galerkin_temperature} by $\gamma_j(t)$ and sum the results over $j=1,\ldots, m$. Using the fact that $\DIV_x \bv_{nm}=0$ we obtain
\begin{align*}
    \frac{1}{2}\frac{\diff}{\diff t}\|\theta_{nm}\|^2_{L^2_x} + \|\sqrt{\kappa(\theta_{nm})}\nabla_x\theta_{nm}\|^2_{L^2_x} = \int_{\Omega}2\nu(\theta_{nm})|\D\bv_{nm}|^2\theta_{nm} + \int_{\Omega}\delta(\theta_{nm})|\F_{nm}\,\F^T_{nm} - \mathbb{I}|^2\theta_{nm}.
\end{align*}
Thus, with the use of H\"{o}lder's inequality, \eqref{infty_bounds_for_velocit_and_elastic} as well as the bounds \eqref{bounds_kappa}--\eqref{bounds_delta}, we arrive at
\begin{align}\label{bounds_temp_1}
    \frac{1}{2}\frac{\diff}{\diff t}\|\theta_{nm}\|^2_{L^2_x} + C_1\|\nabla_x\theta_{nm}\|^2_{L^2_x}\leq C\left(1+\|\bv_{nm}\|^2_{L^\infty_t W^{1, \infty}_x} + \|\F_{nm}\|^4_{L^\infty_t L^\infty_x}\right)\|\theta_{nm}\|_{L^2_x}.
\end{align}
The Gr\"{o}nwall inequality and the estimates \eqref{infty_bounds_for_velocit_and_elastic} then directly leads to
\begin{align}\label{bounds_temp_2}
    \|\theta_{nm}\|_{L^{\infty}_t L^2_x} + \|\nabla_x\theta_{nm}\|_{L^2_{t,x}} \leq C(n,T)\|\theta_{0mr}\|_{L^2_x}\le C(n,r).
\end{align}
Moving forward, we focus on estimates for time derivatives. First, using the equation \eqref{eq:galerkin_velocity}, the orthogonality of the basis $\{\bomega_i\}_{i\in \mathbb{N}}$  and applying the bounds \eqref{infty_bounds_for_velocit_and_elastic} together with the H\"{o}lder inequality, we get
\begin{equation*}
    \begin{aligned}
        |(\alpha^{nm}_j)'(t)|&= \left|\int_{\Omega} \p_t \bv_{nm} \cdot  \bomega_j \diff x \right|\\
        &=\left| \int_{\Omega} \bv_{nm}\otimes \bv_{nm} :  \nabla_x\bomega_j - (\nu(\theta_{nm})\D\bv_{nm} :  \nabla_x\bomega_j) \right.
        - (\F_{nm}\,\F^T_{nm}) :  \nabla_x\bomega_j \diff x \Big| \le C(n).
    \end{aligned}
\end{equation*}
Thus, using the definition of $\bv_{nm}$ in \eqref{def.galerkin}, we get
\begin{align}\label{galerkin_strong_time_derivative_bounds_velocity}
    \|\p_t \bv_{nm}\|_{L^{\infty}_t W^{1,\infty}_x} \leq C(n).
\end{align}%
%At the same time, with the help of the Fundamental Theorem of Calculus and H\"{o}lder's inequality, we may deduce
%\begin{align}\label{arzela_ascoli_est_velocity}
%    |\alpha_j(t) - \alpha_j(s)| \leq \int_s^t|(\alpha_j)'(\tau)|\diff \tau\lesssim |t - s|^{\frac{1}{2}},\quad\text{for any }t, s, \in\R.
%\end{align}
By a very similar arguments, we can deduce from \eqref{eq:galerkin_elastic} and \eqref{infty_bounds_for_velocit_and_elastic} that 
\begin{align}\label{galerkin_strong_time_derivative_bounds_elastic}
    \|\p_t\F_{nm}\|_{L^{\infty}_t W^{1,\infty}_x} \leq C(n).
\end{align}
%and
%\begin{align}\label{arzela_ascoli_est_elastic}
%    |\beta_j(t) - \beta_j(s)| \lesssim |t - s|^{\frac{1}{2}},\quad\text{for any }t, s, \in\R.
%\end{align}
For the temperature, we proceed slightly differently. Due to the continuity (independent of $m$) of the projection of the basis $\{\phi_j\}_{j\in \mathbb{N}}$ and from the identity \eqref{eq:galerkin_temperature}, we have for all times $t\in (0,T)$ 
\begin{equation*}
    \begin{split}
       \|\p_t \theta_{nm}\|^2_{W^{-1,2}(\Omega)} &\le C \left(\int_{\Omega} |\bv_{nm}\theta_{nm}|^2 + |\nabla_x\theta_{nm}|^2 \diff x\right)\\
       &\quad  +C\left(\int_{\Omega}|\D\bv_{nm}|^4+ |\F_{nm}\,\F_{nm}^T - \mathbb{I}|^4\diff x\right)\\
       & \le C(n)\left(1+\|\theta_{nm}\|_{W^{1,2}_x}^2\right),
    \end{split}
\end{equation*}
where we also used the assumptions \eqref{bounds_kappa}--\eqref{bouds_nu}. Integration over $(0,T)$ and the use of \eqref{bounds_temp_2} gives
\begin{equation}\label{galerkin_strong_bounds_time_derivative_temperature}
    \|\p_t\theta_{nm}\|_{L^2((0, T); (W^{1,2}(\Omega))^*)}\le C(n,r).
\end{equation}

Having the estimates \eqref{infty_bounds_for_velocit_and_elastic}, \eqref{galerkin_strong_time_derivative_bounds_velocity}, \eqref{galerkin_strong_time_derivative_bounds_elastic}, \eqref{bounds_temp_2} and \eqref{galerkin_strong_bounds_time_derivative_temperature}, we can use 
the Banach--Alaoglu theorem, the classical Sobolev--Morrey embedding and the Aubin--Lions compactness lemma, see Lemma~\ref{aubin-lions}, and we can find  a triple $\{\bv_{n}, \F_n, \theta_n\}$ such that for a subsequence that we do not relabel we have the following convergence results as $m\to \infty$: for the velocity field
\begin{equation}
\begin{aligned}
    \p_t \bv_{nm} &\overset{*}{\rightharpoonup} \p_t \bv_{n}\qquad &&\text{ weakly* in }L^{\infty}_t W^{1,\infty}_{0,\DIV},\label{galerkin_m_conv_velocity}\\
    \bv_{nm} &\rightarrow \bv_{n}\qquad &&\text{ strongly in } \mathcal{C}([0, T]; W^{1,\infty}(\Omega));
\end{aligned}
\end{equation}
for the elastic stress tensor
\begin{equation}
\begin{aligned}
    \p_t \F_{nm} &\overset{*}{\rightharpoonup} \p_t \F_{n}\qquad &&\text{ weakly* in }L^2_t W^{1,\infty}_x,\\
    \F_{nm} &\rightarrow \F_{n}\qquad &&\text{ strongly in } \mathcal{C}([0, T]; W^{1,\infty}(\Omega));\label{galerkin_m_conv_elastic}
\end{aligned}
\end{equation}
and for the temperature
\begin{equation}
\begin{aligned}
    \p_t \theta_{nm} &\rightharpoonup \p_t \theta_{n}\qquad &&\text{ weakly in }L^2((0, T); W^{-1,2}(\Omega)),\label{galerkin_m_conv_temp_partial_t}\\
    \theta_{nm} &\overset{*}{\rightharpoonup} \theta_{n}\qquad &&\text{ weakly* in }L^\infty_t L^2_x,\\
    \theta_{nm} &\rightarrow \theta_{n}\qquad &&\text{ strongly in }L^2_t L^2_x.%\label{galerkin_m_conv_temp_strong}
\end{aligned}
\end{equation}
In addition, we have the following form for $\bv_n$ and $\F_n$
\begin{align}\label{def.galerkin22}
    \bv_{n}(t, x) = \sum_{i = 1}^n \alpha^{n}_i(t)\bomega_i(x),\quad \F_{n}(t, x) = \sum_{i = 1}^n \beta^{n}_i(t)\A_i(x).
\end{align}
The convergence results \eqref{galerkin_m_conv_velocity}--\eqref{galerkin_m_conv_temp_partial_t} allows us to let $m\to \infty$ in  \eqref{eq:galerkin_velocity}--\eqref{eq:galerkin_temperature} and deduce 
\begin{equation}
\begin{aligned}\label{eq:galerkin_1stconv_velocity}
    &\frac{\diff }{\diff t}\int_{\Omega}\bv_{n}\cdot\bomega_j\diff x - \int_{\Omega}\bv_{n}\otimes \bv_{n}:\nabla_x\bomega_j\diff x\\
    &\qquad + \int_{\Omega}2\nu(\theta_{n})\D \bv_{n} :\nabla_x\bomega_j\diff x + \int_{\Omega}2\F_{n}\,\F_{n}^T : \nabla_x\bomega_j\diff x = 0,
\end{aligned}
\end{equation}
for any $j = 1,\ldots, n$;
\begin{equation}\label{eq:galerkin_1stconv_elastic}
\begin{aligned}
    &\frac{\diff}{\diff t}\int_{\Omega}\F_{n}: \A_j\diff x - \int_{\Omega}\F_{n}\otimes \bv_{n} \because \nabla_x \A_j\diff x - \int_{\Omega}(\nabla_x \bv_{n}\, \F_{n}):\A_j\diff x  \\
    &\quad +\frac{1}{2}\int_{\Omega}\delta(\theta_{n})\left(\F_{n}\,\F_{n}^T\,\F_{n}-\F_{n}\right) : \A_j\diff x  + \varepsilon\int_{\Omega}\nabla_x \F_{n} \because \nabla_x \A_j\diff x = 0,
\end{aligned}
\end{equation}
for any $j = 1, \ldots, n$; 
\begin{equation}\label{eq:galerkin_1stconv_temperature}
\begin{aligned}
    &\langle \partial_t \theta_{n}\phi\rangle - \int_{\Omega}\bv_{n} \theta_{n}\cdot \nabla_x \phi\diff x + \int_{\Omega}\kappa(\theta_{n})\nabla_x\theta_{n}\cdot\nabla_x\phi\diff x\\
    &\qquad -\int_{\Omega}2\nu(\theta_{n})|\D\bv_{n}|^2\phi\diff x - \int_{\Omega}\delta(\theta_n)|\F_{n}\,\F_{n}^T - \mathbb{I}|^2\phi\diff x = 0
\end{aligned}
\end{equation}
for any $\phi \in W^{1,2}(\Omega)$ and almost all $t\in (0,T)$. Furthermore, due to the standard parabolic embedding we also have that $\theta_n\in \mathcal{C}([0,T]; L^2(\Omega))$ and that the unknowns attain the  following initial conditions 
$$
\bv_n(0)=\bv_{0n}, \qquad \F_n(0)=\F_{0n},\qquad \theta_n(0)=\theta_0^r.
$$

\subsection{Convergence with \texorpdfstring{$n\to +\infty$}{n}} \label{sect12}
First, we recall the uniform bounds for $\bv_n$ and $\F_n$. Due to the weak-lower semicontinuity, it follows from \eqref{galerkin_bounds_on_velocity_and_elastic} that 
\begin{align}\label{galerkin_bounds_on_velocity_and_elastic_n}
    \|\bv_{n}\|_{L^\infty_t L^2_x} + \|\F_{n}\|_{L^\infty_t L^2_x} + \|\nabla_x \bv_{n}\|_{L^2_{t, x}} + \|\F_{n}\|_{L^4_{t, x}} + \sqrt{\varepsilon}\|\nabla_x \F_{n}\|_{L^2_{t,x}} \leq C.
\end{align}
For the  time derivatives $\partial_t \bv_n$ and $\partial_t \F_n$, the bounds \eqref{galerkin_strong_time_derivative_bounds_elastic}, \eqref{galerkin_strong_time_derivative_bounds_velocity} are not uniform with respect to~$n$. Therefore, we must proceed differently. Based on the estimate \eqref{galerkin_bounds_on_velocity_and_elastic_n} and on the identities \eqref{eq:galerkin_1stconv_velocity}--\eqref{eq:galerkin_1stconv_elastic}, following \cite[Appendix B - estimates (B.19) and (B.21)]{bulicek2022onplanar}  we  obtain
\begin{align}\label{final_bounds_time_elastic_velocity}
    \|\p_t \bv_{n}\|_{L^2((0, T); W^{-1, 2}_{0, \DIV}(\Omega))} + \|\p_t \F_{n}\|_{L^{\frac{4}{3}}((0, T); W^{-1, 2}(\Omega))} \leq C.
\end{align}

Next, we focus on estimates for $\theta_n$. In this case, we want to test \eqref{eq:galerkin_1stconv_temperature} by $\theta_n^{\lambda - 1}$ for  arbitrary $0 < \lambda < 1$ as was already explained in Lemma~\ref{lem:renormalized_equation_for_theta}. To do so properly, we first provide the minimum principle for $\theta_n$. We set $\phi:=\min \{0, \theta_n - r\}$ in \eqref{eq:galerkin_1stconv_temperature} and in a very similar manner as in~\cite{BaBuMa24}, we observe that 
\begin{align}\label{theta_bigger_than_r}
\theta_n(x, t) \geq r>0,\quad \text{ for a.e. }(t,x) \in (0,T)\times \Omega.
\end{align}
We proceed by the uniform estimate for $\theta_n$.  We set $\phi:=1$ in~\eqref{eq:galerkin_1stconv_temperature}. Then, after integrating the result over  $(0, t)$, using Lemma \ref{lem:integration_of_lipschitz}, the condition $\DIV_x \bv_n = 0$, and the positivity of $\theta_n$ \eqref{theta_bigger_than_r}, we get for almost all $t\in (0,T)$
$$
\|\theta_n(t)\|_{L^1_x}  = \int_0^t\int_\Omega \nu(\theta_n)|\D\bv_n|^2 + \delta(\theta_n)|\F_n\,\F_n^T - \mathbb{I}|^2\diff \tau + \|\theta_n(0)\|_{L^1_x}.
$$
Thus, after applying the estimate \eqref{galerkin_bounds_on_velocity_and_elastic_n},  the definition of $\theta_0^r$ and the assumption \eqref{conv:initII}, we  get
\begin{align}\label{temp_linfty_l1_bounds}
    \|\theta_n\|_{L^\infty_t L^1_x} \leq C(1+ \|\theta_{0}^r\|_{L^1_x})\leq C(1+ \|\theta_{0}\|_{L^1_x})\le C.
\end{align}
Moving forward, we can set $\phi:=\theta_n^{\lambda-1}$ in~\eqref{eq:galerkin_1stconv_temperature}. Performing similar operations as in Lemma~\ref{lem:renormalized_equation_for_theta}, we obtain (using also the uniform bounds \eqref{galerkin_bounds_on_velocity_and_elastic_n})
\begin{align}\label{temp_fractional_gradient_bound}
    \int_0^T\int_{\Omega}\frac{|\nabla_x\theta_{n}|^2}{\theta_n^{\lambda}}\diff x\diff t \leq C(\lambda) \quad\text{ for any }\lambda\in(1, 2).
\end{align}
In particular, setting $\phi:=\theta_n^{-1}$ in~\eqref{eq:galerkin_1stconv_temperature}, we deduce that for all $t\in (0,T)$
\begin{equation}
\begin{aligned}\label{log_bound}
 \int_{\Omega} |\ln \theta_n (t)| \diff x  +  \int_0^T\int_{\Omega}\frac{|\nabla_x \theta_{n}|^2}{\theta_n^{2}}\diff x\diff t &\leq C\left( 1 + \int_{\Omega} \theta(t) + \theta_0^r + |\ln \theta_0^r| \diff x\right)\\
 &\leq C\left( 1 + \int_{\Omega} \theta_0 + |\ln \theta_0| \diff x\right)\le C,
\end{aligned}
\end{equation}
where we used \eqref{temp_linfty_l1_bounds}. Note here, we got the $r$-independent estimate. With this (similarly as in the estimates below \eqref{bound_grad_theta_n_with_lambda}), we may infer uniform bounds on $\theta_n$ and $\nabla_x\theta_n$
\begin{align}
    \|\theta_n\|_{L^{p'}_{t} L^{p}_x} &\leq C(p) &&\textrm{for all } p\in [1,\infty),\label{final_bounds_on_temp}\\
    \|\nabla_x\theta_n\|_{L^{q}_{t,x}} &\leq C(q) &&\textrm{for all } q\in \left[1,\frac43 \right), \label{final_bounds_on_grad_temp}\\
    \|\ln \theta_n\|_{L^{\infty}_{t} L^{1}_x} &\leq C,\label{final_bounds_on_log}\\
    \|\nabla_x \ln \theta_n\|_{L^{2}_{t,x}} &\leq C. \label{final_bounds_on_grad_temp_log}
\end{align}
In order to get the compactness of the temperature, we also deduce the estimate for the time derivative of $\theta_n$. This is done similarly as in \eqref{galerkin_strong_bounds_time_derivative_temperature}, but because we have much weaker estimates on $\D\bv_n$ and $\F_n$ than in previous section, compare \eqref{infty_bounds_for_velocit_and_elastic} and \eqref{galerkin_bounds_on_velocity_and_elastic_n}, we can deduce from \eqref{eq:galerkin_1stconv_temperature} the following bound (note that $W^{2,2}\hookrightarrow \mathcal{C}(\overline{\Omega})$)
\begin{align}\label{final_bounds_time_temp}
    \|\p_t\theta_n\|_{L^1((0, T); W^{-2, 2}(\Omega))} \leq  C.
\end{align}
Consequently, we may apply the Banach--Alaoglu theorem and the generalized version of the  Aubin--Lions lemma, see Lemma~\ref{aubin-lions}, and it follows from \eqref{final_bounds_on_temp}--\eqref{final_bounds_time_temp} that there is a subsequence that we do not relabel such that
\begin{align}
     \theta_n &\rightharpoonup \theta\quad &&\text{ weakly in }L^{q}_{t,x} \textrm{ for all }q\in [1,2),\label{weak_conv_temp_2_lambda}\\
     \theta_n &\rightharpoonup \theta\quad &&\text{ weakly in }L^{p}_t W^{1, p}_x \textrm{ for all } p\in \left[1,\frac43 \right),\\
     \ln \theta_n &\rightharpoonup \ln \theta\quad &&\text{ weakly in }L^{2}_t W^{1, 2}_x,\\ 
     \theta_n &\rightarrow \theta\quad &&\text{ strongly in }L^{q}_t L^{s}_x \textrm{ for all } s\in [1,\infty) \textrm{ and all } q\in [1,s').\label{thn_s}
\end{align}
Similarly, due to \eqref{galerkin_bounds_on_velocity_and_elastic_n}--\eqref{final_bounds_time_elastic_velocity}, we have
\begin{align}
    \bv_n &\rightharpoonup \bv\quad &&\text{ weakly* in }L^{\infty}_t L^2_x,\label{CCC1}\\
    \bv_n &\rightharpoonup \bv\quad &&\text{ weakly in }L^{2}_t W^{1, 2}_{0,x},\label{weak_conv_velocity_sobolev}\\
    \p_t \bv_n &\rightharpoonup \p_t \bv\quad &&\text{ weakly in }L^{2}((0, T); W^{-1, 2}_{0, \DIV}(\Omega))\\
    \bv_n &\rightarrow \bv\quad &&\text{ strongly in }L^q_{t,x}, \text{ for }q\in[1, 4),\\
    \F_n &\rightharpoonup \F\quad &&\text{ weakly* in }L^{\infty}_t L^2_x,\\
    \F_n &\rightharpoonup \F\quad &&\text{ weakly in }L^{2}_t W^{1, 2}_x,\\
    \p_t \F_n &\rightharpoonup \p_t \F\quad &&\text{ weakly in }L^{\frac{4}{3}}((0, T); W^{-1, 2}(\Omega;\R^{2\times 2}))\\
    \F_n &\rightarrow \F\quad &&\text{ strongly in }L^q_{t,x}, \text{ for }q\in[1, 4).\label{strong_conv_elastic_n}
\end{align}
The convergence results \eqref{weak_conv_temp_2_lambda}--\eqref{strong_conv_elastic_n} allow us to let $n\to \infty$ in all terms appearing in  \eqref{eq:galerkin_1stconv_velocity}--\eqref{eq:galerkin_1stconv_temperature} except the terms involving 
$$
\nu(\theta_{n})|\D\bv_{n}|^2,\qquad |\F_{n}\,\F_{n}^T - \mathbb{I}|^2\qquad \textrm{ and } \qquad \partial_t \theta_n.
$$
In particular, for the $\bv$ and $\F$ we obtain the following identities:
\begin{align}\label{eq:galerkin_2ndconv_velocity}
    \langle\p_t \bv,\bomega\rangle - \int_{\Omega}\bv\otimes \bv :\nabla_x\bomega\diff x
    + \int_{\Omega}2\nu(\theta)\D \bv : \nabla_x\bomega\diff x + \int_{\Omega}2\F\,\F^T : \nabla_x\bomega\diff x = 0
\end{align}
for almost all time $t\in (0,T)$ and for all $\bomega\in W^{1,2}_{0, \DIV}(\Omega)$;
\begin{equation}
\begin{aligned}\label{eq:galerkin_2ndconv_elastic}
    &\langle\p_t \F, \A\rangle - \int_{\Omega}\F\otimes \bv \because \nabla_x \A\diff x - \int_{\Omega}(\nabla_x \bv\, \F): \A\diff x  \\
    &\qquad +\frac{1}{2}\int_{\Omega}\delta(\theta)(\F\,\F^T\,\F-\F) : \A\diff x + \varepsilon\int_{\Omega}\nabla_x \F \because \nabla_x \A\diff x = 0
\end{aligned}
\end{equation}
for almost all time $t\in (0,T)$ and for all $\A\in W^{1,2}(\Omega; \R^{2\times 2})$. Moreover, by classical arguments, we have that $\bv \in \mathcal{C}([0,T];L^2_{0,\DIV})$ and $\F\in \mathcal{C}([0,T]; L^2(\Omega; \R^{2\times 2}))$ and also that $\bv(0)=\bv_0$ and $\F(0)=\F_0$. Furthermore, thanks to the Fatou lemma and the uniform bound \eqref{final_bounds_on_log}, we also have $\theta\in L^{\infty}_t L^1_x$.  

To deal with the problematic terms, we employ the energy methods. We multiply the $j$-th equation in \eqref{eq:galerkin_1stconv_velocity} by $\alpha_j^n$ and sum the result over $j=1,\ldots, n$ and integrate over $t\in (0,T)$ to conclude (using the fact that $\DIV_x \bv_n=0$)
\begin{align}\label{2ndconv_tested_by_u_velocity_n}
    \int_{\Omega}\frac{|\bv_n(T)|^2}{2}\diff x
    + \int_0^T\int_{\Omega}2\nu(\theta_n)|\D \bv_n|^2 + 2\F_n\,\F_n^T : \nabla_x \bv_n\diff x\diff t = \int_{\Omega}\frac{|\bv_{0n}|}{2}\diff x.
\end{align}
Similarly, we set $\bomega:=\bv$ in  \eqref{eq:galerkin_2ndconv_velocity}  and integrate the equation over $(0, T)$ and obtain
\begin{align}\label{2ndconv_tested_by_u_velocity}
   \int_{\Omega}\frac{|\bv(T)|^2}{2}\diff x
    + \int_0^T\int_{\Omega}2\nu(\theta)|\D \bv|^2 + 2\F\,\F^T : \nabla_x \bv\diff x\diff t = \int_{\Omega}\frac{|\bv_{0}|}{2}\diff x.
\end{align}
Similarly, we multiply the $j$-th equation in \eqref{eq:galerkin_elastic} by $2\beta_j^{n}$ and sum with respect to $j=1,\ldots,n$ to get after integration over $t\in (0,T)$ 
\begin{equation}
\begin{aligned}\label{eq:galerkin_ener1}
    &\int_{\Omega} |\F_n(T)|^2 \diff x + \int_0^T \int_{\Omega} \delta(\theta_n)|\F_n\,\F_n^T|^2 -  2 \nabla_x \bv_n : \F_n\, \F_n^T + 2\varepsilon |\nabla_x \F_n|^2 \diff x \diff t\\
    &\qquad  =  
   \int_0^T \int_{\Omega}\delta(\theta_n) |\F_n|^2 \diff x \diff t +  \int_{\Omega} |\F_{0n}(T)|^2 \diff x 
\end{aligned}
\end{equation}
and by setting $\A:=\F$ in \eqref{eq:galerkin_2ndconv_elastic}, we have 
\begin{equation}
\begin{aligned}\label{eq:ener2}
    &\int_{\Omega} |\F(T)|^2 \diff x + \int_0^T \int_{\Omega} \delta(\theta)|\F\,\F^T|^2 -  2 \nabla_x \bv : \F\, \F^T +2\varepsilon |\nabla_x \F_n|^2 \diff x \diff t\\
    &\qquad  =  
   \int_0^T \int_{\Omega}\delta(\theta) |\F|^2 \diff x \diff t +  \int_{\Omega} |\F_{0}(T)|^2 \diff x.
\end{aligned}
\end{equation}
Summing \eqref{2ndconv_tested_by_u_velocity_n} and \eqref{eq:galerkin_ener1}, we have
\begin{equation}
\begin{aligned}\label{eq:galerkin_ener12}
    &\int_0^T \int_{\Omega} 2\nu(\theta_n)|\D \bv_n|^2+ \delta(\theta_n)|\F_n\,\F_n^T|^2+2\varepsilon |\nabla_x \F_n|^2 \diff x \diff t  \\
    &\qquad =  
   \int_0^T \int_{\Omega}\delta(\theta_n) |\F_n|^2 \diff x \diff t +  \int_{\Omega}\frac{|\bv_{0n}|}{2}+ |\F_{0n}|^2-\frac{|\bv_n(T)|^2}{2}- |\F_n(T)|^2 \diff x.
\end{aligned}
\end{equation}
Similarly, summing \eqref{2ndconv_tested_by_u_velocity} and \eqref{eq:ener2} leads to 
\begin{equation}
\begin{aligned}\label{eq:galerkin_ener123}
    &\int_0^T \int_{\Omega} 2\nu(\theta)|\D \bv|^2+ \delta(\theta)|\F\,\F^T|^2  +2\varepsilon |\nabla_x \F|^2 \diff x \diff t \\
    &\qquad =  
   \int_0^T \int_{\Omega}\delta(\theta) |\F|^2 \diff x \diff t +  \int_{\Omega}\frac{|\bv_{0}|}{2}+ |\F_{0}|^2-\frac{|\bv(T)|^2}{2}- |\F(T)|^2 \diff x.
\end{aligned}
\end{equation}
Using \eqref{strong_conv_elastic_n}, \eqref{thn_s}, the assumption \eqref{bounds_delta}, the convergence of initial conditions in \eqref{conv_init} and also the weak lower semicontinuity, we get 
%
%
%, and comparing the results, we have
%
%In comparison, we test \eqref{eq:galerkin_1stconv_velocity} by $u_n$ (i.e. multiply the equation by $\alpha_j$ and sum over $j$), and integrate over $(0, t)$ to get
%\begin{align*}
%    \int_{\Omega}|u_{n}(t)|^2\diff x  + \int_{\Omega}\nu(\theta_{n})|D u_{n}|^2\diff x + \int_{\Omega}\F_{n}\,\F_{n}^T : \nabla_xu_n\diff x = \int_{\Omega}|u_{n}(0)|^2\diff x.
%\end{align*}
%We may perform an analogous argument, testing \eqref{eq:galerkin_1stconv_elastic} and \eqref{eq:galerkin_2ndconv_elastic} by $\F_n$ and $\F$ respectively. After adding the corresponding equations together, we may use \eqref{strong_conv_elastic_n} and Lemma \ref{L2convlemma}, to deduce
\begin{equation}
\begin{split}
     \limsup_{n\to +\infty}&\int_0^T\int_{\Omega}2\nu(\theta_{n})|\D \bv_{n}|^2\diff x+\delta(\theta_n)|\F_n\,\F_n^T|^2+2\varepsilon |\nabla_x \F_n|^2\diff x\diff t \\
     &\quad \leq  \int_0^T\int_{\Omega}2\nu(\theta)|\D \bv|^2+\delta(\theta) |\F\,\F^T|^2+2\varepsilon |\nabla_x \F|^2\diff x\diff t.
\end{split}\label{split}
\end{equation}
Finally, using the assumptions \eqref{bouds_nu}--\eqref{bounds_delta}, the convergence results \eqref{thn_s}--\eqref{strong_conv_elastic_n} and the inequality \eqref{split}, we get
\begin{equation}\label{tutu}
\begin{aligned}
 C_1&\limsup_{n\to \infty}\int_0^T \int_{\Omega} |\D\bv_n -\D\bv|^2 + |\F_n \, \F_n^T - \F\, \F^T|^2 \diff x \diff t\\
 &\le \limsup_{n\to \infty}\int_0^T \int_{\Omega} 2\nu(\theta_n)|\D\bv_n -\D\bv|^2 + \delta(\theta_n)|\F_n \, \F_n^T - \F\, \F^T|^2 +2\varepsilon |\nabla_x \F_n - \nabla_x \F|^2\diff x \diff t\\
&\le \limsup_{n\to \infty}\int_0^T \int_{\Omega} 2\nu(\theta_n)|\D\bv_n|^2  + \delta(\theta_n)|\F_n \, \F_n^T|^2 +2\varepsilon |\nabla_x \F_n|^2\diff x \diff t\\
&\qquad +\lim_{n\to \infty}\int_0^T \int_{\Omega} 2\nu(\theta_n)(|\D\bv|^2-2\D\bv_n : \D\bv)  + \delta(\theta_n)(|\F\, \F^T|^2-2 \F_n \, \F_n^T: \F\, \F^T) \diff x \diff t\\
&\qquad + \lim_{n\to \infty} \int_0^T \int_{\Omega} 2\varepsilon (|\nabla_x \F|^2- 2\nabla_x \F_n \because \nabla_x \F) \diff x \diff t\\
&\le \int_0^T \int_{\Omega} 2\nu(\theta)|\D\bv|^2  + \delta(\theta)|\F \, \F^T|^2 +2\varepsilon |\nabla_x \F|^2\diff x \diff t\\
&\qquad -\int_0^T \int_{\Omega} 2\nu(\theta)|\D\bv|^2  + \delta(\theta)|\F\, \F^T|^2+2\varepsilon |\nabla_x \F|^2\diff x \diff t=0
\end{aligned}
\end{equation}
Consequently, due to Korn inequality and the Sobolev embedding we deduced that 
\begin{align}
\bv_n &\to \bv &&\textrm{ strongly in } L^2_t W^{1,2}_x,\label{JW1}\\
\F_n &\to \F &&\textrm{ strongly in } L^2_t W^{1,2}_x, \label{JW2}\\
\F_n &\to \F &&\textrm{ strongly in } L^4_{t,x}. \label{JW3}
\end{align}
%
%
%
%
%which by weak lower semi-continuity implies
%\begin{multline*}
%    \lim_{n\to +\infty}\int_0^t\int_{\Omega}\nu(\theta_{n})|D u_{n}|^2\diff x\diff\tau + \int_0^t\int_\Omega |\F_n\,\F_n^T|^2\diff x\diff\tau\\
%    = \int_0^t\int_{\Omega}\nu(\theta)|D u|^2\diff x\diff\tau + \int_0^t\int_\Omega |\F\,\F^T|^2\diff x\diff\tau.
%\end{multline*}
%This is enough to conclude
%\begin{align*}
%    \int_0^T\int_\Omega |\sqrt{\nu(\theta_n)}Du_n - \sqrt{\nu(\theta)}Du|^2\diff x\diff t + \int_0^T\int_\Omega|\F_n\,\F_n^T - \F\,\F^T|^2\diff x\diff t = 0.
%\end{align*}
Having the above convergence results in hands, combined with \eqref{weak_conv_temp_2_lambda}--\eqref{strong_conv_elastic_n}, we deduce from \eqref{eq:galerkin_1stconv_temperature} the following 
\begin{equation}
\begin{aligned}\label{eq:galerkin_2ndconv_temp}
    &\int_0^T\int_\Omega - \theta\,\p_t\phi - \theta\,\bv \cdot \nabla_x\phi + \kappa(\theta)\nabla_x\theta\cdot\nabla_x\phi\diff x\diff t\\
    & \quad =\int_0^T\int_{\Omega}2\nu(\theta)|\D\bv|^2\phi+\delta(\theta)|\F\,\F^T - \mathbb{I}|^2\phi\diff x\diff t + \int_\Omega \theta^r_0(x)\phi(0, x)\diff x,
\end{aligned}
\end{equation}
for all $\phi\in \mathcal{C}^1_c([0, T)\times \overline{\Omega})$. Moreover, it follows from the above identity that 
\begin{equation}
\partial_t \theta \in L^1((0,T); W^{-2,2}(\Omega)).\label{exec}
\end{equation}

\subsection{Convergence with \texorpdfstring{$\varepsilon\to 0_+$}{e} and \texorpdfstring{$r\to 0_+$}{r}.} We denote by $\{\bv_{\varepsilon}, \F_{\varepsilon}, \theta_{\varepsilon}\}_{\varepsilon>0}$ the solution constructed in the previous section and set  $r=\varepsilon\to 0_+$ to finish the proof of Theorem~\ref{main_theorem_case_constant_g}. We recall the uniform $\varepsilon$-independent estimates  \eqref{galerkin_bounds_on_velocity_and_elastic_n}, \eqref{final_bounds_time_elastic_velocity}, \eqref{temp_linfty_l1_bounds}, \eqref{temp_fractional_gradient_bound} and \eqref{final_bounds_time_temp}, which remains valid also here due to the Fatou lemma and weak-lower semicontinuity. Therefore, we may again extract a subsequence that we do not relabel such that
\begin{align*}
     %\theta_{\varepsilon} &\rightharpoonup \theta_{}\quad &&\text{ weakly in }L^{2 - \lambda}_{t,x},\\
     \theta_{\varepsilon} &\rightharpoonup \theta\quad &&\text{ weakly in }L^{p}_t W^{1, p}_x \textrm{ for all } p\in \left[1, \frac43 \right),\\
     \theta_{\varepsilon} &\rightarrow \theta \quad &&\text{ strongly in }L^{s}_t L^{q}_x \textrm{ for all } q\in [1,\infty) \textrm{ and } s\in [1,q'),\\
     \ln \theta_{\varepsilon} &\rightarrow \ln \theta \quad &&\text{ weakly in }L^{2}_t W^{1,2}_x,\\
    \bv_{\varepsilon} &\rightharpoonup \bv \quad &&\text{ weakly* in }L^{\infty}_t L^2_x,\\
    \bv_{\varepsilon} &\rightharpoonup \bv\quad &&\text{ weakly in }L^{2}_t W^{1, 2}_{0,x},\\
    \p_t \bv_{\varepsilon} &\rightharpoonup \p_t \bv\quad &&\text{ weakly in }L^{2}((0, T); W^{-1, 2}_{0, \DIV}(\Omega))\\
    \bv_{\varepsilon} &\rightarrow \bv\quad &&\text{ strongly in }L^q_{t,x}, \text{ for }q\in[1, 4),\\
    \F_{\varepsilon} &\rightharpoonup \F\quad &&\text{ weakly* in }L^{\infty}_t L^2_x,\\
    \F_{\varepsilon} &\rightharpoonup \F \quad &&\text{ weakly in }L^{4}_{t,x},\\
    \p_t \F_{\varepsilon} &\rightharpoonup \p_t \F \quad &&\text{ weakly in }L^{\frac{4}{3}}((0, T); W^{-1, 2}(\Omega;\R^{2\times 2})),\\
    \varepsilon\nabla_x\F_{\varepsilon} &\rightarrow \OO\quad &&\text{ strongly in }L^{2}_{t,x}.
\end{align*}
Compared to the previous section, we did not get the strong convergence of $\F_{\varepsilon}\to \F$ directly from the uniform estimates. However, we can now use the convergence scheme from the proof of Theorem~\ref{main_theorem_case_general_g} (in fact, here it is easier since the function $\delta$ is bounded) and deduce
\begin{align*}
    \F_{\varepsilon} \rightarrow \F_{}\quad\text{ strongly in }L^2_{t,x}.
\end{align*}
Then we can deduce from \eqref{eq:galerkin_2ndconv_velocity} and \eqref{eq:galerkin_2ndconv_elastic} that
\begin{align}\label{eq:galerkin_2ndconv_velocity_zero}
    \langle\p_t \bv,\bomega\rangle - \int_{\Omega}\bv\otimes \bv :\nabla_x\bomega\diff x
    + \int_{\Omega}2\nu(\theta)\D \bv : \nabla_x\bomega\diff x + \int_{\Omega}2\F\,\F^T : \nabla_x\bomega\diff x = 0
\end{align}
for almost all time $t\in (0,T)$ and for all $\bomega\in W^{1,2}_{0, \DIV}(\Omega)$ and
\begin{equation}
\begin{aligned}\label{eq:galerkin_2ndconv_elastic_zero}
    &\langle\p_t \F, \A\rangle - \int_{\Omega}\F\otimes \bv \because \nabla_x \A+(\nabla_x \bv\, \F): \A-\frac{1}{2}\delta(\theta)(\F\,\F^T\,\F-\F) : \A\diff x  = 0
\end{aligned}
\end{equation}
for almost all time $t\in (0,T)$ and for all $\A\in W^{1,2}_x(\Omega; \mathbb{R}^{2\times 2})$. Moreover, we also have the final formulations in the main theorem \eqref{weak_formulation_u_constant_g} and \eqref{weak_formulation_F_constant_g}.

Next, we need to get the compactness of the terms appearing on the right-hand side of \eqref{eq:galerkin_2ndconv_temp}. We want to repeat the scheme in Section~\ref{sect12}, namely the computation between \eqref{2ndconv_tested_by_u_velocity_n}--\eqref{split}. To do so, we need to justify the setting $\A:=\F$ in \eqref{eq:galerkin_2ndconv_elastic_zero}. This can be however justify by the renormalisation technique developed in  \cite{diperna1989ordinary}, see also \cite{bulicek2022onplanar} in the context of viscoelastic fluids. Thus, similarly as in \eqref{split}, we deduce
\begin{equation}
\begin{split}
     \limsup_{\varepsilon\to 0_+}&\int_0^T\int_{\Omega}2\nu(\theta_{\varepsilon})|\D \bv_{\varepsilon}|^2\diff x+\delta(\theta_{\varepsilon})|\F_{\varepsilon}\,\F_{\varepsilon}^T|^2+2\varepsilon |\nabla_x \F_{\varepsilon}|^2\diff x\diff t \\
     &\quad \leq  \int_0^T\int_{\Omega}2\nu(\theta)|\D \bv|^2+\delta(\theta) |\F\,\F^T|^2\diff x\diff t.
\end{split}\label{split2}
\end{equation}
Having this, we may repeat \eqref{tutu} and conclude
\begin{align}
\bv_n &\to \bv &&\textrm{ strongly in } L^2_t W^{1,2}_x,\label{JW12}\\
\F_n &\to \F &&\textrm{ strongly in } L^4_{t,x}. \label{JW34}
\end{align} 
And finally, we can let $\varepsilon \to 0_+$ in \eqref{eq:galerkin_2ndconv_temp} and deduce  \eqref{weak_formulation_theta_constant_g}. Furthermore, since $\ln \theta \in L^2_t W^{1,2}_x$, we see that $\theta>0$ almost everywhere in $(0,T)\times \Omega$. The proof is complete.

%We carried out the discussion about the procedure under \eqref{strong_convergence_F_n}, therefore we skip the details here. In fact, this is even easier as $\delta(\theta_\varepsilon)\F_\varepsilon\,\F^T_\varepsilon\,\F_\varepsilon\in L^{\frac{4}{3}}_{t,x}$. Combining all of the convergences we arrive at
%\begin{align*}
%    \langle\p_t u_{},\omega\rangle - \int_{\Omega}u_{}\otimes u_{}:\nabla_x\omega\diff x
%    + \int_{\Omega}\nu(\theta_{})D u_{} : \nabla_x\omega\diff x + \int_{\Omega}\F_{}\,\F_{}^T : \nabla_x\omega\diff x = 0,
%\end{align*}
%for any $\omega\in W^{1,2}_0(\Omega)$,
%\begin{multline*}
%    \langle\p_t \F_{}, A\rangle - \int_{\Omega}\F_{}\otimes u_{} \because \nabla_x A\diff x - \int_{\Omega}(\nabla_x u_{}\, \F_{}):A\diff x  \\
%    +\frac{1}{2}\int_{\Omega}\delta(\theta_{})\F_{}\,\F_{}^T\,\F_{} : A\diff x - \frac{1}{2}\int_{\Omega}\delta(\theta_{})\F_{} : A \diff x  = 0,
%\end{multline*}
%for any $A\in W^{1,2}(\Omega; \R^{2\times 2})$, and
%\begin{multline*}
%    -\int_0^T\int_\Omega\theta\,\p_t\phi\diff x\diff t + \int_0^T\int_{\Omega}\theta\,u\cdot \nabla_x\phi\diff x\diff t + \int_0^T\int_{\Omega}\kappa(\theta)\nabla_x\theta\cdot\nabla_x\phi\diff x\diff t\\
%    -\int_0^T\int_{\Omega}\nu(\theta)|Du|^2\phi\diff x\diff t - \int_0^T\int_{\Omega}|\F\,\F^T - \mathbb{I}|^2\phi\diff x\diff t = \int_\Omega \theta_0(x)\phi(0, x)\diff x,
%\end{multline*}
%for any $\phi\in C^1_c([0, T)\times \overline{\Omega})$. Which ends the proof.

\appendix

\section{Auxiliary propositions} \label{AP}
We recall here several useful tools. The first two results are about the compactness of some weakly converging sequences.  
\begin{lem}\label{aubin-lions}\textup{\textbf{(Generalized Aubin--Lions lemma, \cite[Lemma 7.7]{MR3014456})}}
Denote by
$$
W^{1, p, q}(I; X_1, X_2) := \left\{u\in L^p(I; X_1); \frac{du}{dt}\in L^q(I; X_2)\right\}.
$$
Then if $X_1$ is a separable, reflexive Banach space, $X_2$ is a Banach space and $X_3$ is a metrizable locally convex Hausdorff space, $X_1$ embeds compactly into $X_2$, $X_2$ embeds continuously into $X_3$, $1 < p <\infty$ and $1\leq q\leq \infty$, we have
$$
W^{1, p, q}(I; X_1, X_3) \text{ embeds compactly into }L^p(I; X_2).
$$
In particular any bounded sequence in $W^{1, p, q}(I; X_1, X_3)$ has a convergent subsequence in $L^p(I; X_2)$.
\end{lem}

\begin{lem}\label{div-curl}\textup{\textbf{(The div-curl lemma, \cite{conti2011thedivcurl})}}
    Let $\Omega\subset\R^n$ be an open and bounded domain with a Lipschitz boundary, and let $p,q\in (1+\infty)$ with $\frac{1}{p} + \frac{1}{q} = 1$. Suppose $\bu_k\in L^p(\Omega; \R^n)$, $\bv_k\in L^q(\Omega; \R^n)$ are sequences such that
    $$
    \bu_k \rightharpoonup \bu\text{ weakly in }L^p(\Omega; \R^n),\quad \bv_k\rightharpoonup \bv\text{ weakly in }L^q(\Omega; \R^n),
    $$
    and
    $$
    \bu_k\cdot \bv_k \text{ is equiintegrable}.
    $$
    Finally, assume that
    $$
    \DIV \bu_k \rightarrow \DIV \bu \text{ strongly in }(W^{1,\infty}_0(\Omega))^*, \quad \mathrm{curl}\,\bv_k \rightarrow \mathrm{curl}
    \, \bv \text{ strongly in }(W^{1,\infty}_0(\Omega; M^{n\times n}))^*.
    $$
    Then,
    $$
    \bu_k\cdot \bv_k\rightharpoonup \bu\cdot \bv\text{ weakly in }L^1(\Omega).
    $$

\end{lem}
The next very classical result is about the integration by parts formula in Bochner function spaces. 
\begin{lem}\label{lem:integration_of_lipschitz}
    Let $1 <p, q < +\infty$. Suppose $\psi:\R\rightarrow\R$ is a Lipschitz function. For $r\in\R$ we define
    $$
            \Psi(x) = \int_r^x \psi(s)\diff s,\quad x\in\R.
    $$
    Then, for any $u\in W^{1, p, p}(I; W^{1,q}(\Omega)) := \left\{u\in L^p(I; W^{1,q}(\Omega)); \frac{du}{dt}\in (L^p(I; W^{1,q}(\Omega)))^*\right\}$ it holds
    $$
    \int_{t_1}^{t_2}\langle\p_t u, \psi(u)\rangle\diff t = \int_\Omega \Psi(u(t_2))\diff x - \int_{\Omega}\Psi((u(t_1))\diff x,\quad t_1, t_2\in I.
    $$
\end{lem}
\iffalse
The last very classical result also deals with the properties of Bochner spaces. 
\begin{lem}\label{L2convlemma}
Let $1 \leq p < \infty$ and $\{u_n\}_{n \in \N}$ be a sequence such that $u_n \to u$ in $L^p_t L^p_x$. Then, there exists a subsequence $\{u_{n_k}\}_{k \in \N}$, such that
$$
\mbox{for a.e. } t \in (0,T) \qquad u_{n_k}(t,x) \to u(t,x) \mbox{ in } L^p_x
$$
Moreover, if $\{u_{n_k}\}_{k \in \N}$ is bounded in $L^{\infty}_t L^2_x$ we have for a.e. $t \in (0,T)$
$$
\int_{\Omega} |u(t,x)|^2 \diff x \leq \liminf_{k \to \infty} \int_{\Omega} |u_{n_k}(t,x)|^2 \diff x
$$
\end{lem}
\fi

\section*{Data Availability} Data sharing is not applicable to this article as no data sets were generated or analysed during the current study.

\section*{Declarations - Conflict of interest} The authors do not have a conflict of interest to declare that are relevant to the content of this article.

\bibliographystyle{abbrv}
\bibliography{viscoelastic}

\begin{thebibliography}{10}

\bibitem{BaMu89}
J.~M. Ball and F.~Murat.
\newblock Remarks on {C}hacon's biting lemma.
\newblock {\em Proc. Amer. Math. Soc.}, 107(3):655--663, 1989.

\bibitem{bathory2021largedata}
M.~Bathory, M.~Bul\'{\i}\v{c}ek, and J.~M\'{a}lek.
\newblock Large data existence theory for three-dimensional unsteady flows of rate-type viscoelastic fluids with stress diffusion.
\newblock {\em Adv. Nonlinear Anal.}, 10(1):501--521, 2021.

\bibitem{BaBuMa24}
M.~Bathory, M.~Bul\'{\i}\v{c}ek, and J.~M\'alek.
\newblock Coupling the {N}avier-{S}tokes-{F}ourier equations with the {J}ohnson-{S}egalman stress-diffusive viscoelastic model: global-in-time and large-data analysis.
\newblock {\em Math. Models Methods Appl. Sci.}, 34(3):417--476, 2024.

\bibitem{BuFeMa09}
M.~Bul\'{\i}\v{c}ek, E.~Feireisl, and J.~M\'alek.
\newblock A {N}avier-{S}tokes-{F}ourier system for incompressible fluids with temperature dependent material coefficients.
\newblock {\em Nonlinear Anal. Real World Appl.}, 10(2):992--1015, 2009.

\bibitem{bulicek2019onaclass}
M.~Bul\'{\i}\v{c}ek, E.~Feireisl, and J.~M\'{a}lek.
\newblock On a class of compressible viscoelastic rate-type fluids with stress-diffusion.
\newblock {\em Nonlinearity}, 32(12):4665--4681, 2019.

\bibitem{BuJuPoZa22}
M.~Bul\'{\i}\v{c}ek, A.~J\"ungel, M.~Pokorn\'y, and N.~Zamponi.
\newblock Existence analysis of a stationary compressible fluid model for heat-conducting and chemically reacting mixtures.
\newblock {\em J. Math. Phys.}, 63(5):Paper No. 051501, 48, 2022.

\bibitem{bulicek2022onplanar}
M.~Bul\'{\i}\v{c}ek, T.~Los, Y.~Lu, and J.~M\'{a}lek.
\newblock On planar flows of viscoelastic fluids of {G}iesekus type.
\newblock {\em Nonlinearity}, 35(12):6557--6604, 2022.

\bibitem{bulicek2024threeD}
M.~Bul\'{\i}\v{c}ek, T.~Los, and J.~M\'{a}lek.
\newblock On three-dimensional flows of viscoelastic fluids of {G}iesekus type.
\newblock {\em Nonlinearity}, 38(1):015004 (42pp), 2025.

\bibitem{BuMaPrSu21}
M.~Bul\'{\i}\v{c}ek, J.~M\'alek, V.~Pr\r{u}\v{s}a, and E.~S\"{u}li.
\newblock On incompressible heat-conducting viscoelastic rate-type fluids with stress-diffusion and purely spherical elastic response.
\newblock {\em SIAM J. Math. Anal.}, 53(4):3985--4030, 2021.

\bibitem{BuMaRa09}
M.~Bul\'{\i}\v{c}ek, J.~M\'alek, and K.~R. Rajagopal.
\newblock Mathematical analysis of unsteady flows of fluids with pressure, shear-rate, and temperature dependent material moduli that slip at solid boundaries.
\newblock {\em SIAM J. Math. Anal.}, 41(2):665--707, 2009.

\bibitem{burgers1939mechanical}
J.~Burgers.
\newblock Mechanical considerations—model systems—phenomenological theories of relaxations and viscosity.
\newblock {\em First Report on Viscosity and Plasticity (New York: Nordemann Publishing Company)}, pages 5--67, 1939.

\bibitem{Co00}
L.~Consiglieri.
\newblock Weak solutions for a class of non-{N}ewtonian fluids with energy transfer.
\newblock {\em J. Math. Fluid Mech.}, 2(3):267--293, 2000.

\bibitem{conti2011thedivcurl}
S.~Conti, G.~Dolzmann, and S.~M\"{u}ller.
\newblock The div-curl lemma for sequences whose divergence and curl are compact in {$W^{-1,1}$}.
\newblock {\em C. R. Math. Acad. Sci. Paris}, 349(3-4):175--178, 2011.

\bibitem{diperna1989ordinary}
R.~J. DiPerna and P.-L. Lions.
\newblock Ordinary differential equations, transport theory and {S}obolev spaces.
\newblock {\em Invent. Math.}, 98(3):511--547, 1989.

\bibitem{DEO99}
M.~Dressler, B.~J. Edwards, and H.~C. \"{O}ttinger.
\newblock Macroscopic thermodynamics of flowing polymeric liquids.
\newblock {\em Rheologica Acta}, 38(2):117--136, 1999.

\bibitem{Fe04}
E.~Feireisl.
\newblock {\em Dynamics of viscous compressible fluids}, volume~26 of {\em Oxford Lecture Series in Mathematics and its Applications}.
\newblock Oxford University Press, Oxford, 2004.

\bibitem{FeNo17}
E.~Feireisl and A.~Novotn\'y.
\newblock {\em Singular limits in thermodynamics of viscous fluids}.
\newblock Advances in Mathematical Fluid Mechanics. Birkh\"auser/Springer, Cham, second edition, 2017.

\bibitem{FeNo22}
E.~Feireisl and A.~Novotn\'y.
\newblock {\em Mathematics of open fluid systems}.
\newblock Ne\v cas Center Series. Birkh\"auser/Springer, Cham, [2022] \copyright 2022.

\bibitem{giesekus1982asimple}
H.~Giesekus.
\newblock A simple constitutive equation for polymer fluids based on the concept of deformation-dependent tensorial mobility.
\newblock {\em Journal of Non-Newtonian Fluid Mechanics}, 11(1):69--109, 1982.

\bibitem{Hron17}
J.~Hron, V.~Milo\v{s}, V.~Pr\r{u}\v{s}a, O.~Sou\v{c}ek, and K.~T\r{u}ma.
\newblock On thermodynamics of incompressible viscoelastic rate type fluids with temperature dependent material coefficients.
\newblock {\em International Journal of Non-Linear Mechanics}, 95:193--208, 2017.

\bibitem{HuLe}
D.~Hu and T.~Leli\`evre.
\newblock New entropy estimates for {O}ldroyd-{B} and related models.
\newblock {\em Commun. Math. Sci.}, 5(4):909--916, 2007.

\bibitem{johnson1977amodel}
M.~Johnson and D.~Segalman.
\newblock A model for viscoelastic fluid behavior which allows non-affine deformation.
\newblock {\em Journal of Non-Newtonian Fluid Mechanics}, 2(3):255--270, 1977.

\bibitem{lions2000global}
P.~L. Lions and N.~Masmoudi.
\newblock Global solutions for some {O}ldroyd models of non-{N}ewtonian flows.
\newblock {\em Chinese Ann. Math. Ser. B}, 21(2):131--146, 2000.

\bibitem{masmoudi2011global}
N.~Masmoudi.
\newblock Global existence of weak solutions to macroscopic models of polymeric flows.
\newblock {\em Journal de Mathématiques Pures et Appliquées}, 96(5):502--520, 2011.

\bibitem{oldroyd1950onthe}
J.~G. Oldroyd.
\newblock On the formulation of rheological equations of state.
\newblock {\em Proc. Roy. Soc. London Ser. A}, 200:523--541, 1950.

\bibitem{Raj00}
K.~Rajagopal and A.~Srinivasa.
\newblock A thermodynamic frame work for rate type fluid models.
\newblock {\em Journal of Non-Newtonian Fluid Mechanics}, 88(3):207--227, 2000.

\bibitem{MR3014456}
T.~Roub\'{\i}\v{c}ek.
\newblock {\em Nonlinear partial differential equations with applications}, volume 153 of {\em International Series of Numerical Mathematics}.
\newblock Birkh\"{a}user/Springer Basel AG, Basel, second edition, 2013.

\end{thebibliography}
\end{document}